\documentclass{amsart}
\usepackage{cO}
\numberwithin{thm}{subsection}

\author{Loren Spice}
\address{Texas Christian University \\ Fort Worth, TX \ \ 76129}
\email{l.spice@tcu.edu}
\dedicatory{To the memory of Paul J.~Sally, Jr.}
\thanks{The author was partially supported by Simons Foundation Collaboration Grant 246066.}

\title{Explicit asymptotic expansions for tame supercuspidal characters}

\makeindex

\begin{document}
\maketitle

\begin{abstract}
We combine the ideas of a Harish-Chandra--Howe local character expansion, which can be centred at an arbitrary semisimple element, and a Kim--Murnaghan asymptotic expansion, which so far has been considered only around the identity.  We show that, for most smooth, irreducible representations (those containing a good, minimal K-type), Kim--Murnaghan-type asymptotic expansions are valid on explicitly defined neighbourhoods of nearly arbitrary semisimple elements.  We then give an explicit, inductive recipe for computing the coefficients in an asymptotic expansion for a tame supercuspidal representation.  The only additional information needed in the inductive step is a fourth root of unity, which we expect to be useful in proving stability and endoscopic-transfer identities.
\end{abstract}

\tableofcontents

\section{Introduction}

\subsection{Motivation}

According to Harish-Chandra's Lefschetz principle, by analogy with the situation for real groups, the character tables of representations of \(p\)-adic groups are expected to carry significant information about harmonic analysis on those groups.  For example, in \cite{debacker-spice:stability}, DeBacker and the author used the character computations of Adler and the author \cite{adler-spice:explicit-chars} to show that a certain natural candidate for an L-packet, constructed and shown to satisfy many of the necessary properties by Reeder \cite{reeder:sc-pos-depth}, satisfied stability conditions.  Actually, a slight modification to the construction of Yu \cite{yu:supercuspidal}, on which the conjecture is built, was necessary; see \xcite{debacker-spice:stability}*{Definition \xref{defn:twisted-cusp-ind}} and, from a different perspective, our discussion of Weil representations in
\S\ref{sec:cuspidal}.  It is not clear that the appropriate modification could easily have been seen from ``first principles'', but it essentially popped out of the explicit character formul{\ae} of \cite{adler-spice:explicit-chars}.  These investigations have been generalised, and put in a broader setting related to the local Langlands correspondence, by Kaletha in \cites{kaletha:simple-wild,kaletha:epipelagic,kaletha:regular-sc}, the last of which seems to represent the current state of the art in extracting such information from character formul{\ae}.

In \cite{kaletha:regular-sc}*{Corollary 4.10.1}, Kaletha shows that the character formula \cite{debacker-spice:stability}*{Theorem \xref{thm:ratl}}, which is stated by DeBacker and the author only for supercuspidal representations attached to data satisfying a compactness condition, is valid on regular, topologically semisimple elements for the so called regular supercuspidal characters, even if they do not satisfy the compactness condition.  In \cite{kaletha:regular-sc}*{Corollary 4.7.2}, he also re-interprets the roots of unity in \cite{debacker-spice:stability}*{\S\xref{sec:root}} in such a way that they make sense on the dual-group side, even without the fine structure theory of \(p\)-adic groups.  This allows him to construct L-packets \cite{kaletha:regular-sc}*{Proposition 5.2.4}, and to prove stability \cite{kaletha:regular-sc}*{Theorem 6.3.2} and endoscopic-transfer identities \cite{kaletha:regular-sc}*{Theorem 6.3.4} as a consequence of the re-interpreted character formul{\ae} (in particular, globally for \emph{toral} supercuspidal characters, and on the regular, topologically semisimple set for all regular supercuspidal characters).  In order to extend the range of validity of these identities, some generalisation of the character formula is needed.

One immediate obstruction is the fact that the compactness assumption in \xcite{adler-spice:explicit-chars}*{Theorem \xref{thm:char-tau|pi-1}} is not just an artifact of the proof; it is needed even to state the result, by guaranteeing the finite-dimensionality of some representations and so allowing us to view their characters as, not merely densely defined \via the intricate Harish-Chandra machinery of representing functions (\cite{hc:harmonic}*{\S6, p.~60, Theorem 12} and \cite{hc:submersion}*{\S4, p.~99, corollary to Theorem 2}),
but actually \emph{globally} defined objects.  Indeed, \xcite{adler-spice:explicit-chars}*{Theorem \xref{thm:full-char}} is stated in terms of the values \(\Theta_{\pi'_0}(\gamma'_0)\) of a character \(\pi'_0\) at a possibly \emph{singular}, topologically semisimple (modulo centre) element \(\gamma'_0\), and it is not clear how to assign sensibly a numerical value to this symbol in all cases if \(\pi'_0\) is infinite-dimensional, so that the operator \(\pi'_0(\gamma'_0)\) is not trace class.

The work of Kim--Murnaghan \cites{jkim-murnaghan:charexp,jkim-murnaghan:gamma-asymptotic} on asymptotic expansions, similar to the local character expansion but phrased in terms of non-nilpotent orbits, suggests one way forward.  Indeed, \xcite{adler-spice:explicit-chars}*{Corollary \xref{cor:char-tau|pi-1-germ}} shows that the formul{\ae} of \cite{adler-spice:explicit-chars} specialise near the identity to the simplest case of such an asymptotic expansion (in terms of a single, semisimple orbital integral).  Unfortunately, this can only give information near the identity.  Here we turn to the work of Adler--Korman \cite{adler-korman:loc-char-exp}, which follows DeBacker \cite{debacker:homogeneity} and Kim--Murnaghan in proving quantitative results for asymptotic expansions centred at non-identity elements.  The Adler--Korman results concern the local character expansion, in terms of nilpotent orbital integrals.  We view these two programs as suggesting a middle ground:  a Kim--Murnaghan-type asympotic expansion, but centred around arbitrary semisimple points.

This gives an idea for the \emph{shape} of a general character formula, but there is one important piece of information missing:  namely, the coefficients in the asymptotic expansions.  In \xcite{debacker-spice:stability}*{Theorem \xref{thm:ratl}}, DeBacker and the author re-write the character formul{\ae} of \cite{adler-spice:explicit-chars} in a form that seems more amenable to use in computational harmonic analysis.  In this paper we give an explicit inductive recipe, inspired by the re-written characters of \cite{debacker-spice:stability}, for computing the relevant coefficients when the representation being considered is tame supercuspidal (\ie, arises from Yu's construction \cite{yu:supercuspidal}).  In future joint work with DeBacker and Kaletha, we will explore the consequences of these results for the endoscopic-transfer identities of \cite{kaletha:regular-sc} for non-toral, regular supercuspidal representations.

\subsection{Structure of the paper}

One surprising aspect of this paper is that, although our goal is explicit character computations, we do not even mention representations except in \S\S\ref{sec:asymptotic} and \ref{sec:cuspidal}.  It turns out that much of the machinery that we need can be constructed in the setting of general invariant distributions.  We hope that this extra generality will be useful in future applications.

In \S\ref{sec:notn}, we lay out the standard notation that we will use.
In \S\ref{sec:concave}, we briefly recall the Bruhat--Tits theory of groups associated to concave functions \cite{bruhat-tits:reductive-groups-1}*{Proposition 6.4.9}, and use Yu's approach to their structure theory \cite{yu:supercuspidal}*{\S2} to do a few calculations.  In \S\ref{sec:depth-matrix}, we modify the theory to handle a class of compact, open subgroups whose definition involves reductive, algebraic subgroups that need not have full rank.  In previous work \cites{adler-spice:good-expansions,adler-spice:explicit-chars,debacker-spice:stability}, we have often needed to impose considerable tameness hypotheses in order to handle such groups.  Although we are not yet ready completely to discard such hypotheses, we can at least isolate exactly the parts that we need.  Namely, we have Hypotheses \ref{hyp:mexp}, \ref{hyp:depth}, and \ref{hyp:MP-ad}, which deal with the group itself; Hypotheses \ref{hyp:funny-centraliser}, \ref{hyp:gamma}, and \ref{hyp:gamma-central}, which concern the element \(\gamma\) about which we centre our asymptotic expansions; and Hypotheses \ref{hyp:Z*} and \ref{hyp:K-type}, which concern the K-types contained in the representations whose characters we want to compute, and are automatically satisfied for Yu's tame supercuspidal representations.  (There are also Hypotheses \ref{hyp:X*} and \ref{hyp:phi}, which need not be explicitly imposed for the main results, where they are automatically satisfied.)  All of these are known to hold in many cases; we discuss sufficient conditions for each as it is introduced.

We have also avoided as long as possible assuming that our group is connected (or that the automorphism of its identity component induced by the element \(\gamma\) of \S\ref{sec:depth-matrix} is inner).  Although we eventually do inherit this assumption from \cite{yu:supercuspidal} in \S\ref{sec:cuspidal}, we have laid enough of the groundwork by that point that we hope it can serve as a starting point for investigating asymptotic expansions related to twisted characters and base change.

Although our final result provides explicit \emph{computations} only for tame supercuspidal representations, it turns out that the asymptotic-expansion machinery built by DeBacker in \cite{debacker:homogeneity}, and later generalised by Kim--Murnaghan in \cites{jkim-murnaghan:charexp,jkim-murnaghan:gamma-asymptotic} and Adler--Korman in \cite{adler-korman:loc-char-exp}, is sufficiently general that it can handle our recentred \emph{existence} results for most smooth, irreducible representations.
In \S\ref{sec:nearly-good}, we recall the now-standard properties of ``good'' elements, slightly generalised to the non-connected setting.  Our goal is to use this machinery to describe the good minimal K-types, but, before we do so, it turns out to be convenient to take a detour in \S\ref{sec:H-perp} to do some abstract-algebraic computations that allow us to define a ``perpendicular group'', in some sense analogous to the groups \(J\) and \(J_+\) of \cite{yu:supercuspidal}*{\S9}, that will be useful in our later computation of Gauss sums (\S\ref{sec:Gauss}).  In \S\S\ref{sec:dual-blob}, \ref{sec:asymptotic}, we describe the K-types that we consider.  Our first main result, Theorem \ref{thm:asymptotic-exists}, shows that the character of a representation containing one of these K-types has Kim--Murnaghan-style asymptotic expansions, on explicitly described neighbourhoods, about (nearly) arbitrary semisimple elements, not just the identity.

In \S\ref{sec:quantitative}, we build to a quantitative version of the qualitative results of \S\ref{sec:qualitative}.  Inevitably (see \cite{adler-spice:explicit-chars}, and the historical discussion there), explicit calculations in \(p\)-adic harmonic analysis seem to involve certain fourth roots of unity known as Gauss sums.  In \cite{kaletha:regular-sc}*{\S\S4.6--4.7}, Kaletha shows that these fourth roots of unity are actually indirectly predicted by the local Langlands correspondence, since they occur in the definition of the transfer factor.  In \S\ref{sec:Gauss}, we follow Waldspurger \cite{waldspurger:loc-trace-form} in interpreting these roots of unity as Weil indices.  We expect that the recent work of Kottwitz \cite{kottwitz:weil}, which has already been used in \cite{kaletha:epipelagic}*{Theorem 4.10} and \cite{kaletha:regular-sc}*{Corollary 4.7.2}, will continue to be helpful in translating these fourth roots of unity into a form suited to stability and endoscopic-transfer calculations.

Our work in \S\ref{sec:Gauss} falls into two parts.  The easy part is to define a Weil index using the Lie algebra (Notation \ref{notn:Gauss}); the hard part is to show that this index actually arises in computations on the group.  We do this latter in Proposition \ref{prop:Gauss-to-Weil}.  This section is the analogue of \xcite{adler-spice:explicit-chars}*{\S\xref{ssec:gauss}}.  We have managed to avoid the ``centrality assumption'' \xcite{adler-spice:explicit-chars}*{Hypothesis \xref{hyp:X*-central}} there by working as much as possible directly on the group (which behaves well under tame base change), rather than with values of linear characters (which need not extend).  We hope that the occasionally hairy computations are justified by what we see as increased clarity of the underlying concepts.

The heart of this paper is \S\ref{sec:dist}.  As mentioned above, although our final goal is to compute characters, it turns out that many of the tools along the way---particularly vanishing results, which cut down on the support of a character, or on the domain over which an integral must be extended---work just as well for arbitrary invariant distributions.  In particular, we have isolated a key part of \xcite{adler-spice:explicit-chars}*{Proposition \xref{prop:step1-support}} as Lemma \ref{lem:r-to-s+}; and translated the crux of \xcite{adler-spice:explicit-chars}*{Proposition \xref{prop:step1-formula1}} to the setting of general invariant distributions, as Proposition \ref{prop:Gauss-appears}.  With these tools in hand, we can prove the main result of the section, Theorem \ref{thm:dist-G-to-G'}, which is a descent result that allows us to relate the coefficients in asymptotic expansions on a group \bG and a twisted Levi subgroup \(\bG'\).  We show again here the inspiration that we have drawn from the work of Kim--Murnaghan; our Theorem \ref{thm:dist-G-to-G'}  is very reminiscent of the descent arguments appearing in \cite{jkim-murnaghan:gamma-asymptotic}*{\S\S6.2, 7.2}.

Of course, some compatibility condition is necessary between the distributions on \(G\) and \(G'\).  We have phrased it in a way that, we believe, suggests a Hecke-algebra isomorphism crying out to be discovered.  We are not yet able to prove the existence of such an isomorphism, but the first main result of \S\ref{sec:cuspidal}, Theorem \ref{thm:isotypic-pi-to-pi'}, describes a crude but suitable substitute in the setting of Yu's construction of tame supercuspidal representations.  Combining this with Theorem \ref{thm:dist-G-to-G'} allows us to deduce the main result of the paper, Theorem \ref{thm:asymptotic-pi-to-pi'}, which gives an explicit, inductive description of the coefficients in the asymptotic expansion, centred around a (nearly) arbitrary semisimple point, of the character of a positive-depth, tame supercuspidal representation.

\subsection{Acknowledgements}

My advisor, Paul J.~Sally, Jr., instilled in me the idea of the importance of explicit character formul{\ae} in harmonic analysis, but passed away before I could share these results with him.  I like to think that he would have been pleased.  It is my honour to dedicate this paper to his memory.

I owe an enormous debt to Jeff Adler and Stephen DeBacker, who patiently endured endless streams of technical questions.  I also benefited from conversations with Tasho Kaletha, who kept the pressure on me finally to write up the results, and Jessica Fintzen and Cheng-Chiang Tsai.  Adler and Kaletha also made comments on an earlier draft of this paper.  It is a pleasure to thank all of them.

Most of this paper was written while I was on sabbatical leave at the University of Michigan.  I thank them for their hospitality, and for access to their incredible library; and the Simons Foundation for the funding that supported my travel.

\section{Notation and definitions}
\label{sec:notn}

\subsection{Representations and function spaces}

Most of the notation below is standard, but we point out two important points that might be unexpected.  First, the \emph{function} \(\chrc{K, f}\) depends on the Haar measure \(\upd g\) on \(G\), although the \emph{measure} \(\chrc{K, f}\upd g\) does not.  Second, we have followed \cite{yu:supercuspidal}*{\S17} in using the notation of \cite{bushnell-kutzko:gln-book}*{\S4.1}, which builds in a contragredient to the definition of the Hecke algebra \(\Hecke(G\sslash K, \rho)\).

If \(V\) is a vector space over any field \(F\), then we denote by \anonpair the pairing between the dual space \(V^* \ldef \Hom_F(V, F)\) and \(V\), so that, for \(v^* \in V^*\) and \(v \in V\), we write \(\pair{v^*}v\) in place of \(v^*(v)\).

If \(\rho\) is a representation of a group \(G\) on the \C-vector space \(V\) such that point stabilisers are open, then we write \mnotn{\rho^\vee} for the \noterm{contragredient} representation of \(G\) on \(V^\vee\), given by \(\pair{\rho^\vee(g)v^*}v = \pair{v^*}{\rho(g)\inv v}\) for all \(v \in V\), \(v^* \in V^*\), and \(g \in G\).  If \map\gamma{G'}G is a homomorphism from some other group \(G'\), then we write \mnotn{\rho^\gamma} for the representation \(\rho \circ \gamma\) of \(G'\).

If \(X\) is an \(l\)-space in the sense of \cite{bernstein-zelevinskii:gln}*{\S1.1}---\ie, a Hausdorff topological space for which every point has a neighbourhood system consisting of open and compact sets---then we write \mnotn{\Hecke(X)} for its \term{Hecke algebra}, which is, by definition, the \C-vector space of locally constant, compactly supported, \C-valued functions on \(X\).  If
	\begin{itemize}
	\item \(V\) is a finite-dimensional \C-vector space,
	\item \(K_1\) (respectively, \(K_2\)) is a compact group acting on \(X\) on the left (respectively, right),
and	\item \(\rho_i\) is a representation of \(K_i\) on \(V\) with open point stabilisers for \(i \in \sset{1, 2}\),
	\end{itemize}
then we write \mnotn{\Hecke((K_1, \rho_1)\bslash X/(K_2, \rho_2))} for the subspace of those functions \(f \in \Hecke(X) \otimes \End_\C(V^*)\) such that \(f(k_1 x k_2) = \rho_1^\vee(k_1)f(x)\rho_2^\vee(k_2)\) for all \(g \in X\), \(k_1 \in K\), and \(k_2 \in K_2\).  If \(\rho_1\) or \(\rho_2\) is the trivial representation, then we may omit it, writing, for example, \mnotn{\Hecke(K_1\bslash X/K_2)}.  If \(K_1\) equals \(K_2\) and \(\rho_1\) equals \(\rho_2\), then we may write \mnotn{\Hecke(X\sslash K_1, \rho_1)}.

If \(G\) is a topological group acting on \(X\) on the left (respectively, right), then we denote the left (respectively, right) regular action of \(G\) on \(\Hecke(X)\) simply by juxtaposition.  Specifically, for \(f \in \Hecke(X)\) and \(g_1, g_2 \in G\), we define \(g_1 f g_2 \in \Hecke(X)\) by \((g_1 f g_2)(g_1 x g_2) = f(x)\) for all \(x \in X\).

Now suppose that \(X\) is, in addition to being an \(l\)-space, a regular measure space.  If \(K\) is a compact, open subset of \(X\) and \(f\) belongs to \(\Hecke(K)\), then we write \mnotn{\chrc{K, f}} for the element of \(\Hecke(X)\) obtained by extending \(\meas(K)\inv f\) by \(0\).  If \(f\) is the constant function \(1\), then we may omit it and write just \mnotn{\chrc K}.  With the obvious notation, we have \(g_1\chrc{K, f}g_2 = \chrc{g_1 K g_2, g_1 f g_2}\).  We write \(\matnotn{int}{\displaystyle\uint_K f(x)\upd x}\) for \(\displaystyle\int_X \chrc{K, f}(x)\upd x\).

\subsection{Algebraic varieties and algebraic groups}

Let \field be a field, and \bG a smooth group variety over \field.  (We shall soon impose additional assumptions on both.)  We say just ``variety'' for ``smooth variety over \field'', and ``group'' for ``smooth group variety over \field''.  The exception is that we want to discuss mostly, but not only, \emph{complex} vector spaces; so we refer to a ``\C-vector space'' in that case, and, when necessary, to a ``\field-vector space''.

We denote varieties by boldface letters, and their sets of rational points by the corresponding un-bolded letters; so, for example, \(G\) is the group of rational points of \bG.  We denote the identity component of an algebraic group with a following superscript circle, the Lie algebra by the corresponding Fraktur letter or by `Lie', and the dual by `\(\Lie^*\)'; so, for example,
	\begin{itemize}
	\item \matnotn{Gconn}{\bG\conn} is the identity component of \bG,
	\item \mnotn{\Lie(\bG)} or \matnotn g{\pmb\fg} stands for the Lie algebra of \bG,
	\item \(\Lie^*(\bG)\) or \(\pmb\fg^*\) for its dual,
	\item \(\Lie(G)\) or \fg for its space of rational points,
and	\item \(\Lie^*(G)\) or \(\fg^*\) for the space of rational points of its dual (or, equivalently, the dual of its space of rational points).
	\end{itemize}
We always use \(G\conn\) for \((\bG\conn)(\field)\), not necessarily \((\bG(\field))\conn\), which is reduced to a point when \field is totally disconnected.  We write
\begin{align*}
&\map\Int\bG{\Aut(\bG)\conn}, \\
&\map\Ad\bG{\Aut\textsub{Lie}(\Lie(\bG))}, \\
&\map{\Ad^*}\bG{\Aut_\field(\Lie^*(\bG))}, \\
&\map\ad{\Lie(\bG)}{\operatorname{Der}(\Lie(\bG))},
\intertext{and}
&\map{\ad^*}{\Lie(\bG)}{\End_\field(\Lie^*(\bG))}
\end{align*}
for the interior-automorphism, adjoint, and related maps.

We write \mnotn{\Der\bG} for the derived group of \bG; and, if \(\vbG = (\bG^0, \dotsc, \bG^\ell = \bG)\) is a collection of subgroups of \bG, then we write
\begin{align*}
\mnotn{\Lie(\vbG)} &\qtextq{for} (\Lie(\bG^0), \dotsc, \Lie(\bG^\ell)), \\
\mnotn{\Lie^*(\vbG)} &\qtextq{for} (\Lie^*(\bG^0), \dotsc, \Lie^*(\bG^\ell)), \\
\intertext{and}
\mnotn{\Der\vbG} &\qtextq{for} (\Der{\bG^\ell} \cap \bG^0, \dotsc, \Der{\bG^\ell} \cap \bG^{\ell - 1}, \Der{\bG^\ell}).
\end{align*}
We always use \mnotn{\Der G} for \((\Der\bG)(\field)\), not necessarily \(\Der(\bG(\field))\), which may be smaller.

If \bV is a representation of \bG, then we write \(\Disc_\bV\) for the function \anonmap\bG{\GL_1} given by \anonmapto g{\det_\bV(g - 1)}.  If \bX is a pointed, smooth variety on which \bG acts, preserving the preferred point \(x\), and \bV is the tangent space to \bX at \(x\), with the induced action of \bG, then we may write \(\Disc_\bX\), or even just \mnotn[\Disc_{G/H}]{\Disc_X}, in place of \(\Disc_\bV\).  As in \xcite{debacker-spice:stability}*{Definition \xref{defn:disc}}, we write \matnotn{Dred}{\redD_G} for the functions on \(G\) and its Lie algebra given by \(\redD_G(g) = \Disc_{G/\Cent_G(g\semi)}(g)\) and \(\redD_G(Y) = \Disc_{\Lie(G)/\Lie(\Cent_G(Y\semi))}(Y)\), respectively.

We write \matnotn{Grss}{\bG\rss} for the set of regular, semisimple elements of \bG.



\subsection{Algebraic groups over non-Archimedean fields, and subgroups associated to concave functions}
\label{sec:p-adic-group-basics}

For the entire paper, we require that the field \matnotn k\field be complete with respect to a non-trivial discrete valuation \mnotn\ord, for which the residue field \mnotn\ff is finite.  We assume throughout the paper that the residual characteristic of \ff is odd.  (This is needed anyway to use the Weil representation appearing in \cite{yu:supercuspidal}, but we also need to use it again in our discussion of the Weil index in \S\ref{sec:Gauss}, and, in particular, in Proposition \ref{prop:Gauss-to-Weil}.)  Fix a separable closure \(\matnotn{ksep}\sepfield/\field\), and let \matnotn{ktame}\tamefield be the maximal tame extension of \field inside \sepfield.  For any algebraic extension \(E/\field\), we will denote the unique extension of \(\ord\) to a valuation on \(E\) again by \(\ord\).

Throughout the paper, we fix a complex character \matnotn{Lambda}\AddChar of the additive group \field that is non-trivial on the ring of integers in \field, but trivial on its unique maximal ideal.  We will denote the induced character of \ff also by \AddChar.  If \(V\) is a vector space and \(X^*\) is an element of the dual space, then we write \matnotn{Lambda}{\AddChar_{X^*}} for the character \(\AddChar \circ X^*\) of \(V\).

Also for the entire paper, we require that \matnotn G\bG be a reductive, but not necessarily connected, algebraic group over \field.  We write \mnotn{\BB(G)} for the enlarged Bruhat--Tits building of \bG over \field,
and let \anonmapto x{\matnotn x\ox} be the projection from the enlarged to the reduced building.  We will always equip \(G\) (respectively, \(\Lie(G)\)) with Waldspurger's canonical Haar measures \cite{waldspurger:nilpotent}*{\S I.4}, which assign mass \(\card{\Lie(\ol K)}\inv[1/2]\) to the pro-unipotent radical \(K_+\) of a parahoric subgroup \(K\) (respectively, to the pro-nilpotent radical \(\Lie(K_+)\) of its Lie algebra) with reductive quotient \ol K; see \cite{debacker-reeder:depth-zero-sc}*{\S5.1, p.~835}.  Our results are usually stated in such a way that we need not make any reference to a specific choice of measure, but this specific choice is important in, for example, Theorem \ref{thm:dist-G-to-G'}.

If \mc S is a subset of \(\Lie^*(G)\) (or even of some larger set containing it), then we write \mnotn{\OO^G(\mc S)} for the collection of (rational) coadjoint orbits of \(G\) on \(\Lie^*(G)\) that intersect every neighbourhood (in the analytic topology) of \mc S.  Thus, \(\OO^G(0)\) is the analogue of the set of nilpotent orbits in the Lie algebra (\cite{kempf:instability}*{Corollary 4.3} and \cite{adler-debacker:bt-lie}*{Lemma 2.5.1}); and, if \mc S does not intersect \(\Lie^*(G)\), then \(\OO^G(\mc S)\) is empty.

If \(V\) is a finite-dimensional \field-vector space, then the Fourier transform \matnotn{fhat}{\hat f} of \(f \in \Hecke(V)\) is the element of \(\Hecke(V^*)\) given by
\[
\hat f(v^*) = \int_V f(v)\AddChar(\pair{v^*}v)\upd v
\qforall{\(v^* \in V^*\).}
\]
Similarly, for \(f^* \in \Hecke(V^*)\), we define
\[
\matnotn{fcheck}{\check f^*}(v)
= \int_{V^*} f^*(v^*)\ol{\AddChar(\pair{v^*}v)}\upd{v^*}
\qforall{\(v \in V\).}
\]
Note that this depends on the choice of Haar measure \(\upd{v^*}\).  There is a unique choice, called the \term{dual Haar measure} to \(\upd v\), so that \(\Check{\Hat f}\) equals \(f\) for all \(f \in \Hecke(V)\).  Technically speaking, since we have not specified a choice of \(\upd v\), the functions \(\hat f\) and \(\check f^*\) are not well defined, although the measures \(\hat f(v^*)\upd{v^*}\) and \(\check f^*(v)\upd v\) are.  In practice, \(V\) will be the Lie algebra of a reductive \(p\)-adic group, equipped with its canonical measure, so that this ambiguity will not cause a problem; and then we will equip the dual Lie algebra \(V^*\) with the dual measure.


\begin{defn}
\label{defn:normal-harm}
As in \xcite{debacker-spice:stability}*{Definition \xref{defn:normal-harm}}, if \(\pi\) is an admissible representation of \(G\), then we write \mnotn{\Theta_\pi} for the scalar character of \(\pi\), which is the function on \(G\rss\) that represents the distribution character in the sense that
\[
\tr \pi(f) \qeqq \int_G \Theta_\pi(g)f(g)\upd g
\]
for all \(f \in \Hecke(G)\) \cite{adler-korman:loc-char-exp}*{Proposition 13.1}; and \mnotn{\Phi_\pi} for the function \anonmapto g{\abs{\Disc_{G/\Cent_G(g)}(g)}^{1/2}\Theta_\pi(g)} on \(G\rss\).  Similarly, if \(Z^*\) is an element of \(\Lie^*(G)\) that is fixed by the coadjoint action of some maximal torus \bT in \bG, and \OO belongs to \(\OO^G(Z^*)\), then we write \mnotn{\mu^G_\OO} for the orbital-integral distribution on \(\Lie^*(G)\) given by integration against some invariant measure on \OO, and \mnotn{\muhat^G_\OO} for the function on \(\Lie(G)\rss\) that represents its Fourier transform, in the sense that
\[
\mu^G_\OO(\hat f) \qeqq \int_{\Lie(H)} \muhat^G_\OO(Y)f(Y)\upd Y
\]
for all \(f \in \Hecke(\Lie(G)\rss)\); and \(\mnotn{\Ohat^G_\OO}\) for the function \anonmapto Y{\abs{\redD_G(Y)}^{1/2}\abs{\redD_G(Z^*)}^{1/2}\muhat^G_\OO(Y)} on \(\Lie(G)\rss\).  We sometimes write \(\mu^G_{\mc S}\) in place of \(\mu^G_\OO\), and similarly for \(\muhat\) and \(\Ohat\), if \mc S is a non-empty subset of \OO.  As in \xcite{debacker-spice:stability}*{Remark \xref{rem:disc:roots}}, we have written \(\redD_G(Z^*)\) for the product \(\prod \pair{Z^*}{\textup d\alpha^\vee(1)}\) over all weights \(\alpha\) of \(\bT_\sepfield\) on \(\Lie(\bG_\sepfield)\) for which the multiplicand is non-\(0\).  (For now we are just establishing notation, so we ignore questions about convergence of the orbital integral; but see the discussion preceding Theorem \ref{thm:asymptotic-exists}.)
\end{defn}

\begin{rem}
\label{rem:normal-harm}
Note that the distribution character \anonmapto f{\Theta_\pi(f)} on \(\Hecke(G)\) depends on the choice of Haar measure on \(G\), but the scalar character \anonmapto g{\Theta_\pi(g)} on \(G\rss\) does not.  The distribution \anonmapto f{\muhat^G_\OO(f)} on \(\Hecke(\Lie(G)\) depends on the choice of invariant measure on \OO, which is determined by a Haar measure on \(G\) and one on the centraliser in \(G\) of an element of \OO.  A canonical choice is described in \cite{moeglin-waldspurger:whittaker}*{\S I.8} (see also \cite{debacker:homogeneity}*{\S3.4, p.~410}).  The representing function \anonmapto Y{\muhat^G_\OO(Y)} on \(\Lie(G)\rss\) depends on both these, \emph{and} a Haar measure on \(\Lie(G)\).
\end{rem}

As in \xcite{adler-spice:good-expansions}*{\S\xref{sec:notation-generalities}, p.~8}, we put \(\matnotn R\tR = \R \sqcup \Rp\R \sqcup \sset\infty\).  We define \(\Rpp{\Rp r} = \Rp r\) for all \(r \in \R\), and \(\Rp{\pm\infty} = \pm\infty\).  We define \(\tilde r = \Rpp{-r}\) and \(\widetilde{\Rp r} = -r\) for all \(r \in \R\), and \(\widetilde{\pm\infty} = \mp\infty\).  We extend addition on \R to \(\tR \sqcup \sset{-\infty}\) by putting \(r + \Rp s = (\Rp r) + s = \Rpp{r + s}\) for all \(r, s \in \R\); \(r + (-\infty) = -\infty + r = -\infty\) and \(r + \infty = \infty + r = \infty\) for all \(r \in \tR\) with \(r < \infty\); and \(-\infty + (-\infty) = -\infty\) and \(\pm\infty + \infty = \infty + (\pm\infty) = \infty\).  We extend subtraction on \R to a \emph{partial} operation on \(\tR \sqcup \sset{-\infty}\) by putting \((\Rp r) - s = \Rpp{r - s}\) for all \(r, s \in \R\); \(r - (\pm\infty) = \mp\infty\) for all \(r \in \tR\) with \(r < \infty\); and \(\pm\infty - (\mp\infty) = \pm\infty\).

We follow a suggestion of Cheng-Chiang Tsai and, for \((x, r) \in \BB(G) \times \R_{\ge 0}\), replace the usual notations \(G_{x, r}\) and \(G_{x, \Rp r}\) for Moy--Prasad subgroups (\cite{moy-prasad:k-types}*{\S\S2.6, 3.2, 3.5} and \cite{moy-prasad:jacquet}*{\S\S3.2, 3.3}) by \(\sbtl G x r\) and \(\sbtlp G x r\), respectively.  For convenience, we make the convention that \(\sbtlp G x{\Rp r}\) means \(\sbtlp G x r\).  Note that, by definition, \(\sbtl G x r\) equals \(\sbtl{G\conn}x r\).  We use similar notation for groups associated to concave functions \cite{bruhat-tits:reductive-groups-1}*{\S6.4.3},
\justnotn[Gxf]{\sbtl G x f}\justnotn[Gxf]{\sbtlp G x f}%
as in \xcite{adler-spice:good-expansions}*{Definition \xref{defn:vGvr}}, and for the Lie algebra, where we can drop the requirement that \(r\) be non-negative (respectively, that the relevant function be concave); and in Definition \ref{defn:vGvr}, where we define an analogous class of groups.

These conventions inevitably suggest the notations \matnotn{Gxr}{\sbat G x r} for \(\sbtl G x r/\sbtlp G x r\); we adopt this and obvious variants, even though they are a bit misleading in the case where \field has mixed characteristic.  In particular, we write \(\sbjtl\field r = \set{t \in \field}{\ord(t) \ge r}\), and similarly for other notation related to the filtration on \field; so, for example, \ff equals \(\sbjat\field 0 = \sbjtl\field 0/\sbjtlp\field 0\).

\section{Compact, open subgroups}
\label{sec:subgps}

\subsection{Groups associated to concave functions}
\label{sec:concave}

In Definition \ref{defn:vGvr}, we define a class of compact, open subgroups related to those constructed in \cite{adler-spice:good-expansions}*{Definition \xref{defn:vGvr}}.  For technical reasons involving the presence of algebraic subgroups that are not of full rank in \bG, we work mostly with `depth vectors' (see Definition \ref{defn:tame-Levi}) and `depth matrices' (see Definition \ref{defn:vGvr}), rather than concave functions as in \cite{adler-spice:good-expansions}; but we do need to consider such functions for two technical results.

Lemma \ref{lem:heres-a-gp} is mostly a straightforward generalisation of part of \xcite{adler-spice:good-expansions}*{Proposition \xref{prop:heres-a-gp}}, with essentially the same proof.  We use it only in the proof of Lemma \ref{lem:Ko-and-J}, where we need an additional technical fact, an analogue of \xcite{adler-spice:good-expansions}*{Lemma \xref{lem:more-vGvr-facts}}; so, for convenience, we state both facts together.

Working with groups associated to arbitrary concave functions rather than depth matrices allows us to `move' a group \(\sbtl G x f\) to another point, by writing it as \(\sbtl G y{f + (y - x)}\).  The condition on \(f \vee f\) in Lemma \ref{lem:heres-a-gp} is needed to use \xcite{adler-spice:good-expansions}*{Proposition \xref{prop:H1-iso}}, but it is just for maximal generality; for us, it suffices to know that it is satisfied whenever \(f\) is a translate by a linear function of an everywhere positive, concave function.  Lemma \ref{lem:heres-a-gp} relies on \xcite{adler-spice:good-expansions}*{Hypothesis \xref{hyp:torus-H1-triv}}, but this is automatically satisfied when \bG is \tamefield-split, which will be the case when we use the result (in Lemma \ref{lem:Ko-and-J}).

In the notations \(\stab_{\bG'(E)}(\ox)\) and \(\stab_G(\ox)\), note that the symbol \ox stands for a point in the reduced building of \(\bG(E)\) or \(G\), not necessarily of \(\bG'(E)\) or \(G'\).  The lemma relies crucially on the good descent properties of full stabilisers of points in the building; it would not work, for example, if we replaced \(\stab_{\bG'(E)}(\ox)\) and \(\stab_{G'}(\ox)\) by their parahoric subgroups \(\sbtl{\bG'(E)}x0\) and \(\sbtl{G'}x 0\).

\begin{lem}
\initlabel{lem:heres-a-gp}
Suppose that
	\begin{itemize}
	\item \xcite{adler-spice:good-expansions}*{Hypothesis \xref{hyp:torus-H1-triv}} holds,
	\item \(\bG'\) is the centraliser in \(\bG\conn\) of a \tamefield-split torus,
	\item \bT is a maximally \tamefield-split, maximal torus in \bG,
	\item \(x\) is a point of \(\BB(T)\),
	\item \(E/\field\) is a tame extension,
and	\item \(f_1\) and \(f_2\) are \tR-valued, Galois-invariant, concave \cite{bruhat-tits:reductive-groups-1}*{\S6.4.3} functions on the set \(\Weight(\bG_\sepfield, \bT_\sepfield)\) of weights of \(\bT_\sepfield\) in \(\bG_\sepfield\) satisfying \(f_j(0) > 0\) and \(f_j(\alpha) < (f_j \vee f_j)(\alpha)\) \xcite{adler-spice:good-expansions}*{Definition \ref{defn:concave-vee}} whenever \(\alpha \in \Weight(\bG_\sepfield, \bT_\sepfield)\) is such that \(f_j(\alpha) < \infty\), for \(j \in \sset{1, 2}\).
	\end{itemize}
Then
\begin{align*}
\stab_{G'}(\ox)\sbtl G x{f_1} \cap \sbtl G x{f_2}
&\qeqq
\sbtl{G'}x{\max \sset{0, f_2}}\dotm\sbtl G x{\max \sset{f_1, f_2}} \\
\intertext{and}
\stab_{\bG'(E)}(\ox)\sbtl{\bG(E)}x{f_1}\dotm\sbtl{\bG(E)}x{f_2} \cap G
&\qeqq
\stab_{G'}(\ox)\sbtl G x{f_1}\dotm\sbtl G x{f_2}.
\end{align*}
\end{lem}

\begin{proof}
The first statement is proven, just as in \xcite{adler-spice:good-expansions}*{Lemma \xref{lem:more-vGvr-facts}}, by reducing to a combination of \xcite{adler-spice:good-expansions}*{Lemmas \xref{lem:gen-iwahori-factorization} and \xref{lem:Zariski}}; the crucial point is that \(\stab_{G'}(\ox)\sbtl G x{f_1}\) is contained in the product of \(G'\) with the root groups corresponding to roots of \bT in \bG.

For the second statement, let \(g_1\) and \(g_2\) be any concave functions satisfying the analogues of the conditions imposed on \(f_1\) and \(f_2\).  By the first statement (applied to \(\bG(E)\)), we have that
\[
\stab_{\bG'(E)}(\ox)\sbtl{\bG(E)}x{g_1} \cap \sbtl{\bG(E)}x{g_2}
\qeqq
\sbtl{\bG(E)}x g,
\]
where
\[
g(\alpha) = \begin{cases}
\max \sset{0, g_2(\alpha)},           & \alpha \in \Weight(\bG'_\sepfield, \bT_\sepfield) \\
\max \sset{g_1(\alpha), g_2(\alpha)}, & \text{otherwise,}
\end{cases}
\]
hence by \xcite{adler-spice:good-expansions}*{Proposition \xref{prop:H1-iso}} that the cohomology
\[
\operatorname H^1(E/\field, \stab_{\bG'(E)}(\ox)\sbtl{\bG(E)}x{g_1} \cap \sbtl{\bG(E)}x{g_2})
\]
is trivial.  (This is where we use the hypotheses about \(g_j \vee g_j\), and also where we require \xcite{adler-spice:good-expansions}*{Hypothesis \xref{hyp:torus-H1-triv}}.)  Thus, considering the short exact sequence in cohomology associated to the exact sequence of pointed sets
\begin{align*}
1
& {}\to \stab_{\bG'(E)}(\ox)\sbtl{\bG(E)}x{g_1} \cap \sbtl{\bG(E)}x{g_2} & \\
& {}\to \stab_{\bG'(E)}(\ox)\sbtl{\bG(E)}x{g_1} \times \sbtl{\bG(E)}x{g_2} & \\
& {}\to \stab_{\bG'(E)}(\ox)\sbtl{\bG(E)}x{g_1}\dotm\sbtl{\bG(E)}x{g_2}
& {}\to 1,
\end{align*}
and using \xcite{adler-spice:good-expansions}*{Lemma \xref{lem:tame-descent}}, we see that the sequence
\begin{equation}
\tag{$*$}
\sublabel{eq:exact}
\begin{aligned}
1
& {}\to \stab_{\bG'(E)}(\ox)\sbtl{\bG(E)}x{g_1} \cap \sbtl G x{g_2} & \\
& {}\to (\stab_{\bG'(E)}(\ox)\sbtl{\bG(E)}x{g_1} \cap G) \times \sbtl G x{g_2} & \\
& {}\to \stab_{\bG'(E)}(\ox)\sbtl{\bG(E)}x{g_1}\dotm\sbtl{\bG(E)}x{g_2} \cap G
& {}\to 1
\end{aligned}
\end{equation}
of Galois-fixed points remains exact.

Applying \loceqref{eq:exact} with \(g_1 = \infty\), so that \(\sbtl{\bG(E)}x{g_1}\) equals \(\sset1\), and \(g_2 = f_1\), and using that \(\stab_{\bG'(E)}(\ox) \cap G\) equals \(\stab_{G'}(\ox)\), shows that
\[
\stab_{G'}(\ox)\sbtl G x{f_1}
\qeqq
\stab_{\bG'(E)}(\ox)\sbtl{\bG(E)}x{f_1} \cap G.
\]
Then applying \loceqref{eq:exact} again, with \(g_j = f_j\) for \(j \in \sset{1, 2}\), shows that
\[
\stab_{G'}(\ox)\sbtl G x{f_1}\dotm\sbtl G x{f_2}
= (\stab_{\bG'(E)}(\ox)\sbtl{\bG(E)}x{f_1} \cap G)\sbtl G x{f_2}
\]
equals \(\stab_{\bG'(E)}(\ox)\sbtl{\bG(E)}x{f_1}\dotm\sbtl{\bG(E)}x{f_2} \cap G\), as desired.
\end{proof}

Lemma \ref{lem:der-master-comm} is a slightly stronger version of \xcite{adler-spice:good-expansions}*{Lemma \xref{lem:master-comm}}, adapted to take into account the depth of commutators in \(\Der G\) (not just in \(G\)).  Again, we find it convenient to state it in terms of concave functions, for which we can use the result from \xcite{adler-spice:good-expansions} to bootstrap, rather than necessarily the depth matrices appearing in Definition \ref{defn:vGvr} below.  Having done so, however, we use it only to prove Lemma \ref{lem:filtration} (which \emph{does} concern groups and Lie algebras associated to depth matrices).

\begin{lem}
\initlabel{lem:der-master-comm}
Suppose that
	\begin{itemize}
	\item \bT is a maximally \tamefield-split, maximal torus in \bG,
	\item \(x\) is a point of \(\BB(T)\),
	\item \(f_1\) and \(f_2\) are (\(\tR \cup \sset{-\infty}\))-valued, Galois-invariant functions on the set of weights of \(\bT_\sepfield\) in \(\bG_\sepfield\).
	\end{itemize}
If, for \(j \in \sset{1, 2}\),
	\begin{itemize}
	\item \(g_j\) belongs to \(\sbtl G x{f_j}\),
	\item \(Y_j\) to \(\sbtl{\Lie(G)}x{f_j}\),
and	\item \(X^*_j\) to \(\sbtl{\Lie^*(G)}x{f_j}\)
	\end{itemize}
then
	\begin{itemize}
	\item \(\ad(Y_1)Y_2\) belongs to \(\sbtl{\Lie(\Der G)}x{f_1 \bowtie f_2}\),
and	\item \(\ad^*(Y_1)X^*_2\) to \(\sbtl{\Lie^*(\Der G)}x{f_1 \bowtie f_2}\);
	\end{itemize}
if \(f_1\) is concave, then
	\begin{itemize}
	\item \((\Ad(g_1) - 1)Y_2\) belongs to \(\sbtl{\Lie(\Der G)}x{f_1 \rtimes f_2}\),
and	\item \((\Ad^*(g_1) - 1)X^*_2\) belongs to \(\sbtl{\Lie^*(\Der G)}x{f_1 \rtimes f_2}\);
	\end{itemize}
and, if \(f_1\) and \(f_2\) are both concave, then
	\begin{itemize}
	\item \(\comm{g_1}{g_2}\) belongs to \(\sbtl{\Der G}x{f_1 \vee f_2}\).
	\end{itemize}
Here, \(f_1 \vee f_2\) is the function defined by
\begin{align*}
(f_1 \vee f_2)(\alpha)    & {}= \inf_{\sum a_i + \sum b_j = \alpha} \sum f_1(a_i) + \sum f_2(b_j)
\intertext{in \xcite{adler-spice:good-expansions}*{Definition \xref{defn:concave-vee}}, and \(f_1 \bowtie f_2\) and \(f_1 \rtimes f_2\) are its analogues defined by}
(f_1 \bowtie f_2)(\alpha) & {}= \inf_{a + b = \alpha} f_1(a) + f_2(b) \\
\intertext{and}
(f_1 \rtimes f_2)(\alpha) & {}= \inf_{\sum a_i + b = \alpha} \sum f_1(a_i) + f_2(b),
\end{align*}
for all weights \(\alpha\) of \(\bT_\sepfield\) on \(\Lie(\bG_\sepfield)\).
\end{lem}

\begin{proof}
We only prove the statement about commutators in the group; the others are easier.

\def\Comm{\operatorname{Comm}}
We may, and do, assume, upon passing to a tame extension, that \bT is contained in a Borel subgroup of \bG.  Then we have that any group of the form \(\sbtl G x f\) is generated by \(\sbjtl T{f(0)}\) and the various \(\sbtl B x f\), where \bB is a Borel subgroup of \bG containing \bT; and similarly for \(\sbtl{\Der G}x f\).  (In fact, we need only take two opposite Borel subgroups.)  In particular, we have that \(\sbtl G x f\) is generated by \(\sbjtl T{f(0)}\) and \(\sbtl{\Der G}x f\).

We use the basic fact that, if \mc G is a subgroup of \(G\) that normalises \(\sbtl{\Der G}x{f_1 \vee f_2}\), and \mc S is a subset of \(G\), then \(\sett{g \in \mc G}{\(\comm g h \in \sbtl{\Der G}x{f_1 \vee f_2}\) for all \(h \in \mc S\)}\) is a subgroup of \(G\).  We temporarily introduce the notation \(\Comm_{\mc G}(\mc S)\) for this subgroup.

For this paragraph, fix \(j \in \sset{1, 2}\).  By \xcite{adler-spice:good-expansions}*{Lemma \xref{lem:master-comm}} (applied to \(\Der\bG\)), we have that \(\sbtl{\Der G}x{f_j}\) normalises \(\sbtl{\Der G}x{f_1 \vee f_2}\), and that \(\sbjtl T{f_j(0)}\) normalises \(\sbtl U x{f_1 \vee f_2}\) for the unipotent radical \bU of any Borel subgroup \bB of \bG containing \bT; so, since of course \(\sbjtl T{f_j(0)}\) normalises any subgroup of \(T\), we have that \(\sbtl G x{f_j}\) normalises \(\sbtl{\Der G}x{f_1 \vee f_2}\).  Thus, the notation \(\Comm_{\sbtl G x{f_j}}(\mc S)\) makes sense.

For this paragraph, fix a Borel subgroup \bB of \bG containing \bT, and let \bU be its unipotent radical.  By another application of \xcite{adler-spice:good-expansions}*{Lemma \xref{lem:master-comm}}, we have that the commutator of \(\sbjtl T{f_1(0)}\) with \(\sbtl B x{f_2}\) belongs to \(\sbtl U x{f_1 \vee f_2} \subseteq \sbtl{\Der G}x{f_1 \vee f_2}\).  Thus, since \(\sbtl U x{f_2}\) is contained in \(\sbtl{\Der G}x{f_2}\), we have that \(\Comm_{\sbjtl T{f_1(0)}}(\sbtl B x{f_2})\) equals \(\sbjtl T{f_1(0)}\), and that \(\Comm_{\sbtl G x{f_1}}(\sbtl U x{f_2})\) contains both \(\sbjtl T{f_1(0)}\) and \(\sbtl{\Der G}x{f_1}\), hence equals \(\sbtl G x{f_1}\).  Thus
\begin{align}
\sublabel{eq:T}
\tag{$*$}
\Comm_{\sbtl G x{f_2}}(\sbjtl T{f_1(0)})
&\qtextq{contains}
\sbtl B x{f_2}, \\
\intertext{and}
\sublabel{eq:DG}
\tag{$**$}
\Comm_{\sbtl G x{f_2}}(\sbtl G x{f_1})
&\qtextq{contains}
\sbtl U x{f_2}.
\end{align}
Symmetric results also hold.

By \loceqref{eq:T} (applied to a pair of opposite Borel subgroups), we have that \(\Comm_{\sbtl G x{f_2}}(\sbjtl T{f_1(0)})\) equals \(\sbtl G x{f_2}\), hence that \(\Comm_{\sbtl G x{f_1}}(\sbtl G x{f_2})\) contains \(\sbjtl T{f_1(0)}\).  By \loceqref{eq:DG} (or, rather, its analogue with the indices \(j = 1\) and \(j = 2\) switched), we have that \(\Comm_{\sbtl G x{f_1}}(\sbtl G x{f_2})\) also contains \(\sbtl U x{f_2}\) whenever \bU is the unipotent radical of a Borel subgroup \bB of \bG containing \bT.  It follows that \(\Comm_{\sbtl G x{f_1}}(\sbtl G x{f_2})\) equals \(\sbtl G x{f_1}\), as desired.
\end{proof}

\subsection{Groups associated to depth matrices}
\label{sec:depth-matrix}

In Definition \ref{defn:vGvr}, we use Hypothesis \ref{hyp:funny-centraliser} to build a class of compact, open subgroups of \(G\) generalising those constructed in \xcite{adler-spice:good-expansions}*{Definition \xref{defn:vGvr}}.  Definition \ref{defn:tame-Levi} begins to set up that generalisation.

\begin{defn}
\label{defn:tame-Levi}
A subgroup \(\bG'\) of \bG is called a \term{tame, twisted Levi subgroup} if \(\bG'_\tamefield\) is a Levi subgroup of \(\bG_\tamefield\), in the sense of \cite{digne-michel:non-connexe}*{D\'efinition 1.4}.

A collection \(\vbG = (\bG^0, \dotsc, \bG^\ell = \bG)\) of subgroups of \bG is called a \term{tame, twisted Levi sequence} if each \(\bG^j\) is a tame, twisted Levi subgroup of \bG, and there is a \tamefield-split torus \bS in \bG that contains the maximal \tamefield-split torus in each \(\Zent(\bG^j)\conn\).  We write \(\BB(\vG)\) for \(\bigcap_{j = 0}^\ell \BB(G^j)\), and say that the sequence \term{contains} an element \(\gamma \in G\) exactly when \(\gamma\) belongs to \(\bigcap_{j = 0}^\ell G^j\).

For this definition, put \(\bT = \Cent_{\bG\conn}(\bS)\), and write \Weight for the collection of weights of \(\bT_\sepfield\) on \(\Lie(\bG_\sepfield)\).

A \term{depth vector} (for \vbG) is a vector \(\vec a = (a_0, \dotsc, a_d)\) with entries in \(\tR \cup \sset{-\infty}\).  We define the function \(f_{\vbG, \vec a}\) on \Weight by putting \(f_{\vbG, \vec a}(\alpha) = a_j\) if \(\alpha\) is a weight of \(\Cent_\bG(\bS)_\sepfield\) on \(\Lie(\bG^j_\sepfield)\), but not on \(\Lie(\bG^{j_-}_\sepfield)\) for any \(0 \le j_- < j\).  For any \(x \in \BB(T)\), we define \matnotn{Lie}{\sbtl{\Lie(\vG)}x{\vec a}} and \matnotn{Lie}{\sbtl{\Lie^*(\vG)}x{\vec a}} to be \(\sum \sbtl{\Lie(\vG)}x f\) and \(\sum \sbtl{\Lie^*(\vG)}x f\), where the sums run over all Galois-invariant, \tR-valued functions \(f\) on \Weight for which the inequality \(f_{\vbG, \vec\alpha} \le f\) is satisfied.

We say that \(\vec a\) is \term[depth vector!concave]{concave} if the inequality \(2a_{j_+} \ge a_{j_-}\) holds for all \(0 \le j_- \le j_+ \le \ell\), and that it is \term[depth vector!grouplike]{grouplike} if, further, \(a_0\) is positive.  In this case, for any \(x \in \BB(T)\), we define \matnotn{Gxa}{\sbtl\vG x{\vec a}} to be \(\sgen[\big]{\bigcup \sbtl\vG x f}\), where the union runs over all Galois-invariant, \tR-valued, \emph{concave} \cite{bruhat-tits:reductive-groups-1}*{\S6.4.3} functions \(f\) on \Weight for which the inequality \(f_{\vbG, \vec\alpha} \le f\) is satisfied.  (The notation is as in \xcite{adler-spice:good-expansions}*{Definition \xref{defn:vGvr}}.)
\end{defn}

The non-full-rank subgroups that arise in Definition \ref{defn:vGvr} depend on a semisimple element \mnotn\gamma of \(G\) (not necessarily \(G\conn\)).  Choose such an element, and
\justnotn[P]{\bP^\pm}\justnotn[N]{\bN^\pm}%
write \(\bP^-\) (respectively, \(\bP^+\)) for the parabolic subgroup of \bG dilated (respectively, contracted) by \(\gamma\) \cite{deligne:support}*{\S1, p.~155}, and \(\bN^-\) (respectively, \(\bN^+\)) for its unipotent radical.  Write \matnotn M\bM for \(\bP^- \cap \bP^+\).

Let \mnotn r be a non-negative real number, and put \(\mnotn s = r/2\).  In this section, we write \matnotn H\bH for \(\CC\bG r(\gamma)\); but we caution the reader that we use different notation in the proof of Proposition \ref{prop:lattice-orth}.

In \xcite{adler-spice:good-expansions}*{Definition \xref{defn:funny-centralizer}}, Adler and the author defined connected, full-rank subgroups \(\CC\bG i(\gamma)\) of \bG associated to an element \(\gamma\) of \(G\conn\) satisfying certain tameness hypotheses.  (These groups are always connected, whereas we allow those in Hypothesis \ref{hyp:funny-centraliser} to be disconnected; but Hypothesis \ref{hyp:gamma} involves only the identity components of these groups, so the difference does not matter.)  With an eye towards future applications, including particularly removing the requirement that \(\gamma\) belong to the identity component, and possibly the tameness requirement, we state very precisely the properties that we need in Hypothesis \ref{hyp:funny-centraliser}.

Hypothesis \ref{hyp:funny-centraliser} always holds for \(\gamma\) an element of a tame torus satisfying \xcite{adler-spice:good-expansions}*{Definition \xref{defn:S-is-good}}, as long as \xcite{adler-spice:good-expansions}*{Hypotheses \xref{hyp:reduced}--\xref{hyp:torus-H1-triv}} are satisfied \xcite{adler-spice:good-expansions}*{Proposition \xref{prop:compatibly-filtered-tame-rank}, Lemma \xref{lem:simult-approx}, and Lemma \xref{lem:funny-centralizer-descends}}.  (Actually, those results require that \(\gamma\) be compact modulo centre; but we may reduce to that case by working inside \bM.)

Although we \emph{allow} the groups \(\CC\bG i(\gamma)\) to be disconnected, Hypothesis \ref{hyp:funny-centraliser} is phrased in such a way that we may replace each \(\CC\bG i(\gamma)\) by its identity component.

\begin{hyp}
\initlabel{hyp:funny-centraliser}
There is a decreasing sequence of (possibly disconnected, possibly non-full-rank) reductive subgroups \((\mnotn{\CC\bG i(\gamma\pinv)})_{\substack{i \in \tR \cup \sset{-\infty} \\ i \le r}}\) of \bG such that the following hold for all \(i\).  We write \mnotn{\CC\bG i(\gamma)} (respectively, \(\CC\bG i(\gamma\inv)\)) for the parabolic subgroup of \(\CC\bG i(\gamma\pinv)\) dilated (respectively, contracted) by \(\gamma\) \cite{deligne:support}*{\S1, p.~155}.
	\begin{enumerate}
	\item\sublabel{basic}
		\begin{itemize}
		\item \(\CC\bG{-\infty}(\gamma\pinv)\conn\) equals \(\bG\conn\),
		\item \(\CC\bG 0(\gamma) \cap \bG\conn\) equals \(\bM\conn\),
		\item \(\CCp\bG 0(\gamma) \cap \bG\conn\) is the centraliser in \(\bG\conn\) of the absolutely-semisimple-modulo-\(\Zent(\bM\conn)\) part of \(\gamma\) \xcite{spice:jordan}*{Definition \xref{defn:top-F-ss-unip}},
	and	\item \(\CC\bG r(\gamma)\) contains \(\Cent_\bG(\gamma)\).
		\end{itemize}
	\item\sublabel{Lie} The Lie algebra \(\Lie(\CC\bG i(\gamma)_\sepfield)\) is the sum of the weight spaces for the action of \(\gamma\) on \(\Lie(\bG_\sepfield)\) corresponding to weights \(\lambda \in \sepfield \setminus \sbjtlp{\field\sep}0\)
for which the inequality \(\ord(\lambda - 1) \ge i\) holds, and similarly for \(\gamma\inv\).
	\item\sublabel{Levi} If \(\bG'\) is a subgroup of \bG that is normalised by \(\gamma\), then \(\CC\bG i(\gamma)\conn \cap \bG'\) and \(\CC\bG i(\gamma\pinv)\conn \cap \bG'\) are smooth.  If
		\begin{itemize}
		\item \(\bG'\) is a subgroup of \bG,
		\item \(\bS'\) is a \tamefield-split torus centralising \(\bG\primeconn\),
		\item \(\bG'\dotm\bS'\) is a tame, twisted Levi subgroup of \bG containing \(\gamma\),
and		\item \(\gamma\) normalises both \(\bS'\) and \(\bG'\),
		\end{itemize}
then \((\CC\bG i(\gamma\pinv)\conn \cap \bG')(\CC\bG i(\gamma\pinv)\conn \cap \bS')\) is a tame, twisted Levi subgroup of \(\CC\bG i(\gamma\pinv)\conn\).
	\item\sublabel{more-vGvr-facts} If \(\vbG = (\bG^0 \subseteq \dotsb \subseteq \bG^\ell = \bG)\) is a tame, twisted Levi sequence in \bG such that \(\gamma\) belongs to \(G^0\) (hence in each \(G^j\)), then there is a commutative diagram of embeddings of buildings
\[\xymatrix{
\BB(\CC{G^{j_1}}{i_1}(\gamma\pinv)) \ar[r]\ar[d] & \BB(\CC{G^{j_1}}{i_2}(\gamma\pinv)) \ar[d] \\
\BB(\CC{G^{j_2}}{i_1}(\gamma\pinv)) \ar[r]       & \BB(\CC{G^{j_2}}{i_2}(\gamma\pinv))
}\]
(for all \(0 \le j_1 \le j_2 \le \ell\) and \(i_2 \le i_1 \le r\)) such that, if
		\begin{itemize}
		\item \(i \in \tR \cup \sset{-\infty}\) satisfies \(i \le r\),
		\item \(x\) is a point of \(\BB(\CC{G^0}i(\gamma\pinv))\),
	and	\item \(\vec a\) is a depth vector,
		\end{itemize}
then we have that
\[
\sbtl\vG x{\vec a} \cap \CC G i(\gamma\pinv)
\qeqq
\sbtl{\CC\vG i(\gamma\pinv)}x{\vec a},
\]
and similarly for the Lie algebra.
	\end{enumerate}
\end{hyp}

Remark \ref{rem:tame-Levi} discusses the descent properties of tame, twisted Levi subgroups.

\begin{rem}
\label{rem:tame-Levi}
If \(\bG'\) is a Levi subgroup of \bG, then \(\bG\conn \cap \bG'\) equals \(\Cent_{\bG\conn}(\Zent(\bG\primeconn))\) is a Levi subgroup of \(\bG\conn\).  In particular, it is connected, so equals \(\bG\primeconn\).

If \((\bL, \bG', \bG)\) is a tame, twisted Levi sequence in \bG, then \bL contains \(\bS' \ldef \Zent(\bG\primeconn)\).  Recall \cite{digne-michel:non-connexe}*{D\'efinition 1.4} that \(\bG'\) equals \(\Norm_\bG(\bG\primeconn, \bQ\conn)\), where \(\bQ\conn\) is a parabolic subgroup of \(\bG\conn\) with Levi component \(\bG\primeconn\).  Let \(\lambda_0\) be a cocharacter of \(\bS'\) so that \(\bQ\conn\) equals \(\bP_{\bG\conn}(\lambda_0)\), in the notation of \cite{springer:lag}*{\S13.4.1}.  We have that \(\bG'\) equals \(\Norm_\bG \set{\lambda \in \bX_*(\bS')}{\bP_{\bG\conn}(\lambda) = \bQ\conn}\), so that \(\bL \cap \bG'\) equals \(\Norm_\bL \set{\lambda \in \bX_*(\bS')}{\bP_{\bG\conn}(\lambda) = \bQ\conn} \subseteq \Norm_\bL(\Cent_{\bL\conn}(\bS'), \bP_{\bL\conn}(\lambda_0))\).    The reverse containment is obvious, so we have equality.  In particular, \(\bL \cap \bG'\) is a tame, twisted Levi subgroup of \bL.

Our insistence on referring only to identity components in Hypothesis \ref{hyp:funny-centraliser} makes Hypothesis \initref{hyp:funny-centraliser}\subpref{Levi} somewhat awkward.  Using the notation there, we allow ourselves to write something like \mnotn{\CC{\bG'}i(\gamma)} for \(\CC\bG i(\gamma) \cap \bG'\), as long as it is understood that we are speaking only of its identity component.  For example, we may refer to \(\Lie(\CC{\bG'}i(\gamma))\), or use \(\CC{\bG'}i(\gamma\pinv)\) as a term in a tame, twisted Levi sequence \(\CC\vbG i(\gamma\pinv)\) in \(\CC\bG i(\gamma\pinv)\) (as in Hypothesis \initref{hyp:funny-centraliser}\subpref{more-vGvr-facts}), even though we have not guaranteed that it is actually a tame, twisted Levi subgroup, because a group such as \(\sbtl{\CC\vG i(\gamma\pinv)}x{\vec r}\) depends only on the identity components of the groups in the vector \(\CC\vbG i(\gamma\pinv)\).
\end{rem}

By Hypothesis \initref{hyp:funny-centraliser}\subpref{Lie}, there is a canonical \(\CC\bG i(\gamma)\)-stable complement \mnotn{\Lie(\CC\bG i(\gamma))^\perp} to \(\Lie(\CC\bG i(\gamma))\) in \(\Lie(\bG)\) (namely, the sum of the other weight spaces for \(\gamma\)).  We identify \(\Lie^*(\CC\bG i(\gamma))\) with the subset of \(\Lie^*(\bG)\) that annihilates \(\Lie(\CC\bG i(\gamma))^\perp\), and write \mnotn{\Lie^*(\CC\bG i(\gamma))^\perp} for the subset of \(\Lie^*(\bG)\) that annihilates \(\Lie(\CC\bG i(\gamma))\).

Definition \ref{defn:vGvr} is closely related to \xcite{adler-spice:good-expansions}*{Definition \xref{defn:vGvr}}.  If we are dealing with tame elements \(\gamma\) in the identity component of \(G\), so that the groups \(\CC\bG i(\gamma)\) have full rank, then our definition has more restrictive hypotheses (and defines the same groups); but it is set up to accommodate the case where the groups \(\CC\bG i(\gamma)\) may not have full rank in \bG.

\begin{defn}
\label{defn:vGvr}
Suppose that \vbG is a tame, twisted Levi sequence in \bG containing \(\gamma\).  Let
\justnotn{f\supn}\justnotn{f\supd}\justnotn{f\supc}%
\[
f = \bigl((f\supn_i)_{\substack{i \in \tR \\ 0 \le i \le r}}, (f\supd_i, f\supc_i)_{\substack{i \in \tR \cup \sset{-\infty} \\ i < 0}}\bigr)
\]
be a collection of depth vectors.  The entire ensemble \(f\) is called a \term{depth matrix} (for \vbG and \(\gamma\)).  We read the superscripts `n', `d', and `c' as shorthand for `neutral', `dilated', and `contracted', referring to the action of \(\gamma\) \cite{deligne:support}*{\S1, p.~155}.  We occasionally write \(\mnotn{f^-} = (f\supn, f\supd)\) and \(\mnotn{f^+} = (f\supn, f\supc)\).

For any \(x \in \BB(\vG)\), we write \matnotn[\sbtl{\Lie(\vG)}x f]{Lie}{\lsub\gamma\sbtl{\Lie(\vG)}x f} for
\begin{multline*}
\bigoplus_{0 \le i \le r} \Lie(\CCp G i(\gamma))^\perp \cap \sbtl{\Lie(\CC\vG i(\gamma))}x{\vec a\supn_i} \oplus{} \\
\bigoplus_{i < 0}
	\Lie(\CCp G i(\gamma\pinv))^\perp \cap (\sbtl{\Lie(\CC\vG i(\gamma))}x{\vec a\supd_i} \oplus \sbtl{\Lie(\CC\vG i(\gamma\inv))}x{\vec a\supc_i}),
\end{multline*}
and similarly for \matnotn[\sbtl{\Lie^*(\vG)}x f]{Lie}{\lsub\gamma\sbtl{\Lie^*(\vG)}x f}.  (The notation, here and in the group case, is as in Definition \ref{defn:tame-Levi}.)

We say that \(f\) is \term[depth matrix!grouplike]{grouplike} if
	\begin{itemize}
	\item \(f\supn_r\) is grouplike;
	\item each \(f\supn_i\), \(f\supd_i\), and \(f\supc_i\) is concave;
	\item the inequality \(f^-_{i_-} \ge f^-_{i_+}\) holds for all \(i_- \le i_+ < 0\) and all \(0 \le i_- \le i_+ \le r\), and similarly for \(f^+\);
and	\item the inequality \(f\supd_{i_-j_-} + f\supc_{i_-j_+} \ge f\supn_{i_+j_+}\) holds for all \(i_- < 0 \le i_+ \le r\) and \(0 \le j_- \le j_+ \le \ell\).
	\end{itemize}
In this case, for any \(x \in \BB(\vG)\), we write \matnotn[\sbtl\vG x f]{Gxf}{\lsub\gamma\sbtl\vG x f} for
\[
\sgen[\big]{
	\bigcup_{0 \le i \le r} \sbtl{\CC\vG i(\gamma)}x{\vec a\supn_i} \cup
	\bigcup_{i < 0}
		(\sbtl{\CC{N^- \cap \vG}i(\gamma)}x{\vec a\supd_i} \cup
		\sbtl{\CC{N^+ \cap \vG}i(\gamma)}x{\vec a\supc_i})
}.
\]
The element \(\gamma\) will usually be clear, and so will be omitted from the notation.
\end{defn}

\begin{rem}
\label{rem:vGvr}
Note that depth \(-\infty\) means that we include everything (resulting in a non-compact subgroup), and depth \(\infty\) means that we include nothing (resulting in a non-open group).  Thus, for example, \(\sbtl{\Lie(G', G)}x{(\infty, -\infty)}\) stands for \(\Lie(G')^\perp\).  Note also that the notation is ``left biased''; so, for example, \(\sbtl{\Lie(H, G', G)}x{(0, \infty, -\infty)}\) stands for \(\sbtl{\Lie(H)}x 0 + \Lie(G')^\perp\), whereas \(\sbtl{\Lie(G', H, G)}x{(\infty, 0, -\infty)}\) stands for \(\Lie(G')^\perp \cap \bigl(\sbtl{\Lie(H)}x 0 + \Lie(H)^\perp\bigr)\).

We have that
\[
\sbtl{\Lie^*(\vG)}x f
\qeqq
\sett{Y^* \in \Lie^*(G)}{\(\pair{Y^*}Y \in \sbjtlp\field 0\) for all \(Y \in \sbtl{\Lie(\vG)}x{\tilde f}\)}.
\]

We use certain shortcuts for Definition \ref{defn:vGvr}, whose meaning we hope is apparent.  For example, the notation
\(\sbtl{\Lie(\CC G r(\gamma), G)}x{(\Rp0, (r - \ord_{\gamma\pinv})/2)}\)
in Proposition \ref{prop:lattice-orth} stands for \(\sbtl{\Lie(G)}x f\), where
\begin{align*}
f\supn_r             & \qeqq \Rp0                                      \\
f\supn_i             & \qeqq (r - i)/2 && \textq{for} 0 \le i < r, \\
\intertext{and}
f\supd_i  = f\supc_i & \qeqq (r - i)/2 && \textq{for} i < 0;
\end{align*}
and the notation
\(\sbtl{(\CC G r(\gamma), \CC{G'}0(\gamma), \CC G 0(\gamma))}x{(\Rp0, r - \ord_\gamma, (r - \ord_\gamma)/2)}\)
in Proposition \ref{prop:Q-to-B} stands for \(\sbtl{(G', G)}x f\), where
\begin{align*}
f\supn_r            &\qeqq \Rp0                                              \\
f\supn_i            &\qeqq (r - i, (r - i)/2) &&\textq{for} 0 \le i < r, \\
\intertext{and}
f\supd_i = f\supc_i &\qeqq \infty             &&\textq{for} i < 0.
\end{align*}
\end{rem}

Definition \ref{defn:concave-vee} is a modification of \xcite{adler-spice:good-expansions}*{Definition \xref{defn:concave-vee}}, used to describe commutator relationships among groups associated to concave functions.  Our depth matrices are used to handle analogues when some of the groups involved (our \(\CC\bG i(\gamma)\)) do not have full rank, so that it does not make sense to speak of weights in the larger group occurring in the smaller one; but the philosophy is still that the multi-cased definition of \mnotn{f_1 \cvvv f_2} is, in essence, keeping track of the many ways that a sum of weights might be a weight in \(\Lie(\CC{\bG^j}i(\gamma)_\sepfield)\).

\begin{defn}
\label{defn:concave-vee}
Suppose that \(f_1\) and \(f_2\) are depth matrices.  We temporarily put
\begin{align*}
(f_1 {\lsub*\cvvv} f_2)\supn_{i j}
= \inf_{\substack{i_- \le i \le i_+ \\ j_- \le j \le j_+}}
	\min \{
		&(f_1\supn)_{i j} + (f_2\supn)_{i_+j_-}, (f_1\supn)_{i j_-} + (f_2\supn)_{i_+j},
		(f_1\supn)_{i j_+} + (f_2\supn)_{i_+j_+}, \\
		&\qquad(f_1^\pm)_{i_-j} + (f_2^\mp)_{i_-j_-},
		(f_1^\pm)_{i_-j_+} + (f_2^\mp)_{i_-j_+}
	\} \\
\intertext{and}
(f_1 {\lsub*\cvvv} f_2)\supd_{i j}
= \inf_{\substack{i_- \le i \le i_+ \\ j_- \le j \le j_+}}
	\min \{
		&(f_1^-)_{i j} + (f_2^-)_{i_+j_-}, (f_1^-)_{i j_-} + (f_2^-)_{i_+j},
		(f_1^-)_{i j_+} + (f_2^-)_{i_+j_+}, \\
		&\qquad(f_1\supd)_{i_-j} + (f_2\supd)_{i_-j_-},
		(f_1\supd)_{i_-j_+} + (f_2\supd)_{i_-j_+}
	\},
\end{align*}
and similarly for \((f_1 {\lsub*\cvvv} f_2)\supc_{i j}\).  With this provisional (asymmetric) definition in place, we make the symmetric definition \(\mnotn{f_1 \cvvv f_2} = \min \sset{f_1 {\lsub*\cvvv} f_2, f_2 {\lsub*\cvvv} f_1}\); and then put \(\mnotn{f_1 \cvev f_2} = \sup \sett F{\(F \le f_1 \cvvv f_2\) and \(F \le f_1 \cvvv F\)}\) and \(\mnotn{f_1 \cvee f_2} = \sup \sett F{\(F \le f_1 \cvev f_2\) and \(F \le F \cvvv f_2\)}\).
\end{defn}

A grouplike depth matrix \(f\) satisfies \(f \le f \cvee f\); and, if \(f_1\) and \(f_2\) are grouplike, then so is \(f_1 \cvee f_2\).

Lemma \ref{lem:filtration} is an adaptation to our situation of \xcite{adler-spice:good-expansions}*{Lemma \xref{lem:master-comm}}.  The proof follows from a routine application of Lemma \ref{lem:der-master-comm} to groups of the form \(\sbtl{\CC\vG i(\gamma)}x{f_i}\).  We omit it, but the interested reader may see a detailed example of the relevant reasoning in the proof of Proposition \ref{prop:Q-and-B}.

\begin{lem}
\initlabel{lem:filtration}
Let
	\begin{itemize}
	\item \vbG be a tame, twisted Levi sequence in \bG,
	\item \(x\) a point of \(\BB(\vG)\),
and	\item \(f_1\) and \(f_2\) depth matrices.
	\end{itemize}
Then the following properties hold.
	\begin{enumerate}
	\item\sublabel{Lie-Lie} If \(Y_j\) belongs to \(\sbtl{\Lie(\vG)}x{f_j}\) for \(j \in \sset{1, 2}\), then \(\comm{Y_1}{Y_2}\) belongs to \(\sbtl{\Lie(\Der\vG)}x{f_1 \cvvv f_2}\).
	\item\sublabel{gp-Lie} If, in addition to \locpref{Lie-Lie}, we have that \(f_1\) is concave and \(g_1\) belongs to \(\sbtl{\vG}x{f_1}\), then \((\Ad(g_1) - 1)Y_2\) belongs to \(\sbtl{\Lie(\Der\vG)}x{f_1 \cvev f_2}\); and analogously on the dual Lie algebra.
	\item\sublabel{gp-gp} If, in addition to \locpref{gp-Lie}, we have that \(f_2\) is concave and \(g_2\) belongs to \(\sbtl\vG x{f_2}\), then \(\comm{g_1}{g_2}\) belongs to \(\sbtl{\Der\vG}x{f_1 \cvee f_2}\).
	\end{enumerate}
\end{lem}

Lemma \ref{lem:filtration} can be applied only when we know the depths of the elements in a commutator.  We would like to apply similar results to commutators involving \(\gamma\), but we don't have any information about its depth.  The idea of Hypothesis \ref{hyp:funny-centraliser} is that \(\gamma\) is supposed informally to ``live at depth \(i\) modulo \(\Zent(\CC\bG i(\gamma))\).''  We make this precise in Hypothesis \ref{hyp:gamma} (and later state an analogous dual-Lie-algebra condition in Hypothesis \ref{hyp:X*}).

Hypothesis \ref{hyp:gamma} involves a point \mnotn x in \(\BB(G)\) (eventually, in \(\BB(H)\)), which we now fix.  We note that the hypothesis may hold for some points of \(\BB(H)\), and not for others.  The set of points for which it holds is analogous to the set \(\BB_r(\gamma)\) of \xcite{adler-spice:good-expansions}*{Definition \xref{defn:Brgamma}}.

Suppose for this paragraph
that \xcite{adler-spice:good-expansions}*{Hypotheses \xref{hyp:reduced}--\xref{hyp:torus-H1-triv}} are satisfied,
and \(\gamma\) is a compact-modulo-centre element of a tame torus satisfying \xcite{adler-spice:good-expansions}*{Definition \xref{defn:S-is-good}}.  Then Hypothesis \initref{hyp:gamma}(\subref{Lie}, \subref{gp}) is an easy generalisation of \xcite{adler-spice:good-expansions}*{Lemmas \xref{lem:center-comm} and \xref{lem:iso-quotients}} (using Lemma \ref{lem:der-master-comm} in place of \xcite{adler-spice:good-expansions}*{Lemma \xref{lem:master-comm}} to make sure that commutators land in the derived group).  Hypothesis \initref{hyp:gamma}\subpref{orbit} follows from \xcite{adler-spice:good-expansions}*{Lemma \xref{lem:GE2}, Corollary \xref{cor:compatibly-filtered-tame-rank} (and Definition \xref{defn:funny-centralizer}), and Lemma \xref{lem:simult-approx}}.

\begin{hyp}
\initlabel{hyp:gamma}
Let \vbG be a tame, twisted Levi sequence in \bG containing \(\gamma\),
such that \(x\) belongs to \(\BB(\vG)\), and let \(\vec a\) be a depth vector.
The following hold for any \(i \in \tR \cup \sset{-\infty}\) with \(i \le r\).
	\begin{enumerate}
	\item\sublabel{building} The point \(x\) belongs to \(\BB(\CC G r(\gamma))\).
	\item\sublabel{Lie} The map \(\Ad(\gamma) - 1\) carries \(\sbtl{\Lie(\CC\vG i(\gamma\pinv))}x{\vec a}\) into \(\sbtl{\Lie(\CC{N^+ \cap \vG}i(\gamma\inv), \CC{\Der G \cap \vG}i(\gamma\pinv))}x{(\vec a, \vec a + i)}\), and induces an isomorphism
\begin{align*}
	&\sbat{\Lie(\CC\vG i(\gamma\pinv))}x{\vec a}/\sbat{\Lie(\CCp\vG i(\gamma\pinv))}x{\vec a} \\ \mapisoarrow{}
	&\sbat{\Lie(\CC{N^+ \cap \vG}i(\gamma\inv), \CC\vG i(\gamma\pinv))}x{(\vec a, \vec a + i)}
	/ \\
	&\qquad\sbat{\Lie(\CC{N^+ \cap \vG}i(\gamma\inv), \CCp\vG i(\gamma), \CC\vG i(\gamma))}x{(\vec a, \vec a + i, \Rpp{\vec a + i})}.
\end{align*}
The analogous result for \(\gamma\inv\) also holds.
	\item\sublabel{bi-Lie-Lie} Hypothesis \ref{hyp:funny-centraliser} also holds for \(\gamma^2\), and \(\CC\bG r(\gamma^2) \cap \CCp\bG 0(\gamma)\) equals \(\CC\bG r(\gamma)\).  The map \(\Ad(\gamma) - \Ad(\gamma\inv)\) carries \(\sbtl{\Lie(\CC\vG i(\gamma\pinv[2]))}x{\vec a}\) into \(\sbtl{\Lie(\CC{\Der G \cap \vG}i(\gamma\pinv[2]))}x{\vec a + i}\), and induces an isomorphism
\begin{multline*}
	\sbat{\Lie(\CC\vG i(\gamma\pinv[2]))}x{\vec a}/\sbat{\Lie(\CCp\vG i(\gamma\pinv[2]))}x{\vec a} \\ \mapisoarrow{}
	\sbat{\Lie(\CC\vG i(\gamma\pinv[2]))}x{\vec a + i}/\sbat{\Lie(\CCp\vG i(\gamma\pinv[2]))}x{\vec a + i}.
\end{multline*}
	\item\sublabel{gp} If \(i\) is non-negative and \(\vec a\) is grouplike, then the map \(\comm\gamma\anondot\) carries \(\sbtl{\CC\vG i(\gamma)}x{\vec a}\) into \(\sbtl{\CC{\Der G \cap \vG}i(\gamma)}x{\vec a + i}\), and induces 
a bijection
\[
\anoniso
	{\sbat{\CC\vG i(\gamma)}x{\vec a}/\sbat{\CCp\vG i(\gamma)}x{\vec a}}
	{\sbat{\CC\vG i(\gamma)}x{\vec a + i}/\sbat{\CCp\vG i(\gamma)}x{\vec a + i}};
\]
and, if \(i\) is negative and \(\vec a + i\) is grouplike, then the same map carries \(\sbtl{\CC{N^- \cap \vG}i(\gamma)}x{\vec a}\) into \(\sbtl{\CC{N^- \cap \vG}i(\gamma)}x{\vec a + i}\), and induces 
a bijection
\[
\anoniso
{\sbat{\CC{N^- \cap \vG}i(\gamma)}x{\vec a}/\sbat{\CCp{N^- \cap \vG}i(\gamma)}x{\vec a}}
{\sbat{\CC{N^- \cap \vG}i(\gamma)}x{\vec a + i}/\sbat{\CCp{N^- \cap \vG}i(\gamma)}x{\vec a + i}}.
\]
The analogous result for \(\gamma\inv\) also holds.
	\item\sublabel{orbit} If \(g \in \sbtlp{\CC G i(\gamma\pinv)}x 0\) is such that
\[
\Int(g)\bigl(\sbtlp{\CCp G i(\gamma\inv)}x i\dotm\gamma\sbtlp{\CCp G i(\gamma)}x i\bigr)
\cap \sbtlp{\CCp G i(\gamma\inv)}x i\dotm\gamma\sbtlp{\CCp G i(\gamma)}x i
\]
is non-empty, then \(g\) belongs to \(\sbtlp{\CCp G i(\gamma\pinv)}x 0\).
	\end{enumerate}
\end{hyp}

\begin{rem}
\label{rem:gamma}
Hypotheses \ref{hyp:funny-centraliser} and \ref{hyp:gamma} become weaker if we decrease \(r\), so we may cite all results for smaller values of \(r\) if desired.  For example, we do so in the proof of Proposition \ref{prop:dist-r-to-s+}, when we wish to use Lemma \ref{lem:centre} for \(r = 0\).
\end{rem}

\begin{rem}
\label{rem:gamma:Lie*}
Since the filtration on the dual Lie algebra is defined in terms of the filtration on the Lie algebra (Remark \ref{rem:vGvr}), Hypothesis \initref{hyp:gamma}\subpref{Lie} implies the analogous hypothesis on the dual Lie algebra.
\end{rem}


Lemma \ref{lem:commute-gp} states some consequences of Hypothesis \ref{hyp:gamma}, phrased in terms of groups and Lie algebras associated to depth matrices (rather than just depth vectors, as in Hypothesis \ref{hyp:gamma}).  The only place that we need to know that the commutator in Lemma \initref{lem:commute-gp}\subpref{down} belongs to \(\sbtl{\Der\vG}x F\), rather than just \(\sbtl\vG x F\), is in Proposition \ref{prop:Q-to-B}.
Note that Lemma \initref{lem:commute-gp}\subpref{bi-up} involves \(\CC G r(\gamma^2)\), not \(\CC G r(\gamma)\).

As with Lemma \ref{lem:filtration}, the proof of Lemma \ref{lem:commute-gp} follows from a routine application of Hypothesis \ref{hyp:gamma} to groups of the form \(\sbtl{\CC\vG i(\gamma)}x{\vec a}\), together with ``successive approximation arguments'', as, for example, in \cite{adler:thesis}*{Lemma 2.3.2}, to lift results about finite quotients to results about the ambient groups.  We omit it.

\begin{lem}
\initlabel{lem:commute-gp}
Let \vbG be a tame, twisted Levi sequence in \bG containing \(\gamma\), with \(x \in \BB(\vG)\), and let \(f\) and \(F\) be depth matrices.
	\begin{enumerate}
	\item\sublabel{down} If the inequalities \(F^- \le f^- + \min \sset{\ord_\gamma, r}\) and \(F\supc \le f\supc\) hold, then \(\Ad(\gamma) - 1\) carries \(\sbtl{\Lie(\vG)}x f\) into \(\sbtl{\Lie(\Der\vG)}x F\).  If, further, \(f\) and \(F\) are grouplike and satisfy \(F \le f \cvee F\), then \(\comm\gamma\anondot\) carries \(\sbtl{\CC\vG{-\infty}(\gamma)}x f\) into \(\sbtl{\CC{\Der\vG}{-\infty}(\gamma)}x F\).
	\item\sublabel{up} If the inequalities \(F^- \ge f^- + \ord_\gamma\) and \(F\supc \ge f\supc\) hold outside \(\CC\bG r(\gamma)\), then the pre-image of \(\sbtl{\Lie(\vG)}x F\) under \(\Ad(\gamma) - 1\) is contained in \(\sbtl{\Lie(\CC G r(\gamma), \vG)}x{(-\infty, f)}\).  If, further, \(f\) and \(F\) are grouplike and satisfy \(F \le f \cvee F\), then the pre-image in \(\sbtlp M x 0\) of \(\sbtl M x F\) under \(\comm\gamma\anondot\) is contained in \(\sbtl{(\CC G r(\gamma), M \cap \vG)}x{(\Rp0, f)}\).
	\item\sublabel{bi-up} If the inequalities \(F^\mp \ge f^\mp + \ord_{\gamma\pinv}\) hold outside \(\CC\bG r(\gamma^2)\), then the pre-image of \(\sbtl{\Lie(\vG)}x F\) under \(\Ad(\gamma) - \Ad(\gamma)\inv\) is contained in \(\sbtl{\Lie(\CC\vG r(\gamma^2), \vG)}x{(-\infty, f)}\).
	\item\sublabel{onto} If the inequalities \(F^\mp \ge f^\mp + \ord_{\gamma\pinv}\) hold outside \(\CC\bG r(\gamma)\), then \(\Int(\sbtl G x f)(\sbtl{\CC\vG r(\gamma)}x F\dotm h\gamma)\) contains \(\sbtl\vG x F\dotm h\gamma\sbtl\vG x F\) for any \(h \in \sbtl{\CC\vG r(\gamma)}x f\).
	\item\sublabel{orbit} If \(g \in \sbtlp G x 0\) is such that
\[
\Int(g)(\gamma\sbtl{\CC G r(\gamma)}x r) \cap \gamma\sbtl{\CC G r(\gamma)}x r
\]
is non-empty, then \(g\) belongs to \(\sbtlp{\CC G r(\gamma)}x 0\).
	\end{enumerate}
\end{lem}

The result \xcite{debacker-spice:stability}*{Proposition \xref{prop:const}} is stated for a group associated to a very particular concave function, but its proof applies much more generally.  We isolate one general consequence here; the proof is identical.

\begin{lem}[\xcite{debacker-spice:stability}*{Proposition \xref{prop:const}}]
\label{lem:MP-card}
Suppose that \vbG is a tame Levi sequence in \bG containing \(\gamma\), such that \(x\) belongs to \(\BB(\vG)\), and \(f_1\) and \(f_2\) are grouplike depth matrices satisfying \(f_1 \le f_2\) and \((f_1)_{i\,0} = (f_2)_{i\,0}\) for all \(0 \le i \le r\).  Then we have that
\[
\indx{\sbtl\vG x{f_1}}{\sbtl\vG x{f_2}}
\qeqq
\indx{\sbtl{\Lie(\vG)}x{f_1}}{\sbtl{\Lie(\vG)}x{f_2}}.
\]
\end{lem}

The conditions in Lemma \ref{lem:MP-card} require that we quotient out by the ``troublesome'' part of the Moy--Prasad group, which is to say the part corresponding to a maximal torus.  Thus, we may apply Lemma \ref{lem:MP-card} to compute \(\indx{\sbtlp G x 0}{\sbtl{(G', G)}x{(0, \Rp s)}}\) in the proof of Lemma \ref{lem:index-to-disc-X*}, because \(\bG'\) there has full rank; but not to compute \(\indx{\sbtl{(H, M, G)}x{(\Rp0, (r - \ord_\gamma)/2, s)}}{\sbtl{(H, M, G)}x{(\Rp0, s, (r - \ord_{\gamma\pinv})/2)}}\) in Lemma \ref{lem:index-to-disc-gamma}, because \bH need not have full rank.

Lemma \ref{lem:double-coset-count} is used in our explicit ``constant term (about \(\gamma\))'' calculations in Lemma \ref{lem:centre}.

\begin{lem}
\label{lem:double-coset-count}
Suppose that
	\begin{itemize}
	\item \(K^0\) is a compact, open subgroup of \(M\) that is normalised by \(\gamma\),
	\item \(K^\pm\) is a compact, open subgroup of \(N^\pm\) such that \(\Int(\gamma^\pm)K^\pm\) is contained in \(K^\pm\),
and	\item \(K = K^- K^0 K^+\) is a subgroup of \(G\).
	\end{itemize}
Then \(\indx K{K \cap \Int(\gamma)K}\) equals \(\modulus_{P^-}(\gamma)\).
\end{lem}

\begin{proof}
Note that \(K \cap \Int(\gamma)K\) equals \(K^- K^0\dotm\Int(\gamma)K^+\), so that the desired index is \(\indx{K^+}{\Int(\gamma)K^+}\).  Choose \(d \in \R\) so that \(K^+\) contains \(\sbtl{N^+}x d\).  Then we have that \(\indx{K^+}{\Int(\gamma)K^+}\) equals \(\indx{K^+}{\sbtl{N^+}x d}\dotm\indx{\sbtl{N^+}x d}{\Int(\gamma)\sbtl{N^+}x d}\dotm\indx{\Int(\gamma)K^+}{\Int(\gamma)\sbtl{N^+}x d}\inv = \indx{\sbtl{N^+}x d}{\Int(\gamma)\sbtl{N^+}x d}\).
By Lemma \ref{lem:MP-card}, we have that \(\indx{\sbtl{N^+}x d}{\Int(\gamma)\sbtl{N^+}x d}\) equals \(\indx{\sbtl{\Lie(N^+)}x d}{\Ad(\gamma)\sbtl{\Lie(N^+)}x d}\), which equals \({\det}_{\Lie(N^+)}(\Ad(\gamma)\inv) = {\det}_{\Lie(N^-)}(\Ad(\gamma)) = \modulus_{P^-}(\gamma)\).
\end{proof}

Lemma \ref{lem:index-to-disc-gamma} states a useful analogue of \xcite{debacker-spice:stability}*{Proposition \xref{prop:const}}, also proven using Lemma \ref{lem:MP-card}.  It accounts for the appearance of factors involving discriminants in Proposition \ref{prop:Gauss-appears}.  As mentioned after Lemma \ref{lem:MP-card}, the necessity to deal with a quotient of indices, rather than a single index, comes from the fact that \bH need not have full rank.  If it happened that \bH \emph{did} have full rank, then reduction to the Lie algebra would allow us to compute \(\indx{\sbtl{(H, M, G)}x{(\Rp0, (r - \ord_\gamma)/2, s)}}{\sbtl{(H, M, G)}x{(\Rp0, s, (r - \ord_{\gamma\pinv})/2)}}\) itself, giving the expected answer.

\begin{lem}
\label{lem:index-to-disc-gamma}
If \(\bG'\) is a tame, twisted Levi subgroup of \bG containing \(\gamma\), with \(x\) in \(\BB(G')\), and \(r\) is positive, then we have that
\begin{align*}
&\frac
	{\indx{\sbtl{(H, M, G)}x{(\Rp0, (r - \ord_\gamma)/2, s)}}{\sbtl{(H, M, G)}x{(\Rp0, s, (r - \ord_{\gamma\pinv})/2)}}}
	{\indx{\sbtl{(H', M', G')}x{(\Rp0, (r - \ord_\gamma)/2, s)}}{\sbtl{(H', M', G')}x{(\Rp0, s, (r - \ord_{\gamma\pinv})/2)}}}
 \\
\intertext{equals}
&\frac
	{\card{\sbat{(H, G)}x{(\Rp0, (r - \ord_{\gamma\pinv})/2)}}\inv[1/2]}
	{\card{\sbat{(H', G')}x{(\Rp0, (r - \ord_{\gamma\pinv})/2)}}\inv[1/2]}
\dotm
\frac
	{\abs{\Disc_{G/H}(\gamma)}\inv[1/2]\modulus_{P^-}(\gamma)^{1/2}}
	{\abs{\Disc_{G'/H'}(\gamma)}\inv[1/2]\modulus_{P\suppm}(\gamma)^{1/2}}
\times{} \\
&\qquad\frac
	{\indx{\sbat G x s}{\sbat H x s}\inv[1/2]}
	{\indx{\sbat{G'}x s}{\sbat{H'}x s}\inv[1/2]}.
\end{align*}
\end{lem}

\begin{proof}
For the proof, we write \(s_{\gamma\pinv}\) in place of \((r - \ord_{\gamma\pinv})/2\).

Lemma \ref{lem:MP-card} gives that
\begin{align*}
\frac{\indx{\sbat G x s}{\sbat H x s}}{\indx{\sbat{G'}x s}{\sbat{H'}x s}} ={}
&\indx{\sbtl G x s}{\sbtl{(G', H, G)}x{(s, s, \Rp s)}}
\intertext{equals}
&\indx{\sbtl{\Lie(G)}x s}{\sbtl{\Lie(G', H, G)}x{(s, s, \Rp s)}},
\end{align*}
and similarly for
\[
\frac
	{\indx{\sbtl{(H, M, G)}x{(\Rp0, s_\gamma^-, s)}}{\sbtl{(H, M, G)}x{(\Rp0, s, s_{\gamma\pinv}^-)}}}
	{\indx{\sbtl{(H', M', G')}x{(\Rp0, s_\gamma^-, s)}}{\sbtl{(H', M', G')}x{(\Rp0, s, s_{\gamma\pinv}^-)}}}.
\]
Thus we may, and do, work on the Lie algebra.

We have that \(\indx{\sbtl{\Lie(H, M, G)}x{(\Rp0, s_\gamma, s)}}{\sbtl{\Lie(H, M, G)}x{(\Rp0, s, s_{\gamma\pinv})}}\) equals both
\begin{align*}
&\indx{\sbtl{\Lie(H, M, G)}x{(0, s_\gamma, s)}}{\sbtl{\Lie(H, M, G)}x{(0, s, s_{\gamma\pinv})}} \\
\intertext{and}
&\indx{\sbtl{\Lie(H, M, G)}x{(\Rp0, s_\gamma, s)}}{\sbtlp{\Lie(H, M, G)}x{(0, s_\gamma, s)}}\inv\times{} \\
&\qquad\indx{\sbtlp{\Lie(H, M, G)}x{(0, s_\gamma, s)}}{\sbtlp{\Lie(H, M, G)}x{(0, s, s_{\gamma\pinv})}}\times{} \\
&\qquad\indx{\sbtl{\Lie(H, M, G)}x{(\Rp0, s, s_{\gamma\pinv})}}{\sbtlp{\Lie(H, M, G)}x{(0, s, s_{\gamma\pinv})}}\inv \\ ={}
&\card{\sbat{\Lie(H, M, G)}x{(\Rp0, s_\gamma, s)}}\times{} \\
&\qquad\indx{\sbtlp{\Lie(H, M, G)}x{(0, s_\gamma, s)}}{\sbtlp{\Lie(H, M, G)}x{(0, s, s_{\gamma\pinv})}}\times{} \\
&\qquad\card{\sbat{\Lie(H, M, G)}x{(\Rp0, s, s_{\gamma\pinv})}}.
\end{align*}
Since
\begin{align*}
&\card{\sbat{\Lie(H, M, G)}x{(\Rp0, s_\gamma, s)}}
\dotm\card{\sbat{\Lie(H, M, G)}x{(\Rp0, s, s_{\gamma\pinv})}} \\
\intertext{equals}
&\card{\sbat{\Lie(H, M, G)}x{(\Rp0, s_\gamma, s_{\gamma\pinv})}}
\dotm\card{\sbat{\Lie(H, M, G)}x{(\Rp0, s, s)}} \\
&\qquad=
\card{\sbat{\Lie(H, G)}x{(\Rp0, s_{\gamma\pinv})}}
\dotm\card{\sbat{\Lie(H, G)}x{(\Rp0, s)}},
\end{align*}
and similarly for \(G'\), and since
\[
\frac
	{\sbat{\Lie(H, G)}x{(\Rp0, s)}}
	{\sbat{\Lie(H', G')}x{(\Rp0, s)}}
\qeqq
\frac
	{\indx{\sbat{\Lie(G)}x s}{\sbat{\Lie(H)}x s}}
	{\indx{\sbat{\Lie(G')}x s}{\sbat{\Lie(H')}x s}},
\]
we have by \xcite{debacker-spice:stability}*{Corollary \xref{cor:gxf-inv}} that the left-hand side equals
\[
\frac
	{\card{\sbat{\Lie(H, G)}x{(\Rp0, s_{\gamma\pinv})}}\inv[1/2]}
	{\card{\sbat{\Lie(H', G')}x{(\Rp0, s_{\gamma\pinv})}}\inv[1/2]}
\dotm\frac
	{\indx{\sbat{\Lie(G)}x s}{\sbat{\Lie(H)}x s}\inv[1/2]}
	{\indx{\sbat{\Lie(G')}x s}{\sbat{\Lie(H')}x s}\inv[1/2]}
\]
times
\begin{multline*}
\underbrace{\frac
	{\indx{\sbtl{\Lie(H, M)}x{(0, r - \ord_\gamma)}}{\sbtl{\Lie(H, M)}x{(0, r)}}^{1/2}}
	{\indx{\sbtl{\Lie(H', M')}x{(0, r - \ord_\gamma)}}{\sbtl{\Lie(H, M)}x{(0, r)}}^{1/2}}
}_{(\text I)}\times{} \\
\underbrace{\frac
	{\indx{\sbtl{\Lie(N^+)}x{(0, r)}}{\sbtl{\Lie(N^+)}x{(0, r - \ord_{\gamma\inv})}}^{1/2}}
	{\indx{\sbtl{\Lie(N\suppp)}x{(0, r)}}{\sbtl{\Lie(N\suppp)}x{(0, r - \ord_{\gamma\inv})}}^{1/2}}
}_{(\text{II}^+)}
\end{multline*}
times
\[
\underbrace{\frac
	{\indx{\sbtl{\Lie(N^-)}x{(0, r)}}{\sbtl{\Lie(N^-)}x{(0, r - \ord_\gamma)}}^{1/2}}
	{\indx{\sbtl{\Lie(N\suppm)}x{(0, r)}}{\sbtl{\Lie(N\suppm)}x{(0, r - \ord_\gamma)}}^{1/2}}
}_{(\text{II}^-)}.
\]
We have that the numerator of (I) equals \(\abs{\det_{\Lie(M)/\Lie(H)}(\gamma - 1)}\inv[1/2] = \abs{\Disc_{M/H}(\gamma)}\inv[1/2]\), while the numerator of (\(\text{II}^\pm\)) equals \(\abs{\det_{\Lie(N^\pm)}(\gamma)}^{1/2} = \modulus_{P^\pm}(\gamma)^{1/2}\); and similarly for the denominator.  Since \(\abs{\Disc_{M/H}(\gamma)}\modulus_{P^+}(\gamma)\inv\) equals \(\abs{\Disc_{G/H}(\gamma)}\), and similarly for \(G'\), we are done.
\end{proof}

\section{Existence of asymptotic expansions}
\label{sec:qualitative}

\subsection{Good, and nearly good, elements}
\label{sec:nearly-good}

The main result of \S\ref{sec:qualitative}, Theorem \ref{thm:asymptotic-exists}, is the analogue of \cite{jkim-murnaghan:charexp}*{Theorem 5.3.1}.  That result is stated in terms of a so called good minimal K-type, which in turn is defined in terms of an element \(\Gamma\) satisfying properties analogous to the genericity assumptions \cite{yu:supercuspidal}*{\S8, p.~596, \textbf{GE}}.  Our analogue of Kim and Murnaghan's element \(\Gamma\) is denoted by \mnotn{Z^*_o}.  We require that it satisfy similar genericity assumptions, slightly upgraded to handle the possible disconnectedness of \bG.  This section analyses the properties of \(Z^*_o\) and nearby elements.

\begin{hyp}
\initlabel{hyp:Z*}
There is a tame, twisted Levi subgroup \matnotn G{\bG'} of \bG such that the following properties hold.
	\begin{enumerate}
	\item\sublabel{central} The element \(Z^*_o\) is fixed by the coadjoint action of \(\bG'\).
	\item\sublabel{good} The element \(Z^*_o\) satisfies \cite{yu:supercuspidal}*{\S8, p.~596, \textbf{GE1}} (relative to \(\bG'\)).
	\item\sublabel{orbit} If \(g \in G\) is such that
\[
\Ad^*(g)\inv(Z^*_o + \sbjtlpp{\Lie^*(G')}{-r}) \cap (Z^*_o + \sbjtlpp{\Lie^*(G')}{-r})
\]
is non-empty, then \(g\) belongs to \(G'\).
	\end{enumerate}
\end{hyp}

Note that the subgroup \(\bG'\) in Hypothesis \ref{hyp:Z*} is uniquely determined, and that \(Z^*_o\) is good of depth \(-r\), in the sense of \cite{jkim-murnaghan:charexp}*{Definition 2.1.1(2)}.

\begin{rem}
\label{rem:Z*}
The hypotheses on \(Z^*_o\) (relative to \(\bG'\)) imply the analogous hypotheses for \(\Ad^*(g)\inv Z^*_o\) (relative to \(\Int(g)\inv\bG'\)), for any \(g \in G\).
\end{rem}

We now recall the
	\begin{itemize}
	\item non-negative real number \mnotn r,
	\item element \(\mnotn\gamma \in G\), with its associated groups \(\bP^\mp = \CC\bG{-\infty}(\gamma\pinv)\), \(\bN^\mp\), \(\bM = \CC\bG 0(\gamma)\), and \(\matnotn H\bH = \CC\bG r(\gamma)\),
and	\item point \(\mnotn x \in \BB(H)\),
	\end{itemize}
satisfying Hypotheses \ref{hyp:funny-centraliser} and \ref{hyp:gamma}, from \S\ref{sec:depth-matrix}.  These hypotheses say nothing about the relationship between \(\gamma\) and \(Z^*_o\).  Lemmas \ref{lem:near-H'} and \ref{lem:in-H'} discuss conditions under which some such relationship can be deduced.

\begin{lem}
\label{lem:near-H'}
If \(o \in \BB(G')\) and \(X^* \in \sbtl{\Lie^*(H)}x{-r}\) are such that
\begin{align*}
(X^* + \sbtlpp{\Lie^*(G)}x{-r})
\cap{}
&(Z^*_o + \sbtlpp{\Lie^*(G)}o{-r})
\intertext{is non-empty, then \(x\) belongs to \(\BB(G')\) and \(Z^*_o\) to \(\sbtl{\Lie^*(G')}x{-r}\), and there is an element \(k\) of \((\sbtl G x 0 \cap \sbtlp G o 0)\sbtlp G x 0\) such that}
(X^* + \sbtlpp{\Lie^*(G')}x{-r})
\cap
\Ad^*(k)\inv&\bigl(Z^*_o + (\sbtl{\Lie^*(G')}x{-r} \cap \sbtlpp{\Lie^*(G')}o{-r})\bigr)
\end{align*}
is non-empty.
\end{lem}

\begin{proof}
By \cite{jkim-murnaghan:charexp}*{Lemma 2.3.3}, we have that \(x\) belongs to \(\BB(G')\).  Then, by Hypothesis \initref{hyp:Z*}(\subref{central}, \subref{good}), we have that \(Z^*_o\) belongs to \(\sbtl{\Lie^*(G')}x{-r}\).

It follows that
\[
(X^* + \sbtlpp{\Lie^*(G)}x{-r}) \cap \bigl(Z^*_o + (\sbtl{\Lie^*(G)}x{-r} \cap \sbtlpp{\Lie^*(G)}o{-r})\bigr)
\]
is non-empty, so
\[
X^*
\qtextq{belongs to}
Z^*_o + (\sbtl{\Lie^*(G)}x{-r} \cap \sbtlpp{\Lie^*(G)}o{-r}) + \sbtlpp{\Lie^*(G)}x{-r}.
\]
The result now follows from the dual-Lie-algebra analogue of \xcite{adler-spice:good-expansions}*{Lemma \xref{lem:good-comm}}.
\end{proof}

\begin{lem}
\label{lem:in-H'}
If \(\gamma\) centralises \(\bH\conn\) and
\[
\Lie^*(H) \cap (Z^*_o + \sbjtlpp{\Lie^*(G')}{-r})
\]
is non-empty, then \(\gamma\) belongs to \(G'\), and \(Z^*_o\) to \(\Lie^*(H')\).
\end{lem}


\begin{proof}
Any element of \(\Lie^*(H) \cap (Z^*_o + \sbtlpp{\Lie^*(G')}{-r})\) is fixed by \(\Ad^*(\gamma)\), so Hypothesis \initref{hyp:Z*}\subpref{orbit} implies that \(\gamma\) belongs to \(G'\), and then Hypothesis \initref{hyp:Z*}\subpref{central} that \(\gamma\) centralises \(Z^*_o\).  In particular, since the action of \(\gamma\) on \(\Lie(H)^\perp\) is fixed-point-free, we have that \(Z^*_o\) annihilates that space, so belongs to \(\Lie^*(H)\), hence to \(\Lie^*(H')\).
\end{proof}

Until \S\ref{sec:H-perp}, we assume that \(\gamma\) belongs to \(G'\).  (Lemma \ref{lem:in-H'} shows that this is automatic if we require \(\gamma\) to centralise \(\bH\conn\), as we do in Hypothesis \ref{hyp:gamma-central}.)  We use primes to denote the analogues in \(\bG'\) of constructions in \bG, so, for example, \matnotn H{\bH'} stands for \(\CC{\bG'}r(\gamma)\) (subject to the proviso in Remark \ref{rem:tame-Levi}, that we may refer directly only to the identity component of \(\bH'\)).  We remind the reader that, in Proposition \ref{prop:lattice-orth}, the symbol \bH will be used for a different group.

In order to prove the existence of asymptotic expansions in Theorem \ref{thm:asymptotic-exists}, we need to consider, not just good elements such as \(Z^*_o\), but elements of \(\Lie^*(H)\) close to them.  These ``nearly good'' elements enjoy weakened versions of many of the same properties as \(Z^*_o\), which nonetheless suffice for the explicit calculations in \S\ref{sec:quantitative}.  See Remark \ref{rem:X*}.  We isolate these properties in Hypothesis \ref{hyp:X*}, so that they may be used without explicit reference to \(Z^*_o\).  Note, however, that the hypothesis does depend on the element \(\gamma \in G\), the point \(x \in \BB(H)\), and the subgroup \(\bG'\).

Hypotheses \ref{hyp:gamma} and \ref{hyp:X*} bear a close resemblance, the idea being that \(\bG'\) plays the role of \(\CCp\bG{-r}(X^*)\); but they are not literal translations.  For example, the similar-appearing Hypotheses \initref{hyp:gamma}\subpref{orbit} and \initref{hyp:X*}\subpref{orbit} are subtly different.  We genuinely need the extra strength of the latter, for example in Lemma \ref{lem:mu-G-to-G'}, and can get away with assuming it because \(\bG\conn \cap \bG'\) is automatically connected (Remark \ref{rem:tame-Levi}), whereas \(\bG\conn \cap \bH\) need not be.

\begin{hyp}
\initlabel{hyp:X*}
Suppose that \(\gamma\) belongs to \(G'\).  Let \((\bG', \vbG)\) be a tame Levi sequence containing \(X^*\), such that \(x\) belongs to \(\BB(\vG)\), and let \(\vec a\) be a depth vector.
	\begin{enumerate}
	\item\sublabel{building} The point \(x\) belongs to \(\BB(G')\).
	\item\sublabel{depth} The element \(X^*\) belongs to \(\sbtl{\Lie^*(H')}x{-r}\).
	\item\sublabel{Lie} The map \(\ad^*(\anondot)X^*\) carries
\(\sbtl{\Lie(\vG)}x{\vec a}\) into \(\sbtl{\Lie^*(\vG)}x{\vec a - r}\),
and induces an isomorphism
\[
\anoniso
{\sbat{\Lie(\vG)}x{\vec a}/\sbat{\Lie(G' \cap \vG)}x{\vec a}}
{\sbat{\Lie^*(\vG)}x{\vec a}/\sbat{\Lie^*(G' \cap \vG)}x{\vec a - r}}.
\]
	\item\sublabel{gp} If \(\vec a\) is grouplike, then the map \((\Ad^*(\anondot) - 1)X^*\) carries
\(\sbtl{(M \cap \vG)}x{\vec a}\) into \(\sbtl{\Lie^*(M \cap \vG)}x{\vec a - r}\),
and induces 
a bijection
\[
\anoniso
{\sbat{(M \cap \vG)}x{\vec a}/\sbat{(M' \cap \vG)}x{\vec a}}
{\sbat{\Lie^*(M \cap \vG)}x{\vec a - r}/\sbat{\Lie^*(M' \cap \vG)}x{\vec a - r}}.
\]
	\item\sublabel{orbit} If \(g \in G\conn\) is such that
\[
\Ad^*(g)\bigl(X^* + \sbtlpp{\Lie^*(G')}x{-r}\bigr) \cap X^* + \sbtlpp{\Lie^*(G')}x{-r}
\]
is non-empty, then \(g\) belongs to \(G'\).
	\end{enumerate}
\end{hyp}

\begin{lem}
\initlabel{lem:commute-Lie*}
Suppose that \(X^*\) satisfies Hypothesis \ref{hyp:X*} for \bG.  Then it also satisfies the analogous hypothesis for \(\CC\bG i(\gamma\pinv)\) for any \(i \le r\).

Let \vbG be a tame, twisted Levi sequence in \bG containing \(X^*\), with \(x \in \BB(\vG)\), and let \(f\) and \(F\) be depth matrices.
	\begin{enumerate}
	\item\sublabel{down} If the inequality \(F \le f + \min \sset{\ord_{X^*}, \Rpp{-r}}\) holds, then \(\ad^*(\anondot)X^*\) carries \(\sbtl{\Lie(\vG)}x f\) into \(\sbtl{\Lie^*(\vG)}x F\).  If, further, \(f\) is grouplike and satisfies \(F \le f \cvev F\), then \((\Ad^*(\anondot) - 1)X^*\) carries \(\sbtl{(M \cap \vG)}x f\) into \(\sbtl{\Lie^*(\vG)}x F\).
	\item\sublabel{up} If the inequality \(F \ge f - r\) holds outside \(\bG'\), then the pre-image of \(\sbtl{\Lie^*(\vG)}x F\) under \(\ad^*(\anondot)X^*\) is contained in \(\sbtl{\Lie^*(G', \vG)}x{(-\infty, f)}\).  If, further, \(f\) is grouplike and satisfies \(F \le f \cvev F\), then the pre-image in \(\sbtlp M x 0\) of \(\sbtl{\Lie^*(M \cap \vG)}x F\) under \((\Ad^*(\anondot) - 1)X^*\) is contained in \(\sbtl{(M', M \cap \vG)}x{(\Rp0, f)}\).
	\item\sublabel{onto} If the inequality \(F \ge f - r\) holds outside \(\bG'\), and \(f\) is grouplike and satisfies \(F \le f \cvev F\), then \(\Ad^*(\sbtl{(M \cap \vG)}x f)(X^* + \sbtl{\Lie^*(G')}x F)\) contains \(X^* + \sbtl{\Lie^*(M \cap \vG)}x F\).
	\end{enumerate}
\end{lem}

\begin{proof}
As with Lemmas \ref{lem:filtration} and \ref{lem:commute-gp}, it suffices simply to apply Hypothesis \ref{hyp:X*} repeatedly to groups of the form \(\sbtl{\CC\vG i(\gamma)}x{\vec a}\), once we show the claimed heredity.  Hypotheses \initref{hyp:X*}(\subref{building}, \subref{depth}, \subref{orbit}) are automatic, so we need only consider Hypotheses \initref{hyp:X*}(\subref{gp}, \subref{Lie}).  The calculations below are lengthy, but routine.  The idea is that, if we know something about \((\Ad^*(g) - 1)X^*\) relative to \(\gamma\), then we can use it to learn something about \((\Ad^*(\comm g\gamma) - 1)X^*\), so that Hypothesis \initref{hyp:X*}\subpref{gp} (for \bG) gives us information about \(\comm g\gamma\), and then Lemma \initref{lem:commute-gp}\subpref{up} gives us information about \(g\) itself.  Similarly, if we know something about \(\ad^*(Y)X^*\) relative to \(\gamma\), then we can use it to learn something about \(\Ad^*((\Ad(\gamma) - 1)Y)X^*\).

For Hypothesis \initref{hyp:X*}\subpref{Lie}, note that \(\sbat{\Lie(\CC\vG i(\gamma\pinv))}x{\vec a}/\sbat{\Lie(\CC{G' \cap \vG}i(\gamma\pinv))}x{\vec a}\) sits naturally inside \(\sbat\vG x{\vec a}/\sbat{(G' \cap \vG)}x{\vec a}\), and similarly for \(\sbat{\Lie^*(\CC\vG i(\gamma\pinv))}x{\vec a - r}/\sbat{\Lie^*(\CC{G' \cap \vG}i(\gamma\pinv))}x{\vec a - r}\), and that the image of the former is contained in the latter (all by Hypotheses \initref{hyp:funny-centraliser}\subpref{more-vGvr-facts} and \ref{hyp:X*}).  Thus, it suffices to show that, if \(Y \in \sbtl\vG x{\vec a}\) is such that \(\ad^*(Y)X^*\) belongs to \(\sbtl{\Lie^*(\CC\vG i(\gamma))}x{\vec a - r}\), then \(Y\) belongs to \(\sbtl{\Lie(G', \CC\vG i(\gamma), \vG)}x{(\vec a, \vec a, \Rp{\vec a})}\).  In this case, with the clumsy notation \(X^*_{\gamma\inv} = (\Ad^*(\gamma)\inv - 1)X^*\), which belongs to \(\sbtl{\Lie^*(H)}x 0\) (by Hypothesis \initref{hyp:X*}\subpref{depth} and Remark \ref{rem:gamma:Lie*}), and \(X^*_Y = \ad^*(Y)X^*\), which belongs to \(\sbtl{\Lie^*(\CC\vG i(\gamma))}x{\vec a - r}\) (by assumption), we have by Lemmas \initref{lem:filtration}\subpref{Lie-Lie} and \initref{lem:commute-gp}\subpref{down} that
\[
\ad^*((\Ad(\gamma) - 1)Y)X^*
= \bigl((\Ad^*(\gamma) - 1)\ad^*(Y) + \ad^*(Y)\bigr)X^*_{\gamma\inv}
	+ (\Ad^*(\gamma) - 1)X^*_Y
\]
belongs to
\begin{align*}
&\sbtl{\Lie^*(H \cap \vG, N^+ \cap \vG, \vG)}x{(\vec a + r, \vec a, \vec a + \ord_\gamma)}
+ \sbtl{\Lie^*(\vG)}x{\vec a}
+ \sbtl{\Lie^*(\CC\vG i(\gamma))}x{\vec a - r + i} \\ ={}
&\sbtl{\Lie^*(H \cap \vG, M \cap \vG, \vG)}x{(\vec a + r, \vec a, \vec a + \ord_{\gamma\pinv})}
+ \sbtl{\Lie^*(\CC\vG i(\gamma))}x{\vec a - r + i} \\ \subseteq{}
&\sbtl{\Lie^*(M \cap \vG, \vG)}x{(\vec a - r + i, \vec a + \ord_{\gamma\pinv})}.
\end{align*}
Then Hypothesis \initref{hyp:X*}\subpref{Lie} gives that \((\Ad(\gamma) - 1)Y\) belongs to
\[
\sbtl{\Lie(G', M \cap \vG, \vG)}x{(-\infty, \vec a + i, \vec a + r + \ord_{\gamma\pinv})},
\]
and Lemma \initref{lem:commute-gp}\subpref{up} gives that \(Y\) belongs to
\begin{align*}
&\sbtl{\Lie(G', \CCp\vG i(\gamma), \CC\vG i(\gamma), M \cap \vG, N^+ \cap \vG, \vG)}x{(-\infty, -\infty, \vec a, \Rp{\vec a}, \vec a + r + \ord_{\gamma\inv}, \vec a + r)} \\ \subseteq{}
&\sbtl{\Lie(G', \CCp\vG i(\gamma), \CC\vG i(\gamma), N^+, \vG)}x{(-\infty, -\infty, \vec a, -\infty, \Rp{\vec a})}.
\end{align*}
If we write \(Y = Y^- + Y^+\), with \(Y^- \in \sbtl{\Lie(G', \CCp\vG i(\gamma), \CC\vG i(\gamma), \vG)}x{(\vec a, \vec a, \Rp{\vec a})}\) and \(Y^+ \in \Lie(N^+)\), then we see that \(\ad^*(Y)X^* - \ad^*(Y^-)X^*\) belongs to \(\Lie^*(P^-)\) by assumption, but, since it equals \(\ad^*(Y^+)X^*\), also in \(\Lie^*(P^+)\); so that it equals \(0\), and hence (by Hypothesis \initref{hyp:X*}\subpref{Lie} again) that \(Y^+\) belongs to \(\Lie(G')\).  Thus, we have shown that \(Y\) belongs to
\[
\sbtl{\Lie(G', \CCp\vG i(\gamma), \CC\vG i(\gamma), N^+, \vG)}x{(-\infty, -\infty, \vec a, \infty, \Rp{\vec a})}.
\]
Since we already know that \(Y\) belongs to \(\sbtl{\Lie(\vG)}x{\vec a}\), it follows that it belongs to \(\sbtl{\Lie(G', \CC\vG i(\gamma), \vG)}x{(\vec a, \vec a, \Rp{\vec a})}\), as desired.

Similarly, we reduce Hypothesis \initref{hyp:X*}\subpref{gp} to showing that, if \(\vec a\) is grouplike, \(i \in \R_{\ge 0}\) satisfies \(i < r\), and \(g \in \sbtl{\CC\vG i(\gamma)}x{\vec a}\) is such that \((\Ad^*(g) - 1)X^*\) belongs to \(\sbtl{\Lie^*(\CCp\vG i(\gamma))}x{\vec a - r}\), then \(g\) belongs to \(\sbtl{(M', \CCp G i(\gamma), M)}x{(\vec a, \vec a, \Rp{\vec a})}\).

In what follows, we use Lemma \initref{lem:filtration}\subpref{gp-Lie} and Lemma \initref{lem:commute-gp}\subpref{down} repeatedly, without explicit mention.
Note that \(g\) normalises \(\sbtl{\Lie^*(\CC\vG i(\gamma))}x{\vec a - r}\) (since we have the inequality \(\vec a \cvee (\vec a - r) \ge \vec a - r\)), so that \(\Ad^*(\gamma g)\inv X^*\) belongs to
\begin{align*}
&\Ad^*(g)\inv(X^* + \sbtl{\Lie^*(H)}x 0) \\ \subseteq{}
& X^* + \sbtl{\Lie^*(\CCp\vG i(\gamma))}x{\vec a - r} + \sbtl{\Lie^*(H, \CC\vG i(\gamma))}x{(0, \vec a)} \\ ={}
& X^* + \sbtl{\Lie^*(H)}x 0 + \sbtl{\Lie^*(\CCp\vG i(\gamma), \CC\vG i(\gamma))}x{(\vec a - r, \vec a)}.
\end{align*}
Thus \((\Ad^*(\gamma) - 1)\Ad^*(\gamma g)\inv X^*\) belongs to
\begin{multline*}
\sbtl{\Lie^*(H)}x 0 + \sbtl{\Lie^*(H)}x r + \sbtl{\Lie^*(\CCp\vG i(\gamma), \CC\vG i(\gamma))}x{(\Rpp{\vec a - r + i}, \vec a + i)} \\
\subseteq \sbtl{\Lie^*(H)}x 0 + \sbtlp{\Lie^*(\CC\vG i(\gamma))}x{\vec a - r + i},
\end{multline*}
and \((\Ad^*(g) - 1)\Ad^*(\gamma g)\inv X^*\) belongs to
\begin{multline*}
\sbtl{\Lie^*(\CCp\vG i(\gamma))}x{\vec a - r} + \sbtl{\Lie^*(\CC\vG i(\gamma))}x{\vec a} + \sbtl{\Lie^*(\CCp\vG i(\gamma), \CC\vG i(\gamma))}x{(\vec a - r, \vec a)} \\
\subseteq \sbtlp{\Lie^*(\CC\vG i(\gamma))}x{\vec a - r};
\end{multline*}
so
\[
(\Ad^*(\comm g\gamma) - 1)X^* = \comm{\Ad^*(g) - 1}{\Ad^*(\gamma) - 1}\Ad^*(\gamma g)\inv X^*
\]
belongs to
\begin{multline*}
\bigl(\sbtl{\Lie^*(\CC\vG i(\gamma))}x{\vec a} + \sbtlp{\Lie^*(\CC\vG i(\gamma))}x{\vec a - r + i}\bigr)
	+ \sbtlp{\Lie^*(\CC\vG i(\gamma))}x{\vec a - r + i} \\
= \sbtlp{\Lie^*(\CC\vG i(\gamma))}x{\vec a - r + i}.
\end{multline*}
Now we have by Hypothesis \initref{hyp:X*}\subpref{gp} that \(\comm g\gamma\), which certainly belongs to \(\sbtlp M x 0\), in fact belongs to \(\sbtlp{(G', M)}x{(0, \vec a + i)}\); hence by Lemma \initref{lem:commute-gp}\subpref{up} that \(g\) belongs to \(\sbtl{(G', \CCp G i(\gamma), M)}x{(\Rp0, \vec a, \Rp{\vec a})}\).  Since we already know that \(g\) belongs to \(\sbtl{\CC\vG i(\gamma)}x{\vec a}\), it belongs to \(\sbtl{(G', \CCp G i(\gamma), \CC\vG i(\gamma))}x{(\vec a, \vec a, \Rp{\vec a})}\), as desired.
\end{proof}

So far, we have discussed the consequences of Hypothesis \ref{hyp:X*} abstractly, without any reference to the element \(Z^*_o\).  Although we still do not to require that \(X^*\) is close to \(Z^*_o\), we observe in Remark \ref{rem:X*} that choosing such elements is one way to satisfy Hypothesis \ref{hyp:X*}.

\begin{rem}
\label{rem:X*}
Suppose that \(X^*\) is an element of \(\Lie^*(H) \cap (Z^*_o + \sbjtlpp{\Lie^*(G')}{-r})\).  In particular, \(X^*\) belongs to \(\Lie^*(G')\).

Hypothesis \initref{hyp:X*}\subpref{building} follows from \cite{jkim-murnaghan:charexp}*{Lemma 2.3.3}.  Hypothesis \initref{hyp:X*}\subpref{depth} is built into our choice of \(X^*\).  By Lemma \ref{lem:in-H'}, we have that \(Z^*_o\) belongs to \(\sbtl{\Lie^*(G')}x{-r}\), so, since \(X^*\) does as well, we have that \(X^* - Z^*_o\) belongs to \(\sbtl{\Lie^*(G')}x{-r} \cap \sbjtlpp{\Lie^(G')}{-r}\).  Since \((\sbtl{\Lie^*(G')}x{-r} \cap \sbjtlpp{\Lie^*(G')}{-r}) + \sbtlpp{\Lie^*(G')}x{-r}\) is contained in \(\sbtl{\Lie^*(G')}x{-r} \cap \sbjtlpp{\Lie^*(G')}{-r}\), we have that
\[
X^* + \sbtlpp{\Lie^*(G')}x{-r}
\qtextq{is contained in}
Z^*_o + (\sbtl{\Lie^*(G')}x{-r} \cap \sbjtlpp{\Lie^*(G')}{-r}).
\]
In particular, Hypothesis \initref{hyp:X*}(\subref{Lie}, \subref{gp}) is proven as in \cite{jkim-murnaghan:charexp}*{Lemma 2.3.4}, and Hypothesis \initref{hyp:X*}\subpref{orbit} follows from the dual-Lie-algebra analogue of \cite{jkim-murnaghan:charexp}*{Lemma 2.3.6}.
\end{rem}

In Lemma \ref{lem:X*-orbits}, we lay the groundwork for the proof in Lemma \ref{lem:asymptotic-check} that only certain orbits need to be considered when checking the correctness of a possible asymptotic expansion.

\begin{lem}
\label{lem:X*-orbits}
If \(X^*\) belongs to \(\Lie^*(H) \cap (Z^*_o + \sbjtlpp{\Lie^*(G')}{-r})\) and \(g \in G\) is such that
\[
\OO^{H\primeconn}\bigl(\Ad^*(H\conn)(X^* + \sbtlpp{\Lie^*(H)}x{-r})\bigr)
\cap
\OO^{H\primeconn}\bigl(\Ad^*(g)\inv(Z^*_o + \sbjtlpp{\Lie^*(G')}{-r})\bigr)
\]
is non-empty, then \(G'g H\conn\) is the trivial \((G', H\conn)\)-double coset, and the intersection equals \(\OO^{H\primeconn}(X^* + \sbtlpp{\Lie^*(H')}x{-r})\).
\end{lem}

\begin{proof}
Let \(\OO'\) be an element of the intersection.

Remark \ref{rem:X*} gives that Hypothesis \ref{hyp:X*} is satisfied.  By assumption, there is an element \(h \in H\conn\) so that \(\Ad^*(h)(X^* + \sbtlpp{\Lie^*(H)}x{-r})\) intersects \(\OO'\).  By Lemma \initref{lem:commute-Lie*}\subpref{onto}, upon adjusting \(h\) on the right by an element of \(\sbtlp H x 0\), we may, and do, assume that \(\Ad^*(h)(X^* + \sbtlpp{\Lie^*(H')}x{-r})\) intersects \(\OO'\).  Upon adjusting \(h\) on the left by an element of \(H\primeconn\), we may, and do, assume that \(\Ad^*(h)(X^* + \sbtlpp{\Lie^*(H)}x{-r}) \cap \OO'\) intersects \(\Ad^*(g)\inv(Z^*_o + \sbjtlpp{\Lie^*(G')}{-r})\).

Since \(X^* + \sbtlpp{\Lie^*(H')}x{-r}\) is contained in \(Z^*_o + (\sbtl{\Lie^*(G')}x{-r} \cap \sbjtlpp{\Lie^*(G')}{-r})\) (Remark \ref{rem:X*}), we have by Hypothesis \initref{hyp:Z*}\subpref{orbit} that \(g h\) belongs to \(G'\), so that \(G'g H\conn\) is the trivial double coset.

Therefore, upon adjusting \(g\) on the left by an element of \(G'\) (which does not affect the hypothesis, since \(Z^*_o\) is fixed by the coadjoint action of \(\bG'\)), we may, and do, assume that it belongs to \(H\conn\).  Thus \(h' \ldef g h\) belongs to \(G' \cap H\conn\), which equals \(H\primeconn\) by Remark \ref{rem:tame-Levi}.  In particular, \(\Ad^*(h')\inv\OO'\) equals \(\OO'\); so, since \(X^* + \sbtlpp{\Lie^*(H)}x{-r}\) intersects \(\Ad^*(h')\inv\OO' \subseteq \Lie^*(H')\), and since \((X^* + \sbtlpp{\Lie^*(H)}x{-r}) \cap \Lie^*(H')\) equals \(X^* + \sbtlpp{\Lie^*(H')}x{-r}\), we are done.
\end{proof}

\begin{rem}
\label{rem:index-match}
If \(Z^*_o\) belongs to \(\Lie^*(H)\), then, by
Lemma \initref{lem:commute-Lie*}\subpref{onto}, the set \(\Ad^*(H\conn)(Z^*_o + \sbjtlpp{\Lie^*(H')}{-r})\) is a neighbourhood of \(Z^*_o\) in \(\Lie^*(H)\).  Thus each element of \(\OO^{H\conn}(Z^*_o)\) intersects \(Z^*_o + \sbjtlpp{\Lie^*(H')}{-r}\), hence contains some element of \(\OO^{H\primeconn}(Z^*_o)\).  On the other hand, by Hypothesis \initref{hyp:Z*}\subpref{orbit}, distinct elements of \(\OO^{H\primeconn}(Z^*_o)\) have distinct \(H\conn\) orbits.  We frequently use the resulting bijection \anonmap{\OO^{H\conn}(Z^*_o)}{\OO^{H\primeconn}(Z^*_o)} without further explicit mention.
\end{rem}

\subsection{A group analogue of a quadratic form}
\label{sec:H-perp}

In \S\ref{sec:asymptotic}, we express distributions on the group (eventually, characters, in Theorem \ref{thm:asymptotic-exists}) in terms of distributions on the Lie algebra (Fourier transforms of orbital integrals).  This requires some way of passing between the two.  In Hypothesis \ref{hyp:mexp}, we choose a ``mock'' exponential map that we will use to move from a neighbourhood of \(0\) in \(\Lie(H)\) to a neighbourhood of the identity in \(H\), but we still face the problem of extracting information about distributions on \(G\).  The ``perpendicular group'' \(\sbtl{(H, G)}x{(\Rp r, r)}\) in Lemma \ref{lem:H-perp} is an important tool, but a bit of machinery is necessary before we can define it.

We recall the
	\begin{itemize}
	\item non-negative real number \(r\),
	\item element \(\gamma \in G\), with associated groups \(\bM = \CC\bG 0(\gamma)\) and \(\bH = \CC\bG r(\gamma)\),
and	\item point \(x \in \BB(H)\)
	\end{itemize}
from \S\ref{sec:depth-matrix}.  Although we will soon (in Proposition \ref{prop:Q-to-B}) need \(r\) to be positive, there is no harm in avoiding that assumption for a little bit.  We do not need explicitly to mention the element \(Z^*_o\) or the group \(\bG'\) of \S\ref{sec:nearly-good} in this section.

Notation \ref{notn:Q-and-B} defines \(\mf Q_\gamma\) and \(\mf B_\gamma\), which are ``multiplicative analogues'' of a quadratic and a bilinear form, respectively.  Indeed, we have chosen the notation to parallel the quadratic form \(q_{X^*, \gamma}\) and the bilinear form \(b_{X^*, \gamma}\) that are introduced in Notation \ref{notn:Gauss}.  The main use of our multiplicative analogues is in Proposition \ref{prop:Gauss-to-Weil}, where we show that a certain integral appearing in a Frobenius-type formula (see Proposition \ref{prop:Gauss-appears}) may be rewritten as a Weil index.  We have moved these results from \S\ref{sec:Gauss} to \S\ref{sec:H-perp} so that they can be used in Lemma \ref{lem:H-perp}.

\begin{notn}
\label{notn:Q-and-B}
Put
\[
\matnotn{Qgamma}{\mf Q_\gamma}(v) = \comm v\gamma
\qandq
\matnotn{Bgamma}{\mf B_\gamma}(v_1, v_2) = \comm{v_1}{\mf Q_\gamma(v_2)}
\]
for all \(v, v_1, v_2 \in G\).
\end{notn}

As preparation for the bi-multiplicativity results of Proposition \ref{prop:Q-and-B}, we note in Lemma \ref{lem:Hall-Witt} a few basic algebraic identities involving \(\mf Q_\gamma\) and \(\mf B_\gamma\).

\begin{lem}
\initlabel{lem:Hall-Witt}
We have that
\begin{gather}
\sublabel{eq:Q-and-B}
\mf Q_\gamma(v_1 v_2)
\qeqq
\mf B_\gamma(v_1, v_2)\mf Q_\gamma(v_2)\mf Q_\gamma(v_1), \\
\sublabel{eq:B-1-mult}
\mf B_\gamma(v_1 w_1, v_2)
\qeqq
\Int(v_1)\mf B_\gamma(w_1, v_2)\dotm\mf B_\gamma(v_1, v_2),
\end{gather}
\begin{multline}
\sublabel{eq:B-2-mult}
\mf B_\gamma(v_1, v_2 w_2)\mf B_\gamma(v_2, w_2)
\qeqq \\
\Int(v_1)\mf B_\gamma(v_2, w_2)
\dotm\mf B_\gamma(v_1, w_2)
\dotm\Int(\mf Q_\gamma(w_2))\mf B_\gamma(v_1, v_2),
\end{multline}
and
\begin{multline}
\sublabel{eq:B-symmetric}
\mf B_\gamma(\Int(\gamma)v_1,
	v_2)
\qeqq \\
\Int(\mf Q_\gamma(v_1))\inv\mf B_\gamma(\comm{v_1}{v_2}\dotm v_2, v_1)
\dotm\mf B_\gamma(\comm{v_1}{v_2}, v_2)
\Int(\mf Q_\gamma(v_2))\mf Q_\gamma(\comm{v_1}{v_2})
\end{multline}
for all \(v_1, w_1, v_2, w_2 \in G\).
\end{lem}

\begin{proof}
These are all direct computations.  The last is the Hall--Witt identity
\[
\comm[\big]{\comm\gamma{v_1}}{\Int(v_1)v_2}
\dotm\comm[\big]{\comm{v_1}{v_2}}{\Int(v_2)\gamma}
\dotm\comm[\big]{\comm{v_2}\gamma}{\Int(\gamma)v_1}
= 1,
\]
re-written as
\[
\comm[\big]{\Int(\gamma)v_1}{\comm{v_2}\gamma}
= \Int(\comm{v_1}\gamma)\inv\comm[\big]{\Int(v_1)v_2}{\comm{v_1}\gamma}
\dotm\comm[\big]{\comm{v_1}{v_2}}{\comm{v_2}\gamma\dotm\gamma}
\]
and then translated to our notation.
\end{proof}

We think informally of \(\mf Q_\gamma\) and \(\mf B_\gamma\) as being something like a quadratic and a bilinear form, respectively.  We justify the latter claim in Proposition \ref{prop:Q-and-B}, by showing that \(\mf B_\gamma\) is bi-multiplicative, up to an error term, on an appropriate domain; and then the former in Proposition \ref{prop:Q-to-B}.  Note that Proposition \ref{prop:Q-and-B} concerns the values of \(\mf Q_\gamma\) and \(\mf B_\gamma\) only on \(M\), not on all of \(G\).  The restriction comes mainly from the fact that our reasoning requires at least that each \((f_2)_i + i\) is concave, and that is true ``for free'' only if \(i\) is non-negative.

\begin{prop}
\label{prop:Q-and-B}
Suppose that
	\begin{itemize}
	\item \(\vec\bM\) is a tame, twisted Levi sequence in \bM containing \(\gamma\), such that \(x\) is in \(\BB(\vec M)\),
and	\item \(f_1\), \(f_2\), \(g_q\), and \(g_b\) are grouplike depth matrices satisfying
\begin{align*}
g_q & {}\le \min \sset{f_2 + \ord_\gamma, f_2 \cvee g_q}, \\
g_b & {}\le \min \sset{f_1 \cvee (f_2 + \ord_\gamma), g}, \\
\intertext{and}
g   & {}\le \min \sset{f_1 \cvee g_b, g_b \cvee (f_2 + \ord_\gamma), g_b \cvee g_b}.
\end{align*}
	\end{itemize}
Then
	\begin{itemize}
	\item on \(\sbtl{\vec M}x{f_2}\), \(\mf Q_\gamma\) is \(\sbtl{\Der{\vec M}}x{g_q}\)-valued,
and	\item on \(\sbtl{\vec M}x{f_1} \times \sbtl{\vec M}x{f_2}\),
		\begin{itemize}
		\item \(\mf B_\gamma\) is \(\sbtl{\Der{\vec M}}x{g_b}\)-valued
	and	\item \(\mf B_\gamma\) is bi-multiplicative modulo \(\sbtl{\Der{\vec M}}x g\) on the same domain.
		\end{itemize}
	\end{itemize}

Suppose further that \(f_1\) equals \(f_2\), and that the inequality \(g \le (f_1 \cvee f_2) + \ord_\gamma\) holds.  Then
\[
\mf B_\gamma(\Int(\gamma)v_1, v_2)
\qtextq{is congruent to}
\mf B_\gamma(v_2, v_1)
\]
modulo \(\sbtl{\Der{\vec M}}x g\) for all \(v_1, v_2 \in \sbtl{\vec M}x{f_1} = \sbtl{\vec M}x{f_2}\).
\end{prop}

A certain amount of circuitousness in our argument is necessitated by the fact that we have to deal with the vector \(f_2 + \ord_\gamma\) in our bounds, rather than \(g_q\).  The obvious solution seems to be simply to take \(g_q\) to be \(f_2 + \ord_\gamma\), but that need not be grouplike; hence the contortions.

\begin{proof}
The statement about the values of \(\mf Q_\gamma\) is a special case of Lemma \initref{lem:commute-gp}\subpref{down}.

In what follows, we use Lemma \initref{lem:filtration}\subpref{gp-gp} repeatedly, without further mention.

We prove the statement about the values of \(\mf B_\gamma\), and about its bi-multiplicativity, simultaneously.  Write \mc S for the set
\[
\set
	{(v_1, v_2) \in \sbtl{\vec M}x{f_1} \times \sbtl{\vec M}x{f_2}}
	{\mf B_\gamma(v_1, v_2) \in \sbtl{\Der{\vec M}}x{g_b}}.
\]
Since \(g\) is bounded above by \(g_b \cvee g_b\), the values of \(\mf B_\gamma\) on \mc S commute modulo \(\sbtl{\Der{\vec M}}x g\).  By Lemma \initref{lem:Hall-Witt}\subeqref{eq:B-1-mult} and the fact that \(g\) is bounded above by \(f_1 \cvee g_b\), we have that \(\mf B_\gamma\) is multiplicative modulo \(\sbtl{\Der{\vec M}}x g\) in its first variable on \mc S.

We have already shown that, if \(i\) belongs to \(\tR_{\ge 0}\) and \(w_2\) to \(\sbtl{\CC\vG i(\gamma)}x{(f_2)_i}\), then \(\mf Q_\gamma(w_2)\) belongs to \(\sbtl{\CC\vG i(\gamma)}x{(f_2)_i + i}\); so Lemma \initref{lem:Hall-Witt}\subeqref{eq:B-2-mult}, and the fact that \(g\) is bounded above by (\(f_1 \cvee g_b\) and) \(g_b \cvee (f_2 + \ord_\gamma)\), give that, if \((v_1, v_2)\) belongs to \mc S, then \(\mf B_\gamma(v_1, v_2 w_2)\) is congruent to \(\mf B_\gamma(v_1, v_2)\mf B_\gamma(v_1, w_2)\) modulo \(\sbtl{\Der{\vec M}}x g\).  In particular, \((v_1, v_2 w_2)\) belongs to \mc S.

Now suppose, as in the second part of the statement, that \(f_1\) equals \(f_2\), and put \(f = f_1 = f_2\) and \(F = f_1 \cvee f_2\).  Note that \(g\) is bounded above by
\[
f \cvee g_b \le f \cvee g \le F \cvee g
\]
and
\[
g_b \cvee (f + \ord_\gamma) \le (f \cvee (f + \ord_\gamma)) \cvee (f + \ord_\gamma) \le F \cvee (f + \ord_\gamma).
\]
By Lemma \initref{lem:commute-gp}\subpref{down}, we have that \(\gamma\) normalises \(\sbtl{\vec M}x f\).  Note that \(\comm{v_1}{v_2}\) belongs to \(\sbtl{\vec M}x F\).  Since \(g\) is bounded above by \(\min \sset{F + \ord_\gamma, F \cvee g}\), we have already shown that \(\mf Q_\gamma(\comm{v_1}{v_2})\) belongs to \(\sbtl{\Der{\vec M}}x g \subseteq \sbtl{\vec M}x{g_b}\); and, since \(g\) is bounded above by \(F \cvee (f + \ord_\gamma)\), that \(\mf B_\gamma(\comm{v_1}{v_2}, v_1)\) and \(\mf B_\gamma(\comm{v_1}{v_2}, v_2)\) belong to \(\sbtl{\Der{\vec M}}x g \subseteq \sbtl{\vec M}x{g_b}\).  If we make the trivial observations that \(\mf Q_\gamma(v_1)\) and \(\mf Q_\gamma(v_2)\) belong to \(\sbtl{\vec M}x f\), and that the commutator of \(\sbtl{\vec M}x f\) with \(\sbtl{\vec M}x{g_b}\) belongs to \(\sbtl{\Der{\vec M}}x g\) (because \(g\) is bounded above by \(f \cvee g_b\)), then we see that the claimed twisted symmetry follows from Lemma \initref{lem:Hall-Witt}\subeqref{eq:B-symmetric} (and bi-multiplicativity).
\end{proof}

Proposition \ref{prop:Q-to-B} justifies our claim that \(\mf Q_\gamma\) is an analogue of a quadratic form by showing that it is given by using the analogue \(\mf B_\gamma\) of a bilinear form to pair group elements with themselves.  We use it in Proposition \ref{prop:Gauss-to-Weil} to show that Gauss sums occur when computing the values of invariant distributions at certain test functions.  It allows us to avoid the centrality assumption imposed, in an analogous situation, in \xcite{adler-spice:explicit-chars}*{Hypothesis \xref{hyp:X*-central}}.

\begin{prop}
\label{prop:Q-to-B}
Suppose that \(r\) is positive.  For any tame, twisted Levi subgroup \(\bG'\) of \bG containing \(\gamma\), such that \(x\) belongs to \(\BB(G')\), we have that
\[
\mf Q_\gamma(v)^2
\qtextq{is congruent to}
\mf B_\gamma(v, v)
\]
modulo \(\sbtl{\Der(H', \CCp{G'}0(\gamma), \CCp G 0(\gamma))}x{(\Rp r, r, \Rp s)}\) for all \(v \in \sbtl{(H, \CCp{G'}0(\gamma), \CCp G 0(\gamma))}x{(\Rp0, r - \ord_\gamma, (r - \ord_\gamma)/2)}\).
\end{prop}

\begin{proof}
We may, and do, assume, upon passing to a suitable tame extension, that \(\bG\primeconn\) is a Levi subgroup of \(\bG\conn\).

Put
\begin{align*}
\mc V^\perp   & {}= \sbtl{(H, \CCp{G'}0(\gamma), \CCp G 0(\gamma))}x{(\Rp0, r - \ord_\gamma, (r - \ord_\gamma)/2)} \\
\intertext{and}
\mc V^\perp_+ & {}= \sbtl{(H, \CCp{G'}0(\gamma), \CCp G 0(\gamma))}x{(0, r - \ord_\gamma, (r - \ord_\gamma)/2)}.
\end{align*}
Write \(\ol{\mf Q}_\gamma\) and \(\ol{\mf B}_\gamma\) for the compositions of \(\mf Q_\gamma\) and \(\mf B_\gamma\), respectively, with the projection \anonmap G{G/\sbtl{\Der(H', G', G)}x{(\Rp r, r, \Rp s)}}.

By Proposition \ref{prop:Q-and-B}, applied to the depth matrices
\begin{align*}
f = f_1 = f_2 & {}= \max \sset{(\Rp0, \Rp0), (r - \ord_\gamma, (r - \ord_\gamma)/2)}
\intertext{(which give rise to the group \(\sbtl{(H, \CCp{G'}0(\gamma), \CCp G 0(\gamma))}x{(\Rp0, r - \ord_\gamma, (r - \ord_\gamma)/2)} = \mc V^\perp\)),}
g_q = g_b     & {}= (r, \Rp s) \\
\intertext{(which give rise to the group \(\sbtl{(\CCp{G'}0(\gamma), \CCp G 0(\gamma))}x{(r, \Rp s)}\)),} \\
\intertext{and \(g\), where}
g_i           & {}= (r, \Rp s)\qtextq{for} 0 < i < r \\
\intertext{and}
g_r           & {}= \Rp{(r, s)}
\end{align*}
(which gives rise to the group \(\sbtl{(H', \CCp{G'}0(\gamma), \CCp G 0(\gamma))}x{(\Rp r, r, \Rp s)}\)), we have that
	\begin{itemize}
	\item the values of \(\ol{\mf Q}_\gamma\) on \(\mc V^\perp\), and of \(\ol{\mf B}_\gamma\) on \(\mc V^\perp \times \mc V^\perp\), belong to the image of \(\sbtl{(\CCp{G'}0(\gamma), \CCp G 0(\gamma))}x{(r, \Rp s)}\), which is Abelian by Lemma \initref{lem:filtration}\subpref{gp-gp};
	\item \(\ol{\mf B}_\gamma\) is bi-multiplicative on \(\mc V^\perp \times \mc V^\perp\);
	\item \(\Int(\gamma)v_1 = v_1\mf Q_\gamma(v_1\inv)\) is congruent to \(v_1\) modulo \(\sbtl{(\CCp{G'}0(\gamma), \CCp G 0(\gamma))}x{(r, \Rp s)} \subseteq \mc V^\perp_+\) for all \(v_1 \in \mc V^\perp\);
and	\item \(\ol{\mf B}_\gamma(\Int(\gamma)v_1, v_2)\) equals \(\ol{\mf B}_\gamma(v_2, v_1)\) for all \((v_1, v_2) \in \mc V^\perp \times \mc V^\perp\).
	\end{itemize}
In fact, applying the same result, with \((f_1, f_2, g_q, g_b, g)\) replaced by \((\Rp{f_1}, f_2, g_q, g, \Rp g)\) or \((f_1, \Rp{f_2}, g_q, g, g)\), gives that \(\ol{\mf B}_\gamma\) is trivial on \(\mc V^\perp \times \mc V^\perp_+\) and \(\mc V^\perp_+ \times \mc V^\perp\), which means (in light of the above) that actually \(\ol{\mf B}_\gamma\) is \emph{symmetric} on \(\mc V^\perp \times \mc V^\perp\).

In particular, since the value of \(\mf Q_\gamma\) at the identity is \(1\), we have by Lemma \initref{lem:Hall-Witt}\subeqref{eq:Q-and-B} that
\[
\ol{\mf Q}_\gamma(v)\ol{\mf Q}_\gamma(v\inv)
\qeqq
\ol{\mf B}_\gamma(v, v\inv)\inv = \ol{\mf B}_\gamma(v, v)
\]
for all \(v \in \mc V^\perp\).  It thus suffices to show that the set
\[
\mc I \ldef \set{v \in \mc V^\perp}{\ol{\mf Q}_\gamma(v) = \ol{\mf Q}_\gamma(v\inv)}
\]
equals \(\mc V^\perp\).  Note that \mc I is a group, since it clearly contains the identity; and since, for \(v_1, v_2 \in \mc I\), we have by bi-multiplicativity and symmetry, and Lemma \initref{lem:Hall-Witt}\subeqref{eq:Q-and-B}, that \(\ol{\mf Q}_\gamma(v_1 v_2\inv)\) equals
\begin{align*}
&\ol{\mf Q}_\gamma(v_1)\ol{\mf Q}_\gamma(v_2\inv)\ol{\mf B}_\gamma(v_1, v_2\inv) \\ ={}
&\ol{\mf Q}_\gamma(v_1\inv)\ol{\mf Q}_\gamma(v_2)\ol{\mf B}_\gamma(v_1\inv, v_2) \\ ={}
&\ol{\mf Q}_\gamma(v_1\inv)\ol{\mf Q}_\gamma(v_2)\ol{\mf B}_\gamma(v_2, v_1\inv) \\ ={}
&\ol{\mf Q}_\gamma(v_2 v_1\inv),
\end{align*}
hence that \(v_1 v_2\inv\) belongs to \mc I.

Remember that we have passed to a tame extension, and so arranged that \(\bG\primeconn\) is a Levi subgroup of \(\bG\conn\).
By Proposition \ref{prop:Q-and-B}, applied to the restrictions of the depth matrices \((f_1, f_2, g_q, g, g)\) (not \((f_1, f_2, g_q, g_b, g)\)) to \(\bQ\primeconn\), we have that \(\ol{\mf Q}_\gamma\) is trivial on \(\sbtl{(\CC Q r(\gamma), \CCp{G'}0(\gamma), \CCp G 0(\gamma))}x{(\Rp0, r - \ord_\gamma, (r - \ord_\gamma)/2)}\), so that that group is contained in \mc I.  (The point of restricting to a parabolic subgroup is, informally speaking, that, because no two weights in \bQ but outside of \(\bG'\) sum to a weight in \(\bG'\), we can get away with a stricter bound on the values of \(\mf Q_\gamma\).)  Since \(\bQ\conn\) was \emph{any} parabolic subgroup of \(\bG\conn\) with Levi component \(\bG\primeconn\), it follows (say, by taking two opposite such subgroups) that \mc I contains
\(\mc V^\perp\), as desired.
\end{proof}

Our conditions on depth matrices in Definition \ref{defn:vGvr} mean that we have not yet defined notation like \(\sbtl{(H, G)}x{(\Rp r, r)}\), for a ``perpendicular group'' to \(\sbtl H x r\) in \(\sbtl G x r\).  We define such a complement in Lemma \ref{lem:H-perp}.  It is meant primarily to be used in Lemma \ref{lem:centre}, our main tool for moving from distributions on \(G\) to distributions on \(H\), but we also use it to define the notion of a dual blob (Definition \ref{defn:dual-blob}).

\begin{lem}
\label{lem:H-perp}
For any
	\begin{itemize}
	\item tame, twisted Levi sequence \vbG in \bG containing \(\gamma\), such that \(x\) is in \(\BB(\vG)\),
and	\item concave depth vector \(\vec a\),
	\end{itemize}
we have that
\[
\matnotn{DerH}{\sbtl{\Der(H \cap \vG, \vG)}x{(\Rp{\vec a}, \vec a) + r}}
\ldef
\sbtlp{\Der\vG}x{\vec a + r}\dotm
\mf Q_\gamma\bigl(\sbtl{(H \cap \vG, M \cap \vG)}x{(\Rp{\vec a}, \vec a + r - \ord_\gamma)}\bigr)
\]
is a subgroup of \(\sbtl{\Der\vG}x{\vec a + r}\) such that
\begin{align*}
\sbtl{(H \cap \vG)}x{\vec a + r}\dotm\sbtl{\Der(H \cap \vG, \vG)}x{(\Rp{\vec a}, \vec a) + r}
&\qeqq
\sbtl\vG x{\vec a + r} \\
\intertext{and}
\sbtl{(H \cap \vG)}x{\vec a + r} \cap \sbtl{\Der(H \cap \vG, \vG)}x{(\Rp{\vec a}, \vec a) + r}
&\qeqq
\sbtlp{\CC{\Der\vG}r(\gamma)}x{\vec a + r}.
\end{align*}
Analogous properties hold for the group
\[
\matnotn{HvG}{\sbtl{(H \cap \vG, \vG)}x{(\Rp{\vec a}, \vec a) + r}} \ldef \sbtlp\vG x{\vec a + r}\dotm\sbtl{\Der(H \cap \vG, \vG)}x{(\Rp{\vec a}, \vec a) + r}.
\]
\end{lem}

\begin{proof}
For this proof, put \(\mc H = \sbtl{(H \cap \vG, M \cap \vG)}x{(\Rp{\vec a}, \vec a + r - \ord_\gamma)}\) and \(\mc H_+ = \sbtlp{(H \cap \vG, M \cap \vG)}x{(\vec a, \vec a + r - \ord_\gamma)}\).

It follows from Proposition \ref{prop:Q-and-B} that \(\mf Q_\gamma\) is \(\sbtl{\CC{\Der\vG}0(\gamma)}x{\vec a + r}\)-valued, and multiplicative modulo \(\sbtlp{\CC{\Der\vG}0(\gamma)}x{\vec a + r}\), on \mc H; and \(\sbtlp{\CC{\Der\vG}0(\gamma)}x{\vec a + r}\)-valued on \(\mc H_+\).  In particular, \(\sbtl{\Der(H \cap \vG, \vG)}x{(\Rp{\vec a}, \vec a) + r}\) is a group.

By Lemma \initref{lem:commute-gp}\subpref{onto}, for any \(g \in \sbtl{(M \cap \vG)}x{\vec a + r}\), there are \(h \in \sbtl{(H \cap \vG)}x{\vec a + r}\) and \(v \in \mc H\) so that \(\Int(v)(h\gamma)\) equals \(g\gamma\).  Then \(g\) equals \(h\dotm\comm{h\inv}v\dotm\mf Q_\gamma(k)\), which, by Lemma \initref{lem:filtration}\subpref{gp-gp}, belongs to \(h\dotm\sbtl{(H \cap \vG, \vG)}x{(\Rp{\vec a}, \vec a) + r}\).  Since \(\sbtl{(H \cap \vG, \vG)}x{(\Rp{\vec a}, \vec a) + r}\) contains \(\sbtl{N^\pm}x{\vec a + r}\), we have the first equality.

If \(v \in \mc H\) and \(g \in \sbtlp\vG x{\vec a + r}\) are such that \(h \ldef g\dotm\mf Q_\gamma(v)\) belongs to \(\sbtl{(H \cap \vG)}x{\vec a + r}\), then, in particular, \(g\) belongs to \(\sbtl{(M \cap \vG)}x{\vec a + r}\); and \(\Int(v)\gamma\) equals \(g\inv h\gamma\).  By Lemma \initref{lem:commute-gp}\subpref{onto}, there is \(k \in \mc H_+\) so that \(\Int(k)(g\inv h\gamma)\) belongs to \(\gamma\sbtl{(H \cap \vG)}x{\vec a + r}\).  By Lemma \initref{lem:commute-gp}\subpref{orbit}, we have that \(k v \in \mc H\) belongs to \(\sbtlp H x 0\), hence to \(\sbtlp{(H \cap \vG)}x{\vec a} \subseteq \mc H_+\).  Thus, \(v\) itself belongs to \(\mc H_+\), hence \(\mf Q_\gamma(v)\) to \(\sbtlp{\CC{\Der\vG}0(\gamma)}x{\vec a + r}\); and finally \(g\dotm\mf Q_\gamma(v) \in \sbtl{(H \cap \vG)}x{\vec a + r}\) belongs to \(\sbtlp{\CC{\Der\vG}0(\gamma)}x{\vec a + r}\), hence to \(\sbtlp{\CC{\Der\vG}r(\gamma)}x{\vec a + r}\).
\end{proof}

\begin{rem}
\label{rem:H-perp}
The proof of Lemma \ref{lem:H-perp} actually shows that \(\mf Q_\gamma\) induces an isomorphism
\[
\anoniso
	{\sbat{(H \cap \vG, M \cap \vG)}x{(\Rp{\vec a}, \vec a + r - \ord_\gamma)}}
	{\sbat G x{\vec a + r}/\sbat H x{\vec a + r}}.
\]
\end{rem}

\subsection{Dual blobs}
\label{sec:dual-blob}

We recall the
	\begin{itemize}
	\item non-negative real number \(r\),
	\item element \(\gamma \in G\), with associated groups \(\bM = \CC\bG 0(\gamma)\) and \(\bH = \CC\bG r(\gamma)\),
and	\item point \(x \in \BB(H)\)
	\end{itemize}
from \S\ref{sec:depth-matrix}.

We are now ready to set up our tool for moving between \(\Lie(H)\) and \(H\), to be used primarily in Theorem \ref{thm:asymptotic-exists}.  The most natural candidate is the exponential map, but this is only defined if the characteristic of \field is \(0\), and even then may have a very small radius of convergence.  It turns out that all we need is a ``mock'' exponential map satisfying certain (near-)equivariance properties, which we state in Hypothesis \ref{hyp:mexp}.  The hypothesis will not be satisfied unless we assume that \(r\) is positive, so we do so now.

Hypothesis \initref{hyp:mexp}\incpref{coset} is a strengthening of \cite{debacker:homogeneity}*{Hypothesis 3.2.1(1)}.  Hypothesis \initref{hyp:mexp}\incpref{elt}\subpref{cent} is implied by \cite{debacker:homogeneity}*{Hypothesis 3.2.1(2)}, and Hypothesis \initref{hyp:mexp}\incpref{elt}\subpref{disc} is satisfied if the multiset of depths (in the multiplicative filtration) of eigenvalues of \(\mexp(Y)\) coincides with the multiset of depths (in the additive filtration) of eigenvalues of \(Y\).  In particular, the exponential map itself, the Cayley transform, and the ``\(1 +{}\)'' map, as discussed in \cite{debacker:homogeneity}*{Remark 3.2.2}, all satisfy Hypothesis \initref{hyp:mexp}\incpref{elt}\subpref{disc}.

\begin{hyp}
\initlabel{hyp:mexp}
There is a homeomorphism \map{\matnotn{exp}\mexp}{\sbjtl{\Lie(H)}r}{\sbjtl H r}, whose inverse we denote by \matnotn{log}\mlog, satisfying the following properties.
	\begin{enumerate}
	\item\reinclabel{coset} For all tame Levi sequences \(\vec\bH\) in \bH with \(x \in \BB(\vec H)\), we have that
		\begin{enumerate}
		\item\sublabel{coset} \(\mexp(\sbtl{\Lie(\vec H)}x{\vec a})\) equals \(\sbtl{\vec H}x{\vec a}\) for all grouplike depth vectors \(\vec a \ge r\),
and		\item\sublabel{iso} \(\mexp(Y_1)\mexp(Y_2)\) belongs to \(\mexp(Y_1 + Y_2)\sbtl H x{\vec a_1 \cvee \vec a_2}\) for all grouplike depth vectors \(\vec a_1, \vec a_2 \ge r\) and elements \(Y_j \in \sbtl{\Lie(H)}x{\vec a_j}\), for \(j \in \sset{1, 2}\).
		\end{enumerate}
	\item\reinclabel{elt} For all \(Y \in \sbjtl{\Lie(H)}r\), we have that
		\begin{enumerate}
		\item\sublabel{cent} \(\Cent_\bH(\mexp(Y))\conn\) equals \(\Cent_\bH(Y)\conn\)
and		\item\sublabel{disc} \(\abs{\redD_H(Y)}\) equals \(\abs{\redD_H(\mexp(Y))}\).
		\end{enumerate}
	\end{enumerate}
\end{hyp}

\begin{defn}
\label{defn:dual-blob}
By Hypothesis \initref{hyp:mexp}\incpref{coset}(\subref{coset}, \subref{iso}), if
	\begin{itemize}
	\item \(\vbG = (\bG^0, \dotsc, \bG^\ell = \bG)\) is a tame, twisted Levi sequence containing \(\gamma\) such that \(x\) is in \(\BB(\vG)\),
	\item \(\vec a\) is a concave depth vector,
	\item \(\vec b\) is a grouplike depth vector satisfying \(\vec b \le \Rpp{\vec a + r}\) and \((\vec a + r) \cvee \vec b \ge \vec a + r\),
and	\item \(X^*\) belongs to \(\bigcap_{j = 0}^\ell \Lie^*(G^j) \cap \sbtl{\Lie^*(H)}x{\smash{\Tilde{\Vec b}}\vphantom{\vec b}}\),
	\end{itemize}
then \(\AddChar_{X^*} \circ \mlog\) is a character of \(\sbat{(H \cap \vG)}x{\min \sset{\vec a + r, \vec b}}\), which may be extended trivially across \(\sbtl{(H \cap \vG, \vG)}x{(\Rp{\vec a}, \vec a) + r}\dotm\sbtl\vG x{\vec b}\) to a character of \(\sbtl\vG x{\min \sset{\vec a + r, \vec b}}\).  We say that this character has \term{dual blob} \(X^* + \sbtl{\Lie^*(\vG)}x{\min \sset{\vec a + r, \vec b}\sptilde}\).
\end{defn}

\begin{rem}
\label{rem:dual-blob}
With the notation and terminology of Definition \ref{defn:dual-blob}, a character of \(\sbtl{(H \cap \vG)}x{\min \sset{\vec a + r, \vec b}}\) has a dual blob if and only if it is trivial on \(\sbtl{(H \cap \vG)}x{\vec b}\), and a character of \(\sbtl{(H \cap \vG, \vG)}x{\min \sset{\vec a + r, \vec b})}\) has a dual blob if and only if it is trivial on \(\sbtl{(H \cap \vG, \vG)}x{(\Rp{\vec a}, a) + r}\dotm\sbtl\vG x{\vec b}\).  In particular, it is necessary, but need not be sufficient, for it to be trivial on \(\sbtl\vG x{\vec b}\).

We will use Definition \ref{defn:dual-blob} to manufacture, for \(a \ge r\), characters of \(\sbtl G x a\) from 
elements of \(\sbtl{\Lie^*(H)}x{-a}\) (using \(\vec a = a - r\) and \(\vec b = \Rp a\)); and also characters of \(\sbtl{(G', G)}x{(r, \Rp s)}\), where \(\bG'\) is a tame, twisted Levi subgroup of \bG, containing \(\gamma\), such that \(x\) is in \(\BB(G')\), from 
elements of \(\sbtl{\Lie^*(H', H)}x{(-r, -s)}\) (using \(\vec a = (0, \Rp0)\) and \(\vec b = \Rp{(r, s)}\)).
\end{rem}

Hypothesis \ref{hyp:depth} is the analogue of \cite{adler-korman:loc-char-exp}*{Hypothesis 8.3}.  It says roughly that depths, of elements and characters, are the same whether measured in \(H\) or \(G\).  This is known, at least for elements, when \bH is a Levi subgroup of \bG by \cite{adler-debacker:bt-lie}*{Lemma 3.7.25}, and hence by tame descent even when \bH is just a tame, twisted Levi subgroup.

\begin{hyp}
\initlabel{hyp:depth}
For all
	\begin{itemize}
	\item \(a \in \tR\) with \(r \le a < \infty\),
	\item \(x \in \BB(H)\),
and	\item \(X^* \in \sbtl{\Lie^*(H)}x a\),
	\end{itemize}
if \bQ is a parabolic subgroup of \bG with unipotent radical \bU and Levi subgroup \bL satisfying \(x \in \BB(L)\), and
	\begin{enumerate}
	\item\sublabel{coset} the coset \(X^* + \sbtlpp{\Lie^*(G)}x{-a}\) contains an element of the annihilator of \(\Lie(Q)\)
or	\item\sublabel{char} the character 
of \(\sbtl G x a\) with dual blob \(X^* + \sbtlpp{\Lie^*(G)}x{-a}\) agrees with 
the trivial character of \(\sbtl Q x a\) on the intersection of their domains,
	\end{enumerate}
then the coset \(X^* + \sbtlpp{\Lie^*(H)}x{-a}\) is degenerate.
\end{hyp}

\begin{rem}
\label{rem:depth}
With the notation of Hypothesis \ref{hyp:depth}, the conditions on \(X^*\) are satisfied if there is some point \(o \in \BB(G)\) such that \(X^* + \sbtlpp{\Lie^*(G)}x{-a}\) intersects \(\sbtlpp{\Lie^*(G)}o{-a}\), or 
the character of \(\sbtl G x a\) with dual blob \(X^* + \sbtlpp{\Lie^*(G)}x{-a}\) agrees with the trivial character of \(\sbtl G o a\) on the intersection of their domains.  Indeed, in these cases, we need only choose a maximal split torus whose apartment contains both \(o\) and \(x\), and then take for \bQ the parabolic subgroup containing that torus that is dilated by \(o - x\) \cite{springer:lag}*{\S13.4.1}.
\end{rem}

\subsection{Asymptotic expansions and K-types}
\label{sec:asymptotic}

The main result of this section, Theorem \ref{thm:asymptotic-exists}, shows that certain characters have asymptotic expansions around (nearly) arbitrary semisimple elements in terms of Fourier transforms of orbital integrals, as in \cite{jkim-murnaghan:charexp}*{Theorem 5.3.1}.  This relies on Harish-Chandra's ``semisimple descent'' \cite{hc:queens}*{\S\S18--20}, as discussed in Lemma \ref{lem:centre}.

We recall the
	\begin{itemize}
	\item non-negative real number \(r\),
	\item element \(\gamma \in G\), with associated groups \(\bP^\mp = \CC\bG{-\infty}(\gamma\pinv)\), \(\bN^\mp\), and \(\bH = \CC\bG r(\gamma)\),
and	\item point \(x \in \BB(H)\)
	\end{itemize}
from \S\ref{sec:depth-matrix}.  We should regard \(x\) as fixed only provisionally, until Corollary \ref{cor:sample}; we need to allow it to vary in the proof of Theorem \ref{thm:asymptotic-exists}.  We will also introduce more notation before Definition \ref{defn:centred-char}.

Lemma \ref{lem:fa-explicit} is a technical computation to prepare for this descent.  It is convenient (particularly in Proposition \ref{prop:dist-r-to-s+}) to state and prove Lemmas \ref{lem:fa-explicit} and \ref{lem:centre} uniformly for \(r = 0\) and for \(r\) positive.

\begin{lem}
\initlabel{lem:fa-explicit}
For any
	\begin{itemize}
	\item tame, twisted Levi sequence \vbG in \bG containing \(\gamma\), such that \(x\) is in \(\BB(\vG)\),
	\item concave depth vector \(\vec a\),
and	\item \(h \in \sbtl H x r\),
	\end{itemize}
we have that
\begin{align*}
\uint_{\sbtl\vG x{\vec a + r}\dotm h\gamma\sbtl\vG x{\vec a + r}} &F(g)\upd g \\
\intertext{equals}
\uint_{\sbtl{(H \cap \vG, \vG)}x{(\Rp{\vec a}, \vec a + r - \ord_{\gamma\pinv})}} \uint_{\sbtl{(H \cap \vG)}x{\vec a + r}\dotm h\gamma} &F(\Int(g)m)\upd m\,\upd g
\end{align*}
for all \(F \in \Hecke(G)\).
\end{lem}

\begin{proof}
For any 
concave depth vector \(\vec b\), put \(\mc H_{\vec b} = \sbtl{(H \cap \vG, \vG)}x{(\Rp{\vec b}, \vec b + r - \ord_{\gamma\pinv})}\).  Note that Lemmas \initref{lem:filtration}\subpref{gp-gp} and \initref{lem:commute-gp}\subpref{down} give that \(\Int(\mc H_{\vec b})(\sbtl{(H \cap \vG)}x{\vec b + r}\dotm h\gamma)\) is contained in the \(\sbtl{\vec G}x{\vec b + r}\)-double coset to which \(h\gamma\) belongs.

It suffices to prove, for each
concave depth vector \(\vec b\) satisfying \(\vec a \le \vec b \le \vec a \cvee \vec b\),
the desired equality in case \(F\) is the characteristic function of a \(\sbtl\vG x{\vec b + r}\)-double coset contained in the \(\sbtl\vG x{\vec a + r}\)-double coset to which \(h\gamma\) belongs.  By Lemma \initref{lem:commute-gp}\subpref{onto}, such a double coset is of the form \(\sbtl\vG x{\vec b + r}\Int(k_{\vec a})\inv(h_{\vec a}h\gamma)\sbtl\vG x{\vec b + r} = \Int(k_{\vec a})\inv(\sbtl\vG x{\vec b + r}\dotm h_{\vec a}h\gamma\sbtl\vG x{\vec b + r})\) for some \(k_{\vec a} \in \mc H_{\vec a}\) and \(h_{\vec a} \in \sbtl{(H \cap \vG)}x{\vec a + r}\).  Upon replacing \(F\) by \(k_{\vec a}F k_{\vec a}\inv\), which does not change either side of the desired equality, we may, and do, assume that \(k_{\vec a}\) is the identity.  Put
\[
U = \sbtl\vG x{\vec b + r}\dotm h_{\vec a}h\gamma\sbtl\vG x{\vec b + r} = \sbtl{(N^+ \cap \vG)}x{\vec b + r}\dotm h_{\vec a}h\gamma\sbtl\vG x{\vec b + r}.
\]
Note that \(\sbtl{(H \cap \vG)}x{\vec a + r}\dotm h\gamma \cap U\) equals \(\sbtl{(H \cap \vG)}x{\vec b + r}\dotm h_{\vec a}h\gamma\).  If \(g \in \mc H_{\vec a}\) is such that \(\Int(g)\inv U = \Int(g)\inv(\sbtl\vG x{\vec b + r}\dotm h_{\vec a}h\gamma\sbtl\vG x{\vec b + r})\) intersects \(\sbtl{(H \cap \vG)}x{\vec a + r}\dotm h\gamma\), then, by another application of Lemma \initref{lem:commute-gp}\subpref{onto}, there is \(k_{\vec b} \in \mc H_{\vec b}\) so that \(\Int(k_{\vec b}g)\inv(\sbtl{(H \cap \vG)}x{\vec b + r}\dotm h_{\vec a}h\gamma)\) intersects \(\sbtl{(H \cap \vG)}x{\vec a + r}\dotm h\gamma\).  By Lemma \initref{lem:commute-gp}\subpref{orbit}, we have that \(k_{\vec b}g\) belongs to \(\sbtlp{(H \cap \vG)}x{\vec a}\), hence \(g\) to \(\mc H_{\vec b}\dotm\sbtlp{(H \cap \vG)}x{\vec a}\).  If \(h_0 \in \sbtlp{(H \cap \vG)}x{\vec a}\) is such that \(g h_0\) belongs to \(\mc H_{\vec b}\), hence stabilises \(U\), then we have that
\begin{multline}
\tag{$*$}
\sublabel{eq:RHS-const}
\sbtl{(H \cap \vG)}x{\vec a + r}\dotm h\gamma \cap \Int(g)\inv U
\\
\qeqq
\Int(h_0)(\sbtl{(H \cap \vG)}x{\vec a + r}\dotm h\gamma \cap U) = \Int(h_0)(\sbtl{(H \cap \vG)}x{\vec b + r}\dotm h_{\vec a}h\gamma).
\end{multline}

By Lemma \ref{lem:double-coset-count}, we have that \(\meas(U)\) equals
\[
\indx{\sbtl\vG x{\vec b + r}}{\sbtl\vG x{\vec b + r} \cap \Int(h_{\vec a}h\gamma)\inv\sbtl\vG x{\vec b + r}}\dotm\meas(\sbtl\vG x{\vec b + r})
= \modulus_{P^-}(\gamma)\meas(\sbtl\vG x{\vec b + r}),
\]
and, similarly, \(\meas(\sbtl\vG x{\vec a + r}\dotm h\gamma\sbtl\vG x{\vec a + r})\) equals \(\modulus_{P^-}(\gamma)\meas(\sbtl\vG x{\vec a + r})\).  Thus, the left-hand side of the desired equality is \(\indx{\sbtl\vG x{\vec a + r}}{\sbtl\vG x{\vec b + r}}\inv\).

Now we consider the right-hand side.  By Lemma \initref{lem:commute-gp}(\subref{down}, \subref{up}), we have that \(\comm\gamma\anondot\) induces 
a bijection
\anoniso{\sbat{(H, P^- \cap \vG)}x{(\Rp0, \vec c + r - \ord_\gamma)}}{\sbat{(H, P^- \cap \vG)}x{(\Rp0, \vec c + r)}}, and \(\comm{\gamma\inv}\anondot\) induces 
a bijection
\anoniso{\sbat{(N^+ \cap \vG)}x{\vec c + r - \ord_{\gamma\inv}}}{\sbat{(N^+ \cap \vG)}x{\vec c + r}}, so that
\[
\card{\sbat{(H, \vG)}x{(\Rp0, \vec c + r - \ord_{\gamma\pinv})}}
\qeqq
\card{\sbat{(H, G)}x{(\Rp0, \vec c + r)}},
\]
for all concave depth vectors \(\vec c\),
and hence that
\begin{align*}
&\frac{\indx{\sbtl\vG x{\vec a + r}}{\sbtl\vG x{\vec b + r}}}{\indx{\sbtl{(H \cap \vG)}x{\vec a + r}}{\sbtl{(H \cap \vG)}x{\vec b + r}}} \\
&\qquad= \indx{\sbtl\vG x{\vec a + r}}{\sbtl{(H \cap \vG, \vG)}x{(\vec a + r, \vec b + r)}}
= \indx{\sbtl{(H, \vG)}x{(\Rp0, \vec a + r)}}{\sbtl{(H, \vG)}x{(\Rp0, \vec b + r)}} \\
\intertext{equals}
&\indx{\sbtl{(H, \vG)}x{(\Rp0, \vec a + r - \ord_{\gamma\pinv})}}{\sbtl{(H, \vG)}x{(\Rp0, \vec b + r - \ord_{\gamma\pinv})}}.
\end{align*}
Combining this with \loceqref{eq:RHS-const} shows that the right-hand side is
\begin{align*}
&\uint_{\mc H_{\vec a}}
	\meas(\sbtl{(H \cap \vG)}x{\vec a + r})\inv\meas(\sbtl{(H \cap \vG)}x{\vec a + r}\dotm h\gamma \cap \Int(g)\inv U)
\upd g \\ ={}
&\indx{\mc H_{\vec a}}{\mc H_{\vec b}\dotm\sbtlp{(H \cap \vG)}x{\vec a}}\inv\times{} \\
&\qquad\indx{\sbtl{(H \cap \vG)}x{\vec a + r}}{\sbtl{(H \cap \vG)}x{\vec b + r}}\inv \\ ={}
&\indx{\sbtl{(H, \vG)}x{(\Rp0, \vec a + r - \ord_{\gamma\pinv})}}{\sbtl{(H, \vG)}x{(\Rp0, \vec b + r - \ord_{\gamma\pinv})}}\inv\times{} \\
&\qquad\indx{\sbtl{(H \cap \vG)}x{\vec a + r}}{\sbtl{(H \cap \vG)}x{\vec b + r}}\inv \\ ={}
&\indx{\sbtl\vG x{\vec a + r}}{\sbtl\vG x{\vec b + r}}\inv,
\end{align*}
so the desired equality follows.
\end{proof}

Our approach to Lemma \ref{lem:centre} requires that \(\gamma\sbjtl H r\) be \(H\conn\)-invariant.  The most natural way that this can happen is if \(\gamma\) centralises \(\bH\conn\).  This sort of hypothesis has arisen before (see Lemma \ref{lem:in-H'}), and will again (though, after this section, not until Theorem \ref{thm:asymptotic-pi-to-pi'}), so we state it formally now.

\begin{hyp}
\label{hyp:gamma-central}
The action of \(\gamma\) centralises \(\bH\conn\).
\end{hyp}

Lemma \ref{lem:centre} is a quantitative version of \cite{adler-korman:loc-char-exp}*{Definition 7.3}, which applies Harish-Chandra's submersion principle to the (submersive, by \cite{rodier:loc-int}*{p.~774, Proposition 1}) map \anonmap{G\conn \times \gamma\sbjtl H r}G given by \anonmapto{(g, m)}{\Int(g)m}.  It is important to note that the distribution \(T_\gamma\) really depends only on \(\gamma\), not on \(x\), \vbG, \(\vec a\), \(h\), or \(g\).  In Proposition \ref{prop:dist-r-to-s+}, we need this result in the case \(r = 0\).

\begin{lem}
\label{lem:centre}
If \(T\) is an invariant distribution on \(G\), then there is an \(H\conn\)-invariant distribution \(T_\gamma\) on \(H\), supported by \(\gamma\sbjtl H r\), such that, with the notation of Lemma \ref{lem:fa-explicit}, we have that
\[
T_\gamma\bigl(\chrc{\sbtlp{(H \cap \vG)}x{\vec a + r}}h\gamma\bigr)
\qeqq
T\bigl(\chrc{\sbtlp\vG x{\vec a + r}}h g\gamma\bigr)
\]
for any \(g \in \sbtl{(H \cap \vG, \vG)}x{(\Rp{\vec a}, \vec a) + r}\).
\end{lem}

\begin{proof}
As in \cite{adler-korman:loc-char-exp}*{Definition 7.3}, given \(f \in \Hecke(\gamma\sbjtl H r)\), we
	\begin{enumerate}
	\item choose a compact, open subgroup \(K\) of \(G\conn\), and let \(\alpha \in \Hecke(G\conn \times \gamma\sbjtl H r)\) be the function \mapto{\chrc K \otimes f}{(g, m)}{\chrc K(g)f(m)};
	\item use \cite{adler-korman:loc-char-exp}*{Theorem 7.1} to find a function \(f_\alpha \in \Hecke(\Int(G\conn)(\gamma\sbjtl H r))\) such that
\[
\int_G f_\alpha(g)F(g)\upd g
\qeqq
\int_G \int_{\gamma\sbjtl H r} \alpha(g, m)F(\Int(g)m)\upd m\,\upd g
\]
for all \(F \in \Hecke(G)\);
and	\item define \(T_\gamma(f) \ldef \modulus_{P^-}(\gamma)\inv T(f_\alpha)\).
	\end{enumerate}
(The modulus function is not necessary for the definition, but it saves us from having to carry around an unpleasant constant  multiple coming from our desire to work with single, rather than double, cosets.  We have otherwise followed the notation of \cite{adler-korman:loc-char-exp}, but point out a slight potential for confusion:  \(\alpha\) depends on \(f\), although it is not explicitly indicated in the notation; and the symbol `\(f\)' in `\(f_\alpha\)' is just punctuation, and is not the same as the \(f\) at which we are evaluating \(T_\gamma\).)

Although the function \(\alpha\), and hence \(f_\alpha\), depends on the choice of \(K\), the number \(T(f_\alpha)\) does not.  Indeed, if \(K'\) is contained in \(K\), and we write \(\alpha'\) for the analogue of \(\alpha\) constructed with respect to \(K'\), then the fact that \(\alpha\) equals \(\displaystyle\uint_K (k \otimes 1)\alpha'\upd k\) means that \(f_\alpha\) equals \(\displaystyle\uint_K k\inv f_{\alpha'}k\,\upd k\), so that, by invariance of \(T\), the equality \(T(f_\alpha) = T(f_{\alpha'})\) holds.

Put \(\mc H_+ = \sbtlp{(H, G)}x{(\vec a, \vec a + r - \ord_{\gamma\pinv})}\).  Once we take into account all the normalisations of measures, Lemma \ref{lem:fa-explicit} says that, if \(f\) is \(\chrc{\sbtlp{(H \cap \vG)}x{\vec a + r}}\dotm h\gamma\) and we choose \(K\) to be \mc H, then \(f_\alpha\) is \(\chrc{\sbtlp\vG x{\vec a + r}\dotm h\gamma\sbtlp\vG x{\vec a + r}}\).  By definition (Lemma \ref{lem:H-perp}), we have that \(g\) is congruent modulo \(\sbtlp\vG x{\vec a + r}\) to an element of the form \(\mf Q_\gamma(v)\) with \(v \in \sbtl{(H \cap \vG, \vG)}x{(\Rp{\vec a}, \vec a + r - \ord_\gamma)}\); so, by Lemma \initref{lem:filtration}\subpref{gp-gp}, we have that
\begin{align*}
&\sbtlp\vG x{\vec a + r}\dotm h g\gamma\sbtlp\vG x{\vec a + r} \\ ={}
&\sbtlp\vG x{\vec a + r}\dotm h\Int(v)\gamma\dotm\sbtlp\vG x{\vec a + r} \\ ={}
&\Int(v)\bigl(
	\sbtlp\vG x{\vec a + r}\dotm\Int(v)\inv h\dotm\gamma\sbtlp\vG x{\vec a + r}
\bigr) \\
\intertext{equals}
&\Int(v)\bigl(
	\sbtlp\vG x{\vec a + r}\dotm h\gamma\sbtlp\vG x{\vec a + r}
\bigr).
\end{align*}
Now invariance of \(T\) gives that
\begin{align*}
&T\bigl(\chrc{\sbtlp\vG x{\vec a + r}\dotm h g\gamma\sbtlp\vG x{\vec a + r}}\bigr) \\
\intertext{equals}
&T(f_\alpha) = \modulus_{P^-}(\gamma)T_\gamma\bigl(\chrc{\sbtlp\vG x{\vec a + r}}h\gamma\bigr).
\end{align*}
On the other hand, Lemma \ref{lem:double-coset-count} gives that
\[
\chrc{\sbtlp\vG x{\vec a + r}\dotm h g\gamma\sbtlp\vG x{\vec a + r}}
\qeqq
\modulus_{P^-}(\gamma)\chrc{\sbtlp\vG x{\vec a + r}}h g\gamma.
\]
The result follows.
\end{proof}

\begin{cor}
\label{cor:centre}
With the notation of Lemma \ref{lem:centre}, for any \(a \in \tR\) with \(r \le a < \infty\) and any character \(\phi\) of \(\sbat H x a\), we have that
\[
T_\gamma\bigl(\gamma h\chrc{\sbtl H x a, \hat\phi^\vee}\bigr)
\qeqq
T\bigl(\gamma h\chrc{\sbtl G x a, \hat\phi^\vee}\bigr)
\]
for all \(h \in \sbtl H x r\), where \(\hat\phi\) is the extension of \(\phi\) trivially across \(\sbtl{(H, G)}x{(\Rp a, a)}\) to \(\sbtl G x a\).
\end{cor}

\begin{proof}
Since \(\gamma\) normalises \(\sbtl H x r\) (by Lemma \initref{lem:commute-gp}\subpref{down}, say), we have by invariance of \(T\) and \(T_\gamma\) that the stated equality is equivalent to
\[
T_\gamma\bigl(\chrc{\sbtl H x a, \hat\phi^\vee}h\gamma\bigr)
\qeqq
T\bigl(\chrc{\sbtl G x a, \hat\phi^\vee}h\gamma\bigr)
\]
for all \(h \in \sbtl H x r\).  It is this that we actually prove.

By Lemma \ref{lem:centre} (and Lemma \ref{lem:H-perp}), and remembering our normalisation convention for \anonchrc, we have that
\begin{align*}
&T\bigl(\chrc{\sbtl G x a, \hat\phi^\vee}h\gamma\bigr) \\
= \card{\sbat G x a}\inv\sum_{g_a \in \sbat{(H, G)}x{(\Rp a, a)}}
	\sum_{h_a \in \sbat H x a} \hat\phi^\vee(h_a)&T\bigl(\chrc{\sbtlp G x a}h_a g_a h\gamma\bigr) \\
\intertext{equals}
\card{\sbat G x a}\inv\card{\sbat{(H, G)}x{(\Rp a, a)}}
	\sum_{h_a \in \sbat H x a} \hat\phi^\vee(h_a)&T_\gamma\bigl(\chrc{\sbtlp H x a}h_a h\gamma\bigr) \\
= \frac{\card{\sbat H x a}\dotm\card{\sbat{(H, G)}x{(\Rp a, a)}}}{\card{\sbat G x a}}&T_\gamma\bigl(\chrc{\sbtl H x a, \hat\phi^\vee}h\gamma\bigr),
\end{align*}
so that the desired equality follows from the fact that \(\card{\sbat H x a}\dotm\card{\sbat{(H, G)}x{(\Rp a, a)}}\) equals \(\card{\sbat G x a}\).
\end{proof}

\begin{rem}
\label{rem:centre}
In \cite{adler-korman:loc-char-exp}*{Theorem 12.1}, and hence in \cite{adler-korman:loc-char-exp}*{Corollary 12.9}, which uses the theorem to deduce the validity of local character expansions around semisimple elements on an explicitly specified domain, Adler and Korman have to impose the bound (in their notation) \(r > \max \sset{\rho(\pi), 2s(\gamma)}\).  The quantitative nature of Corollary \ref{cor:centre} would allow us to improve this bound to \(r > \max \sset{\rho(\pi), s(\gamma)}\).  In our notation, it is what allows us to show that the asymptotic expansion in Theorem \ref{thm:asymptotic-exists} is valid on all of \(\sbjtl H r\).
\end{rem}

We now need to recall the
	\begin{itemize}
	\item element \(Z^*_o \in \Lie^*(G)\)
and	\item tame, twisted Levi subgroup \(\bG'\)
	\end{itemize}
satisfying Hypothesis \ref{hyp:Z*}, from \S\ref{sec:nearly-good}; and the
	\begin{itemize}
	\item mock-exponential map \(\mexp\),
	\end{itemize}
satisfying Hypothesis \ref{hyp:mexp}, from \S\ref{sec:dual-blob}.  As in Definition \ref{defn:dual-blob}, it is the map \(\mexp\) that allows us to speak of the dual blobs of \emph{certain} characters.

In \cite{jkim-murnaghan:charexp}*{Theorem 5.3.1}, Kim and Murnaghan apply their general result \cite{jkim-murnaghan:charexp}*{Theorem 3.1.7} on asymptotic expansions to the special case of characters.  Their hypotheses refine the Moy--Prasad notion of an \textit{unrefined, minimal K-type} \cite{moy-prasad:k-types}*{Definition 5.1} to give what they call a \textit{good, minimal K-type} \cite{jkim-murnaghan:charexp}*{Definition 2.4.6(2)}.  Singling out which characters of a Moy--Prasad group \(\sbtl G o r\) with \(o \in \BB(G)\) are minimal, or good, in this sense requires what is now called a Moy--Prasad isomorphism between \(\sbat{\Lie(G)}o r\) and \(\sbat G o r\), which is used to define the dual blob of such a character.  In the situation of \cite{jkim-murnaghan:charexp}, such a map is already available, and, indeed, may be deduced from the exponential map \cite{jkim-murnaghan:charexp}*{\S1.4, (H\(k\))}.  Since we have insisted on the existence of such a map only on \(H\), not on \(G\), we may not obviously speak of the dual blob of a character of \(\sbtl G o r\) if \(o\) is not a point of \(\BB(H)\).  Fortunately, all that we need to know is the analogue of the notion of K-types being associate \cite{moy-prasad:k-types}*{Definition 5.1}, and this makes sense even if only one of the K-types in question has a well defined notion of a dual blob.

Let \mnotn o be a point of \(\BB(G')\), and \mnotn{\phi_o} a character of \(\sbat{G'}o r\).  Write \(\hat\phi_o\) for the extension of \(\phi_o\) trivially across \(\sbtlp{(G', G)}o{(r, s)}\) to \(\sbtl{(G', G)}o{(r, \Rp s)}\).  Hypothesis \ref{hyp:K-type} is our version of the assertion that \((\sbtl G o r, \hat\phi_o)\) is a good, minimal K-type in the sense of \cite{jkim-murnaghan:charexp}*{Definition 2.4.6(2)}, and even that \((\sbtl{(G', G)}o{(r, \Rp s)}, \hat\phi_o)\) is a slightly refined, minimal K-type in the sense of \cite{adler:thesis}*{\S2.3}.

If we have a family of Moy--Prasad isomorphisms that allows us to speak of dual blobs in general---as, for example, when \bG splits over \tamefield \cite{yu:models}*{\S8.1(iii) and remark after Corollary 5.6}---then these circumlocutions are not necessary, and Hypothesis \ref{hyp:K-type} is automatically satisfied when \((\sbtl{G'}o r, \phi_o)\) is the character with dual blob \(Z^*_o + \sbtlpp{\Lie^*(G')}o{-r}\).  See, for example, the argument in \cite{moy-prasad:k-types}*{\S7.2, Case 1}, and \cite{adler:thesis}*{Lemma 1.8.1}.

\begin{hyp}
\label{hyp:K-type}
If
	\begin{itemize}
	\item \(y\) is an element of \(\BB(H)\),
	\item \(b \in \tR\) satisfies \(2b > r\),
	\item \(X^*\) is an element of \(\sbtl{\Lie^*(H)}y{-r}\),
	\item \((\sbtl{(G', G)}y{(r, b)}, \hat\phi)\) is the character with dual blob \(X^* + \sbtl{\Lie^*(G', G)}y{(r, b)\sptilde}\),
	\item \(g\) is an element of \(G\),
 and	\item \((\Int(g)\inv\sbtl{(G', G)}o{(r, b)}, \hat\phi_o^g)\) agrees with \((\sbtl{(G', G)}y{(r, b)}, \hat\phi)\) on the intersection of their domains,
	\end{itemize}
then \(X^* + \sbtl{\Lie^*(G', G)}y{(r, b)\sptilde}\) intersects \(\Ad^*(g)\inv(Z^*_o + \sbtl{\Lie^*(G', G)}o{(r, b)\sptilde})\).
\end{hyp}

In Definition \ref{defn:centred-char}, we use Lemma \ref{lem:centre} to define certain invariant distributions on \(H\) that carry all the information that we need about characters of \(G\) near \(\gamma\).

\begin{defn}
\label{defn:centred-char}
If \(\pi\) is an admissible representation of \(G\), then we define \matnotn{Theta}{\widecheck\Theta_{\pi, \gamma}} and \matnotn{Theta}{\widecheck\Theta_{\pi, \gamma, Z^*_o}} to be the distributions on \(\Lie^*(H)\) given for any \(f^* \in \Hecke(\Lie^*(H))\) by
\begin{align*}
\widecheck\Theta_{\pi, \gamma}(f^*)        & {}= \Theta_{\pi, \gamma}\bigl(
	\check f^*_\gamma) \\
\intertext{and}
\widecheck\Theta_{\pi, \gamma, Z^*_o}(f^*) & {}= \widecheck\Theta_{\pi, \gamma}(f^*_{Z^*_o}),
\end{align*}
where we have introduced the \textit{ad hoc} notations
	\begin{itemize}
	\item \(\check f^*_\gamma\) for the function that vanishes outside \(\gamma\sbjtl H r\), and is given on that domain by \anonmapto{\gamma\dotm\mexp(Y)}{\check f^*(Y)},
and	\item \(f^*_{Z^*_o}\) for the function that vanishes outside \(\Lie^*(H) \cap \Ad^*(H\conn)(Z^*_o + \sbjtlpp{\Lie^*(G')}{-r})\), and agrees with \(f^*\) on that domain;
	\end{itemize}
and where
	\begin{itemize}
	\item \(\Theta_{\pi, \gamma}\) is the distribution \(T_\gamma\) deduced from \(T = \Theta_\pi\) in Lemma \ref{lem:centre}.
	\end{itemize}
\end{defn}

Kim and Murnaghan \cite{jkim-murnaghan:charexp}*{Theorem 3.1.7} give sufficient ``sampling'' conditions for a distribution to agree, on an appropriate space of test functions, with a combination of orbital integrals.  Lemma \ref{lem:sample} is a technical computation that allows us, in Corollary \ref{cor:sample}, to show that these conditions are satisfied for the characters of an irreducible representation containing \((\sbtl G o r, \hat\phi_o)\), which is our analogue of a good, minimal K-type.

\begin{lem}
\label{lem:sample}
With the notation of Definition \ref{defn:centred-char}, if
	\begin{itemize}
	\item \(\gamma\) centralises \(\bH\conn\),
	\item \(a \in \tR\) satisfies \(r \le a < \infty\),
and	\item \(X^*\) belongs to \(\sbtl{\Lie^*(H)}x{-a}\),
	\end{itemize}
then
\begin{align*}
&\widecheck\Theta_{\pi, \gamma}\bigl(\chrc{X^* + \sbtlpp{\Lie^*(H)}x{-r}}\bigr) \\
\intertext{equals}
\meas(\sbtl H x a)\sum_{Y \in \sbtl\fh x r/\sbtl\fh x a} \AddChar_{X^*}^\vee(Y)&\tr \pi\bigl(\gamma\mexp(Y)\chrc{\sbtl G x a, \hat\phi^\vee}\bigr),
\end{align*}
where \(\hat\phi\) is the character of \(\sbtl G x a\) with dual blob \(X^* + \sbtlpp{\Lie^*(G)}x{-a}\).
\end{lem}

\begin{proof}
Remembering our normalisation convention for \anonchrc, we have that
\begin{align*}
&\chrc{X^* + \sbtlpp{\Lie^*(H)}x{-r}}\spcheck \\
\intertext{equals}
\meas(\sbtl{\Lie(H)}x r)&\chrc{\sbtl{\Lie(H)}x r, \AddChar_{X^*}^\vee} \\
= \meas(\sbtl{\Lie(H)}x a)\sum_{Y \in \sbtl{\Lie(H)}x r/\sbtl{\Lie(H)}x a} \AddChar_{X^*}^\vee(Y)\dotm&\bigl(Y + \chrc{\sbtl{\Lie(H)}x a, \AddChar_{X^*}}\bigr). \\
\intertext{By Hypothesis \initref{hyp:mexp}\subpref{coset}, the composition of this with \(\mlog\) (as a function on \(\sbjtl H r\)) equals}
\meas(\sbtl H x a)\sum_{Y \in \sbtl{\Lie(H)}x r/\sbtl{\Lie(H)}x a} \AddChar_{X^*}^\vee(Y)\dotm&\bigl(\mexp(Y)\chrc{\sbtl H x a, \phi^\vee}\bigr).
\end{align*}
(Note that we have \emph{not} assumed that \(\mexp\) is measure preserving; the point is that
\begin{align*}
\meas(\sbtl{\Lie(H)}x a)&\chrc{\sbtl{\Lie(H)}x a, \AddChar_{X^*}^\vee} \circ \mexp
\intertext{and}
\meas(\sbtl H x a)&\chrc{\sbtl H x a, \phi^\vee}
\end{align*}
both belong to \(\Hecke(\sbat H x a, \phi)\), and both take the value \(1\) at the identity.)  The result now follows from Corollary \ref{cor:centre} (and Definition \ref{defn:centred-char}).
\end{proof}

\begin{cor}
\initlabel{cor:sample}
With the notation and hypotheses of Lemma \ref{lem:sample}, suppose further that \(\pi\) is irreducible and contains \((\sbtl G o r, \hat\phi_o)\).  Then
\[
\widecheck\Theta_{\pi, \gamma}\bigl(\chrc{X^* + \sbtlpp{\Lie^*(H)}x{-r}}\bigr)
\]
vanishes unless \(a\) is greater than \(r\) and \(X^* + \sbtlpp{\Lie^*(H)}x{-a}\) is degenerate, or \(a\) equals \(r\) and there is some \(g \in G\) so that \(X^*\) belongs to \(\Ad^*(g)\inv Z^*_o + \sbjtlpp{\Lie^*(H \cap \Int(g)\inv G')}{-r}\).
\end{cor}

\begin{proof}
Suppose that the quantity does not vanish.  Then Lemma \ref{lem:sample} gives that \(\pi\) contains \((\sbtl G x a, \hat\phi)\), where \(\hat\phi\) is the character of \(\sbtl G x a\) with dual blob \(X^* + \sbtlpp{\Lie^*(G)}x{-r}\).  As in \cite{moy-prasad:k-types}*{\S7.2}, since irreducibility guarantees that the various \(G\)-translates of the (non-\(0\)) \((\sbtl G o r, \hat\phi_o)\)-isotypic subspace of \(\pi\) span the entire space of \(\pi\), we have that there is some \(g \in G\) so that the \((\sbtl G x a, \hat\phi)\)- and \((\Int(g)\inv\sbtl G o r, \hat\phi_o^g)\)-isotypic subspaces intersect non-trivially.  In particular, \((\sbtl G x a, \hat\phi)\) and \((\Int(g)\inv\sbtl G o r, \hat\phi_o^g)\) agree on the intersection of their domains.

First suppose that \(a\) is greater than \(r\).  Then it follows from Remark \ref{rem:depth} that \(X^* + \sbtlpp{\Lie^*(H)}x{-a}\) is degenerate.

Next suppose that \(a\) equals \(r\).  Then, by Hypothesis \ref{hyp:K-type}, we have that
\[
(X^* + \sbtlpp{\Lie^*(G)}x{-r}) \cap \Ad^*(g)\inv(Z^*_o + \sbtlpp{\Lie^*(G)}o{-r})
\]
is non-empty.  By Lemma \ref{lem:near-H'}, we have that \(g\dota x\) belongs to \(\BB(G')\) and \(Z^*\) to \(\sbtl{\Lie^*(G')}{g\dota x}{-r}\); and we may, and do, assume, upon modifying \(g\) on the right by an element of \((\sbtl G x 0 \cap \Int(g)\inv\sbtlp G o 0)\sbtlp G x 0\) (which does not change either of the containments), that
\[
(X^* + \sbtlpp{\Lie^*(H)}x{-r}) \cap \Ad^*(g)\inv(Z^*_o + \sbjtlpp{\Lie^*(G')}{-r})
\]
is non-empty.  For notational convenience, we assume that \(g\) is the identity.

By Lemma \ref{lem:in-H'}, we have that \(\gamma\) belongs to \(G'\) and \(Z^*_o\) to \(\Lie^*(H')\); and the coset
\[
(X^* - Z^*_o) + \sbtlpp{\Lie^*(G)}x{-r}
\]
intersects \(\sbjtlpp{\Lie^*(G')}{-r}\), so Remark \ref{rem:depth} gives that the coset \((X^* - Z^*_o)  + \sbtlpp{\Lie^*(H)}x{-r}\) is degenerate; so, finally, \cite{adler-debacker:bt-lie}*{Corollary 3.5.2} gives that the coset \((X^* - Z^*_o) + \sbtlpp{\Lie^*(H')}x{-r}\) is degenerate.  That is, \(X^* + \sbtlpp{\Lie^*(H')}x{-r}\) intersects \(Z^*_o + \sbjtlpp{\Lie^*(H')}{-r}\).
\end{proof}

Although it is not strictly necessary, we find it convenient to state and prove Lemma \ref{lem:finite}, which guarantees (at least if the number of nilpotent orbits is finite) that the sum in the asymptotic expansion in Theorem \ref{thm:asymptotic-exists} has only finitely many terms.

\begin{lem}
\label{lem:finite}
The set \(\set{g \in G}{\Ad^*(g)\inv Z^*_o \in \Lie^*(H)}\) is a union of finitely many \((G\primeconn, H\conn)\)-double cosets.
\end{lem}

\begin{proof}
By writing the desired set as a union of subsets of \(G\conn\) of the form \(\set{g \in G\conn}{\Ad^*(g_1 g)\inv Z^*_o \in \Lie^*(H)}\), as \(g_1\) ranges over \(G/G\conn\), we see that it suffices to consider such a subset of \(G\conn\).  For notational convenience, we assume that \(g_1\) is the identity.

We first prove this result over the separable closure of \field.  For notational convenience, in the first part of the proof, we replace \field by \sepfield, so that, for example, \(G\conn\) stands for \(\bG\conn(\sepfield)\).  However, we will drop this notational convenience when we prove the result as stated.

Without loss of generality, \(Z^*_o\) belongs to \(\Lie^*(H)\).  By Hypothesis \ref{hyp:gamma-central} and Lemma \ref{lem:in-H'}, we have that \(\gamma\) belongs to \(G'\), hence by Hypothesis \initref{hyp:funny-centraliser}\subpref{Levi} that \(\bH\primeconn\) is a (no longer twisted) Levi subgroup of \(\bH\conn\).  Let \(\bT'\) be a maximal torus in \bH that is contained in \(\bH'\).  It is necessarily split, since we have replaced \field by \sepfield.

Write \mc S for (the analogue over \sepfield of) the set in the statement.  Suppose that \(g \in G\conn\) is such that \(\Ad^*(g)\inv Z^*_o\) belongs to \(\Lie^*(H)\).  As above, we have that \bT is a maximal torus in \bH that is contained in \(\bH \cap \Int(g)\inv\bG'\).
In particular, there is some \(h \in H\conn\) so that \(\Int(h)\inv\bT\) equals \(\bT'\).  Also, \(\Int(g)\bT\) is a maximal (split) torus in \(\bG'\), so that there is an element \(g' \in G\primeconn\) such that \(\Int(g'g)\bT\) equals \(\bT'\).  Thus, \(g'g h\) normalises \(\bT'\).

We have shown that each double coset in \mc S intersects \(\Norm_{G\conn}(\bT')\).  Since the map from \(\Norm_{G\conn}(\bT')\) to the set of double cosets intersecting it factors through the finite group \(\Norm_{G\conn}(\bT')/\bT'\), we have shown the geometric version of the statement.

Now we resume working with our original, discretely valued field \field, and so drop the notational convenience of replacing it by \sepfield.  By what we have already shown, it suffices to show that the intersection with \(G\conn\) of a \((\bG\primeconn(\sepfield), \bH\conn(\sepfield))\)-double coset is a union of finitely many \((G\primeconn, H\conn)\)-double cosets.  Fix \(g \in G\conn\).  The exact sequence of pointed sets
\[
1 \to \bH\conn(\sepfield) \cap \Int(g)\inv\bG\primeconn(\sepfield) \to \bH\conn(\sepfield) \times \bG\primeconn(\sepfield) \to \bG\primeconn(\sepfield)g\bH\conn(\sepfield) \to 1,
\]
where the first arrow is the twisted diagonal embedding \anonmapto h{(h, \Int(g)h)} and the second is the multiplication map \anonmapto{(h, g')}{g'g h}, gives rise to the exact sequence in cohomology
\[
H\conn \times G\primeconn \to \bG\primeconn(\sepfield)g\bH\conn(\sepfield) \cap G \to \mathrm H^1(\sepfield/\field, \bH\conn(\sepfield) \cap \Int(g)\inv\bG\primeconn(\sepfield)).
\]
As observed in \xcite{adler-spice:explicit-chars}*{proof of Lemma \xref{lem:finite-double-coset}}, the Galois cohomology set of any connected, reductive \(p\)-adic group, such as \(\bH\conn \cap \Int(g)\inv\bG\primeconn\) (which equals \(\Cent_\bH(\Int(g)\inv\Zent(\bG\primeconn)\conn)\) by Remark \ref{rem:tame-Levi}) is finite, so the result follows.
\end{proof}


Our main result, Theorem \ref{thm:asymptotic-exists}, is a simultaneous generalisation of Adler--Korman's result \cite{adler-korman:loc-char-exp}*{Corollary 12.9}, which describes the domain of validity of local character expansions near semisimple elements, and Kim--Murnaghan's result \cite{jkim-murnaghan:charexp}*{Theorem 5.3.1}, which shows that a different sort of asymptotic expansion is valid near the identity.  Theorem \ref{thm:asymptotic-exists} describes the domain of validity of a Kim--Murnaghan-style asymptotic expansion about \(\gamma\).

So far throughout the document, we have worked with a \emph{fixed} point \(x \in \BB(H)\), satisfying Hypothesis \ref{hyp:gamma}.  For the proof of Theorem \ref{thm:asymptotic-exists}, we must forget our binding, and regard \(x\) as a free variable.

Because we have been accumulating hypotheses gradually throughout the paper, we re-capitulate all those that are currently in force.  We are imposing
	\begin{itemize}
	\item Hypotheses
		\ref{hyp:funny-centraliser},
		\ref{hyp:gamma} (for \emph{all} points \(x \in \BB(H)\)),
	and	\ref{hyp:gamma-central} (on \(\gamma\)),
	\item Hypotheses
		\ref{hyp:Z*}
	and	\ref{hyp:K-type} (on \(Z^*_o\) and \(\hat\phi_o\)),
	\item Hypothesis \ref{hyp:mexp} (on \(\mexp\)),
and	\item Hypothesis \ref{hyp:depth} (on \bH).
	\end{itemize}
We emphasise that all of these, except possibly Hypothesis \ref{hyp:gamma-central}, are satisfied under suitable tameness hypotheses; for example, if \(\gamma\) is a compact-modulo-centre element of a tame torus satisfying \xcite{adler-spice:good-expansions}*{Definition \xref{defn:S-is-good}}.  Under the same tameness hypotheses, we may arrange Hypothesis \ref{hyp:gamma-central} upon replacing \(\gamma\) by the product \(\gamma_{< r} = \prod_{0 \le i < r} \gamma_i\) of the terms in a normal \(r\)-approximation (with the notation of \xcite{adler-spice:good-expansions}*{\S\xref{sec:normal}} and terminology of \xcite{adler-spice:good-expansions}*{Definition \xref{defn:r-approx}}).  In the notation of \xcite{adler-spice:good-expansions}*{Definition \xref{defn:Brgamma}}, we have that \(\gamma_{< r}\sbjtl H r\) contains \(\gamma\sbtl H x r\) for all \(x \in \BB_r(\gamma)\).

As stated, Theorem \ref{thm:asymptotic-exists} (and Corollary \ref{cor:asymptotic-exists} and Lemma \ref{lem:asymptotic-check}) is contingent on \cite{jkim-murnaghan:charexp}*{Theorem 3.1.7}.  Of course, this theorem is true, but it imposes some stringent hypotheses (see \cite{jkim-murnaghan:charexp}*{\S1.4}).  Rather than repeating those hypotheses, we prefer to emphasise that we are using them just as a black box to obtain the necessary homogeneity result; as long as the orbital integrals ``close to \(Z^*_o\)'' span an appropriate space of distributions, we are fine.

The theorem is vacuous unless \(\pi\) contains our analogue \((\sbtl G o r, \hat\phi_o)\) of a good, minimal K-type.  However, by \cite{jkim-murnaghan:charexp}*{Theorem 2.4.10}, under suitable tameness hypotheses \cite{jkim-murnaghan:charexp}*{\S1.4}, \emph{every} irreducible representation contains such a K-type, so this is not actually too much of a restriction.

In order for the statement of the theorem even to make sense, we require that the relevant orbital integrals converge.  This is true unconditionally in characteristic 0 \cite{ranga-rao:orbital}*{Theorem 3}; and, as remarked in \cite{debacker:homogeneity}*{\S3.4, p.~409}, by the same proof, even in equal characteristic if \(p\) is sufficiently large.

\begin{thm}
\label{thm:asymptotic-exists}
Suppose that \cite{jkim-murnaghan:charexp}*{Theorem 3.1.7(1, 5)} is satisfied, and all of the relevant orbital integrals converge.  For any irreducible representation \(\pi\) containing \((\sbtl G o r, \hat\phi_o)\), there is a finitely supported, \(\OO^{H\conn}(\Ad^*(G)Z^*_o)\)-indexed vector \(b(\pi, \gamma)\) so that
\[
\Phi_\pi(\gamma\dotm\mexp(Y))
= \sum_{\OO \in \OO^{H\conn}(\Ad^*(G)Z^*_o)}
	b_\OO(\pi, \gamma)\Ohat^{H\conn}_\OO(Y)
\]
for all \(Y \in \Lie(H)\rss \cap \sbjtl{\Lie(H)}r\).
\end{thm}

By Hypothesis \initref{hyp:mexp}\incpref{elt}\subpref{cent}, we have that \(\mexp(Y)\) lies in \(H\rss \cap \sbjtl H r\) for all \(Y \in \Lie(H)\rss \cap \sbjtl{\Lie(H)}r\), so that the notation makes sense.

We have used `\(b\)', rather than `\(c\)', for the coefficients in the expansion because we prefer to reserve `\(c\)' for a differently organised form of the coefficients in the same asymptotic expansion; see Theorem \ref{thm:asymptotic-pi-to-pi'}.

\begin{rem}
\label{rem:which-char}
For a given finitely supported \(\OO^{H\conn}(\Ad^*(G)Z^*_o)\)-indexed vector \(b(\pi, \gamma)\) and element \(Y \in \Lie(H)\rss \cap \sbjtl H r\), the equations
\begin{align*}
\Phi_\pi(\gamma\dotm\mexp(Y))
= \sum_{\OO \in \OO^{H\conn}(\Ad^*(G)Z^*_o)}
	&b_\OO(\pi, \gamma)\Ohat^{H\conn}_\OO(Y) \\
\intertext{and}
\abs{\Disc_{G/H}(\gamma)}^{1/2}\Theta_\pi(\gamma\dotm\mexp(Y))
= \sum_{\substack{g \in G'\bslash G/H\conn \\ \Ad^*(g)\inv Z^*_o \in \Lie^*(H)}} &\abs{\redD_H(\Ad^*(g)\inv Z^*_o)}^{1/2}\times{} \\
	&\qquad\sum_{\OO \in \OO^{H\conn}(\Ad^*(g)\inv Z^*_o)} b_\OO(\pi, \gamma)\muhat^{H\conn}_\OO(Y)
\end{align*}
are equivalent by Hypothesis \initref{hyp:mexp}\incpref{elt}\subpref{disc}.  We thus need not distinguish between the \emph{existence} of asymptotic expansions of \(\Phi_\pi\) and of \(\Theta_\pi\) (although, of course, the coefficients will differ).
\end{rem}

\begin{proof}
We use Remark \ref{rem:which-char} to allow us to work with \(\Theta_\pi\) and \(\muhat^{H\conn}_\OO\) rather than \(\Phi_\pi\) and \(\Ohat^{H\conn}_\OO\).  To emphasise that the coefficients in this expansion are different, we denote them by \(\tilde b\) instead of \(b\).

We use the function spaces
\begin{align*}
\mc D^{-r}_{\Rpp{-r}}
& {}= \sum_{x \in \BB(H)} \Cnts(\sbat{\Lie^*(H)}x{-r})          \\
\intertext{and}
\mc D_{\Rpp{-r}}
& {}= \sum_{x \in \BB(H)} \Cc(\Lie^*(H)/\sbtlpp{\Lie^*(H)}x{-r})
\end{align*}
of \cite{jkim-murnaghan:charexp}*{Definition 3.1.1}, adapted from \(\Lie(G)\) to \(\Lie^*(H)\); and also the distribution spaces \(\mc J^{\Ad^*(g)\inv Z^*_o}_{\Rpp{-r}}\) of \cite{jkim-murnaghan:charexp}*{Definition 3.1.2(2)}, similarly adapted.

With the notation of Definition \ref{defn:centred-char}, consider the difference
\[
\widecheck\Theta_{\pi, \gamma}^0
\ldef \widecheck\Theta_{\pi, \gamma}
- \sum_{\substack{g \in G'\bslash G/H\conn \\ \Ad^*(g)\inv Z^*_o \in \Lie^*(H)}}
	\widecheck\Theta_{\pi, \gamma, \Ad^*(g)\inv Z^*_o}.
\]
The indexing set for the sum is finite, by Lemma \ref{lem:finite}, so that the definition makes sense; and the sets \(\Ad^*(H\conn)(\Ad^*(g)\inv Z^*_o + \sbjtlpp{\Lie^*(H \cap \Int(g)\inv G')}{-r})\) corresponding to different double cosets \(G'g H\conn\) are disjoint, by Hypothesis \initref{hyp:Z*}\subpref{orbit}, so we have by Corollary \ref{cor:sample} that \(\widecheck\Theta_{\pi, \gamma}^0\) vanishes on \(\mc D^{-r}_{\Rpp{-r}}\).  Also by Corollary \ref{cor:sample}, each \(\widecheck\Theta_{\pi, \gamma, \Ad^*(g)\inv Z^*_o}\) with \(\Ad^*(g)\inv Z^*_o\) in \(\Lie^*(H)\), as well as \(\widecheck\Theta_{\pi, \gamma}^0\), belongs to \(\mc J^{\Ad^*(g)\inv Z^*_o}_{\Rpp{-r}}\).  In particular, for each \(g \in G\) such that \(\Ad^*(g)\inv Z^*_o\) belongs to \(\Lie^*(H)\), we have by \cite{jkim-murnaghan:charexp}*{Theorem 3.1.7(5)} that there is an \(\OO^{H\conn}(\Ad^*(g)\inv Z^*_o)\)-indexed vector \(\tilde b(\pi, \gamma)\) so that
\[
\widecheck\Theta_{\pi, \gamma, \Ad^*(g)\inv Z^*_o}
\qeqq
\sum_{\OO \in \OO^{H\conn}(\Ad^*(g)\inv Z^*_o)}
	\tilde b_\OO(\pi, \gamma)\mu^H_\OO
\]
on \(\mc D_{\Rpp{-r}}\).  By \cite{jkim-murnaghan:charexp}*{Theorem 3.1.7(1, 2)}, we have that \(\widecheck\Theta_{\pi, \gamma}^0\) vanishes on \(\mc D_{\Rpp{-r}}\), so that
\[
\widecheck\Theta_{\pi, \gamma}
\qeqq
\sum_{g \in G'\bslash G/H\conn} \widecheck\Theta_{\pi, \gamma, \Ad^*(g)\inv Z^*_o}
= \sum_{\OO \in \OO^{H\conn}(\Ad^*(G)Z^*_o)} \tilde b_\OO(\pi, \gamma)\mu^{H\conn}_\OO
\]
on \(\mc D_{\Rpp{-r}}\).  Exactly as in the proof of \cite{debacker:homogeneity}*{Theorem 3.5.2} and \cite{jkim-murnaghan:charexp}*{Theorem 5.3.1}, it follows that the function \(\Theta_\pi(\gamma\dotm\mexp(\anondot))\) representing \anonmapto f{\widecheck\Theta_{\pi, \gamma}(\hat f)} on \(\Lie(H)\rss \cap \sbjtlp{\Lie(H)}r\) equals
\[
\sum_{\OO \in \OO^{H\conn}(\Ad^*(G)Z^*_o)} \tilde b_\OO(\pi, \gamma)\muhat^{H\conn}_\OO.\qedhere
\]
\end{proof}

\begin{cor}
\label{cor:asymptotic-exists}
With the notation and hypotheses of Theorem \ref{thm:asymptotic-exists}, we have that \(\Phi_\pi\) vanishes on \(\gamma\sbjtl H r\) unless \(\Ad^*(G)Z^*_o\) intersects \(\Lie^*(H)\); in particular, unless \(\gamma\) belongs to \(\Int(G)G'\).
\end{cor}

Theorem \ref{thm:asymptotic-exists} only tells us that a vector \(b(\pi, \gamma)\) as in the statement exists; it gives no idea of how to compute it, or even to test the correctness of a candidate vector.  Lemma \ref{lem:asymptotic-check}, which is analogous to \cite{jkim-murnaghan:charexp}*{Corollary 6.1.2}, describes a general approach to verifying the coefficients in an asymptotic expansion.  See \cite{debacker:homogeneity}*{Theorem 2.1.5} for a similar result.  Because we state Lemma \ref{lem:asymptotic-check} only for a distribution already known to have an asymptotic expansions, it does not need the long list of hypotheses of Theorem \ref{thm:asymptotic-exists}.

\begin{lem}
\initlabel{lem:asymptotic-check}
Suppose that
	\begin{itemize}
	\item \(T_\gamma\) is an element of the span of \(\set{\muhat^{H\conn}_\OO}{\OO \in \OO^{H\conn}(\Ad^*(G)Z^*_o)}\),
and	\item \(c(T_\gamma)\) is a finitely supported, \(\OO^{H\conn}(\Ad^*(G)Z^*_o)\)-indexed vector.
	\end{itemize}
If Hypothesis \ref{hyp:Z*} and \cite{jkim-murnaghan:charexp}*{Theorem 3.1.7(1, 2)} are satisfied, then
\begin{align*}
T_\gamma
&\qeqq&
\sum_{\OO \in \OO^{H\conn}(\Ad^*(G)Z^*_o)} &c_\OO(T_\gamma)\muhat^{H\conn}_\OO
\intertext{if and only if}
T_\gamma\bigl(\chrc{\sbtl\fh x r, \AddChar_{X^*}^\vee}\bigr)
&\qeqq&
\sum_{\OO' \in \OO^{(H \cap \Int(g)\inv G')\conn}(\Ad^*(g)\inv Z^*_o)} &c_{\OO'}(T_\gamma)\muhat^{H\conn}_{\OO'}\bigl(\chrc{\sbtl\fh x r, \AddChar_{X^*}^\vee}\bigr)
\end{align*}
for all \(x \in \BB(H \cap \Int(g)\inv G')\) and \(X^* \in \sbtl{\Lie^*(H)}x{-r} \cap \Ad^*(g)\inv(Z^*_o + \sbjtlpp{\Lie^*(G')}{-r})\), whenever \(g \in G\) is such that \(\Ad^*(g)\inv Z^*_o\) belongs to \(\Lie^*(H)\).
\end{lem}

\begin{proof}
As in the proof of Theorem \ref{thm:asymptotic-exists}, we use the function spaces
\begin{align*}
\mc D^{\prime\,{-r}}_{\Rpp{-r}}
& {}= \sum_{x \in \BB(H')} \Cnts(\sbat{\Lie^*(H)}x{-r})          \\
\intertext{and}
\mc D_{\Rpp{-r}}
& {}= \sum_{x \in \BB(H)} \Cc(\Lie^*(H)/\sbtlpp{\Lie^*(H)}x{-r})
\end{align*}
of \cite{jkim-murnaghan:charexp}*{Definition 3.1.1}, adapted from \(\Lie(G)\) to \(\Lie^*(H)\).

Let \(\check T_\gamma\) be the distribution on \(\Lie^*(H)\) given by \(\check T_\gamma(f^*) = T_\gamma(
\check f^*)\)
for all \(f^* \in \Hecke(\Lie^*(H))\).  Then \(\check T_\gamma\) belongs to the span of \(\set{\mu_\OO}{\OO \in \OO^{H\conn}(\Ad^*(G)Z^*_o)}\).
Choose, for each \(g_0 \in G\), an element \(\check T_{\gamma, \Ad^*(g_0)\inv Z^*_o}\) in the span of \(\set{\mu_\OO}{\OO \in \OO^{H\conn}(\Ad^*(g)\inv Z^*_o)}\) such that \(\check T_\gamma\) equals \(\sum_{g_0} \check T_{\gamma, \Ad^*(g_0)\inv Z^*_o}\).  It suffices to show, for each \(g_0 \in G\conn\), that
\begin{equation}
\tag{$*$}
\sublabel{eq:homogeneity}
\check T_{\gamma, \Ad^*(g_0)\inv Z^*_o}
\qeqq
\sum_{\OO \in \OO^{H\conn}(\Ad^*(g_0)\inv Z^*_o)} c_\OO(T_\gamma)\mu^{H\conn}_\OO
\end{equation}
on \(\mc D_{\Rpp{-r}}\).  For notational convenience, we assume that \(g_0\) is the identity.

By \cite{jkim-murnaghan:charexp}*{Lemma 3.1.5 and Theorem 3.1.7(1, 2)}, it is enough to show that \(\check T_{\gamma, Z^*_o}\) agees with the right-hand side of \loceqref{eq:homogeneity} on \(\mc D^{\prime\,{-r}}_{\Rpp{-r}}\), which is to say that they agree on the characteristic function \(f^*\) of each coset \(X^* + \sbtlpp{\Lie^*(H)}x{-r}\) with \(x \in \BB(H')\) and \(X^* \in \sbtl{\Lie^*(H)}x{-r}\).  By our normalisation convention for \anonchrc, we have that \(
\check f^*\)
equals \(\chrc{\sbtl{\Lie(H)}x r, \AddChar_{X^*}^\vee}\).

If \(X^* + \sbtlpp{\Lie^*(H)}x{-r}\) does not intersect \(Z^*_o + \sbjtlpp{\Lie^*(G')}{-r}\), then both sides of the equality vanish; so we may, and do, assume that the two sets intersect.

If \(g\) belongs to \(G\), and \(\check T_{\gamma, \Ad^*(g)\inv Z^*_o}(f^*)\) does not vanish, then we have that \(\Ad^*(g H\conn)\inv(Z^*_o + \sbjtlpp{\Lie^*(G')}{-r})\) intersects \(X^* + \sbtlpp{\Lie^*(H)}x{-r}\), hence, by Remark \ref{rem:X*} and Hypothesis \initref{hyp:Z*}\subpref{orbit}, that \(G'g H\conn\) is the trivial double coset.  That is, we have that \(\check T_\gamma(f^*) = \sum_g \check T_{\gamma, \Ad^*(g)\inv Z^*_o}(f^*)\) equals \(\check T_{\gamma, Z^*_o}(f^*)\).  By assumption, \(\check T_{\gamma, Z^*_o}(f^*)\) equals
\[
\sum_{\OO' \in \OO^{H\primeconn}(Z^*_o)} c_{\OO'}(T_\gamma)\muhat^{H\primeconn}_{\OO'}\bigl(\chrc{\sbtl{\Lie(H')}x r, \AddChar_{X^*}^\vee}\bigr),
\]
as desired.

We have proven the `if' direction.  The `only if' direction is much easier, and involves essentially the argument that we used to show that \(\check T_\gamma(f^*)\) equalled \(\check T_{\gamma, Z^*_o}(f^*)\).
\end{proof}

\section{Computation of asymptotic expansions}
\label{sec:quantitative}

\subsection{Gauss sums and Weil indices}
\label{sec:Gauss}

We know from Theorem \ref{thm:asymptotic-exists} that the characters of certain representations have asymptotic expansions about (nearly) arbitrary semisimple elements.  Although Lemma \ref{lem:asymptotic-check} provides a way to verify the correctness of a potential asymptotic expansion, it gives no idea what the coefficients in such an expansion should be.  The main result of \S\ref{sec:quantitative}, Theorem \ref{thm:asymptotic-pi-to-pi'}, shows how to use the inductive structure in Yu's construction of supercuspidals \cite{yu:supercuspidal} to reduce character computations for such representations of \(G\) to analogous computations on a tame, twisted Levi subgroup \(G'\).  The recipe involves some fourth roots of unity, which have previously been relevant in stability calculations (see, for example, \xcite{debacker-spice:stability}*{Proposition \xref{prop:stable-sign}} and \cite{kaletha:regular-sc}*{\S4.7}).  In \xcite{adler-spice:explicit-chars}*{\S\xref{ssec:gauss}}, Adler and the author interpreted analogous fourth roots of unity as Gauss sums.  In this paper, we interpret them as Weil indices, in the spirit of \cite{waldspurger:loc-trace-form}*{\S VIII.1}.  Pleasantly, we manage to avoid the centrality assumption \xcite{adler-spice:explicit-chars}*{Hypothesis \xref{hyp:X*-central}} by using Proposition \ref{prop:Q-to-B}.

\begin{defn}
\label{defn:Weil-index}
In \cite{weil:weil}*{\S14}, Weil associated to a (non-degenerate) quadratic space over a local field a constant \via the Fourier transform.  We follow the equivalent description in \cite{waldspurger:loc-trace-form}*{\S VIII.1}.  Let \((V, q)\) be a quadratic space over \field, and \(b\) the unique symmetric bilinear form so that \(q(v)\) equals \(b(v, v)\) for all \(v \in V\).  If \(q\) is non-degenerate, then there is a unique, unit-modulus complex number, denoted in \cites{weil:weil,waldspurger:loc-trace-form} by \(\gamma_\AddChar(q)\) and usually called the \term{Weil index} (of \((V, q)\) with respect to \(\AddChar\)), such that, whenever \(L\) is a lattice containing \(L^\bullet \ldef \sett{v \in V}{\(b(v, w) \in \sbjtlp\field 0\) for all \(w \in L\)}\), we have that
\[
\meas(L)^{1/2}\uint_L \AddChar_{1/2}(q(v))\upd v
\qeqq
\meas(L^\bullet)^{1/2}\gamma_\AddChar(q).
\]
(In \cite{waldspurger:loc-trace-form}*{\S VIII.1}, it is required that \(L^\bullet\) be contained in \(2L\); but, since we are assuming that \(p\) is odd, this requirement is equivalent to ours.)  In general, we define the Weil index of \((V, q)\) with respect to \(\AddChar\) to be that of \((V/{\operatorname{rad}(V, q)}, q_+)\), where \(\operatorname{rad}(V, q) = \sett{v \in V}{\(b(v, w) = 0\) for all \(w \in V\)}\) is the radical of \((V, q)\), and \(q_+\) is the induced (non-degenerate) quadratic form.
\end{defn}

In Notation \ref{notn:Gauss}, we define a quadratic form \(q_{X^*, \gamma}\) using \(1 - \Ad(\gamma)\).  In fact, using just \(-\Ad(\gamma)\) would not change the form, but we find it convenient to write it this way.  This quadratic form gives rise to the Gauss sums \(\Gauss\) that are ubiquitous in character computations; in this paper, they pop out of our calculations in Proposition \ref{prop:Gauss-appears}.

\begin{notn}
\label{notn:Gauss}
For any \(X^* \in \Lie^*(G)\) and \(\gamma \in G\), we define
\[
\mnotn{b_{X^*, \gamma}}(Y_1, Y_2) = \pair[\big]{X^*}{\comm{Y_1}{(1 - \Ad(\gamma))Y_2}}
\qandq
\mnotn{q_{X^*, \gamma}}(Y) = b_{X^*, \gamma}(Y, Y)
\]
for all \(Y, Y_1, Y_2 \in \Lie(G)\).  We write \mnotn{\Gauss_G(X^*, \gamma)} for the Weil index of the pairing \(q_{X^*, \gamma}\) on \(\Lie(G)\).

We have for \(g \in G\) that \(\Ad(g)\) furnishes an isomorphism of \((\Lie(G), q_{X^*, \gamma})\) onto \((\Lie(G), q_{\Ad^*(g)X^*, \Int(g)\gamma})\), so that \(\Gauss_G(X^*, \gamma)\) equals \(\Gauss_G(\Ad^*(g)X^*, \Int(g)\gamma)\).  Thus, if \(\OO'\) is contained the coadjoint orbit of \(\Cent_G(\gamma)\) containing \(X^*\), then we may write \mnotn{\Gauss_G(\OO', \gamma)} for \(\Gauss_G(X^*, \gamma)\).

If
	\begin{itemize}
	\item \bH is a reductive subgroup of \bG,
	\item we have an \bH-stable decomposition \(\Lie(\bG) = \Lie(\bH) \oplus \Lie(\bH)^\perp\),
and	\item \(\gamma\) belongs to \(H\) and \(X^*\) annihilates \(\Lie(\bH)^\perp\),
	\end{itemize}
then we write \(\Gauss_{G/H}(X^*, \gamma)\) (or \(\Gauss_{G/H}(\OO', \gamma)\)) for the quotient \(\Gauss_G(X^*, \gamma)/\Gauss_H(X^*, \gamma)\).
\end{notn}


We now recall the
	\begin{itemize}
	\item non-negative real number \mnotn r,
	\item element \(\mnotn\gamma \in G\), with associated groups \(\bP^\mp = \CC\bG{-\infty}(\gamma\pinv)\), \(\bN^\pm\), and \(\bM = \CC\bG 0(\gamma)\),
and	\item point \(x \in \BB(\CC G r(\gamma))\),
	\end{itemize}
satisfying Hypotheses \ref{hyp:funny-centraliser} and \ref{hyp:gamma}, from \S\ref{sec:depth-matrix}; and the
	\begin{itemize}
	\item tame, twisted Levi subgroup \(\bG'\)
	\end{itemize}
from \S\ref{sec:nearly-good}.  We do \emph{not} impose Hypothesis \ref{hyp:gamma-central}; and do not need the element \(Z^*_o\) from \S\ref{sec:nearly-good}, or the mock-exponential map from Hypothesis \ref{hyp:mexp}.

We \emph{do} fix an element \mnotn{X^*} satisfying Hypothesis \ref{hyp:X*}, and we do require that \(\gamma\) belongs to \(G'\).  As in \S\ref{sec:nearly-good}, we use primes to denote the analogues in \(\bG'\) of constructions in \bG, so, for example, \(\bM'\) stands for \(\CC{\bG'}0(\gamma)\) (subject to the proviso in Remark \ref{rem:tame-Levi}, that we may refer directly only to the identity component of \(\bM'\)).

Contrary to our notation elsewhere, in the proof of Proposition \ref{prop:lattice-orth} we find it convenient to use \bH to stand for \(\CC\bG r(\gamma^2)\), not \(\CC\bG r(\gamma)\); but we return to the usual notation before Proposition \ref{prop:Gauss-to-Weil}.

The pairing \(q_{X^*, \gamma}\) is usually degenerate, but, in Proposition \ref{prop:lattice-orth}, we pick out a non-degenerate sub\emph{space}, and even a non-degenerate sub\emph{lattice}.

\begin{prop}
\label{prop:lattice-orth}
We have that \(\sbtl{\Lie(\CC G r(\gamma), G)}x{(\Rp0, (r - \ord_{\gamma\pinv})/2)}\) pairs \via \(b_{X^*, \gamma}\) with itself into \(\sbjtl\field 0\), and with \(\sbtl{\Lie(\CC G r(\gamma), N^+, G)}x{(\Rp0, (r - \ord_{\gamma\inv})/2, \Rpp{(r - \ord_\gamma)/2})}\) on the right or \(\sbtl{\Lie(\CC G r(\gamma), N^-, G)}x{(\Rp0, (r - \ord_\gamma)/2, \Rpp{(r - \ord_{\gamma\inv})/2})}\) on the left into \(\sbjtlp\field 0\).

Further, the \(q_{X^*, \gamma}\)-orthogonal modulo \(\sbjtlp\field 0\) of
\[
\sbtl{\Lie(\CC G r(\gamma^2), G', G)}x{(\infty, \infty, (r - \ord_{\gamma\pinv})/2)}
\]
in \(\Lie(\CC G r(\gamma^2))^\perp \cap \Lie(G')^\perp\) is
\[
\sbtlp{\Lie(\CC G r(\gamma^2), G', G)}x{(\infty, \infty, (r - \ord_{\gamma\pinv})/2)}.
\]
\end{prop}

\begin{proof}
For this proof, put \(\bH = \CC\bG r(\gamma^2)\conn\).  (Elsewhere, we write \bH for \(\CC\bG r(\gamma)\conn\).)  Also for this proof, put \(s_{\gamma\pinv}^\pm = (r \pm \ord_{\gamma\pinv})/2\),
\begin{align*}
V          & {}= \sbtl{\Lie(\CC G r(\gamma), G)}x{(\Rp0, s_{\gamma\pinv}^-)}, \\
V_{1, {+}} & {}= \sbtl{\Lie(\CC G r(\gamma), N^-, G)}x{(\Rp0, s_\gamma^-, \Rp{s_{\gamma\inv}^-})}, \\
\intertext{and}
V_{2, {+}} & {}= \sbtl{\Lie(\CC G r(\gamma), N^+, G)}x{(\Rp0, s_{\gamma\inv}^-, \Rp{s_\gamma^-})}.
\end{align*}
(We apologise for resulting notation like \(\Rp{s_{\gamma\inv}^-}\).)

The argument for the first statement is similar to, but easier than, the proof of Proposition \ref{prop:Q-and-B}.

For this paragraph, fix an element \(Y_1\) of \(V\), a real number \(i_2\) satisfying \(i_2 < r\), and an element \(Y_2\) of \(\sbtl{\Lie(\CC G{i_2}(\gamma\pinv))}x{(r - i_2)/2}\).  By Lemma \initref{lem:commute-gp}\subpref{down}, we have that \((1 - \Ad(\gamma))Y_2\) belongs to \(\sbtl{\Lie(\CC G{i_2}(\gamma))}x{(r + i_2)/2} + \sbtl{\Lie(\CC G{i_2}(\gamma\inv))}x{\max \sset{(r + i_2)/2, (r - i_2)/2}}\), and to \(\sbtlpp{\Lie(\CC G{i_2}(\gamma))}x{(r + i_2)/2} + \sbtl{\Lie(\CC G{i_2}(\gamma\inv))}x{\max \sset{\Rpp{(r + i_2)/2}, (r - i_2)/2}}\) if \(Y_2\) belongs to \(V_{2, {+}}\).  Thus, by Lemma \initref{lem:filtration}\subpref{Lie-Lie}, we have that \(\comm{Y_1}{(1 - \Ad(\gamma))Y_2}\) belongs always to \(\sbtl{\Lie(N^+, G)}x{(r - \ord_{\gamma\inv}, r)}\), hence \(b_{X^*, \gamma}(Y_1, Y_2) = \pair[\big]{X^*}{\comm{Y_1}{(1 - \Ad(\gamma))Y_2}}\) to \(\sbjtl\field 0\), and to \(\sbtl{\Lie(N^+, G)}x{(r - \ord_{\gamma\inv}, \Rp r)}\), hence \(b_{X^*, \gamma}(Y_1, Y_2)\) to \(\sbjtlp\field 0\), if \(Y_1\) belongs to \(V_{1, {+}}\) \emph{or} \(Y_2\) to \(V_{2, {+}}\).  The first statement follows.

For the second statement, put
\[
V^\perp    = \sbtl{\Lie(H, G', G)}x{(\infty, \infty, s_{\gamma\pinv}^-)}.
\]
Suppose that \(Y_1 \in V^\perp\) is such that \(b_{X^*, \gamma}(Y_1, Y_2) + b_{X^*, \gamma}(Y_2, Y_1)\), which equals
\[
\pair[\big]
	{\ad^*((\Ad(\gamma)\inv - \Ad(\gamma))Y_1)X^* +
	\ad^*(\Ad(\gamma)\inv Y_1)(\Ad(\gamma)\inv - 1)X^*}
{Y_2},
\]
belongs to \(\sbjtlp\field 0\) for all \(Y_2 \in V^\perp\).  Then
\[
(\ad^*((\Ad(\gamma)\inv - \Ad(\gamma))Y_1)X^* + \ad^*(\Ad(\gamma)\inv Y_1)(\Ad(\gamma)\inv - 1)X^*
\]
belongs to \(\sbtlp{\Lie^*(H, G', G)}x{(-\infty, -\infty, -s_{\gamma\pinv}^-)}\).
Since \((\Ad(\gamma)\inv - 1)X^*\) belongs to \(\sbtl{\Lie^*(H')}x 0\) by Hypothesis \initref{hyp:X*}\subpref{depth} and Remark \ref{rem:gamma:Lie*}, and since we have by Lemma \initref{lem:commute-gp}\subpref{down} that \(\Ad(\gamma)\inv Y_1 = Y_1 - (1 - \Ad(\gamma)\inv)Y_1\) belongs to \(\sbtl{\Lie(H, N^+, G)}x{(\Rp r, s_{\gamma\inv}^+, s_\gamma^-)}\), we have by Lemma \initref{lem:commute-Lie*}\subpref{down} that \(\ad^*(\Ad(\gamma)\inv Y_1)(\Ad(\gamma)\inv - 1)X^*\) belongs to \(\sbtl{\Lie^*(H, N^+, G)}x{(\Rp r, s_{\gamma\inv}^+, s_\gamma^-)} \subseteq \sbtlp{\Lie^*(H, G', G)}x{(-\infty, -\infty, -s_{\gamma\pinv}^-)}\).
Thus, we have that
\[
\ad^*((\Ad(\gamma)\inv - \Ad(\gamma))Y_1)X^*
\qtextq{belongs to}
\sbtlp{\Lie^*(H, G', G)}x{(-\infty, -\infty, -s_{\gamma\pinv}^-)}.
\]
By Lemma \initref{lem:commute-Lie*}\subpref{up}, we have that
\[
(\Ad(\gamma)\inv - \Ad(\gamma))Y_1
\qtextq{belongs to}
\sbtlp{\Lie(H, G', G)}x{(-\infty, -\infty, s_{\gamma\pinv}^+)}.
\]
Thus, by Lemma \initref{lem:commute-gp}\subpref{bi-up}, we have that \(Y_1 \in \sbtl{\Lie(H, G', G)}x{(\infty, \infty, s_{\gamma\pinv})}\) belongs to \(\sbtlp{\Lie(H, G', G)}x{(-\infty, -\infty, s_{\gamma\pinv}^-)}\), hence to \(\sbtlp{\Lie(H, G', G)}x{(\infty, \infty, s_{\gamma\pinv}^-)}\), as claimed.
\end{proof}

\begin{cor}
\label{cor:lattice-orth}
We have that
\begin{gather*}
\frac{\Gauss_{G/\CC G r(\gamma^2)}(X^*, \gamma)}{\Gauss_{G'/\CC{G'}r(\gamma^2)}(X^*, \gamma)} \\
\intertext{equals}
\card{\sbat{\Lie(\CC G r(\gamma^2), G', G)}x{(\Rp0, \Rp0, (r - \ord_{\gamma\pinv})/2)}}^{-1/2}\sum_Y \AddChar_{1/2}(q_{X^*, \gamma}(Y)),
\end{gather*}
where the sum over \(Y\) runs over
\begin{multline*}
\sbtl{\Lie(\CC G r(\gamma^2), G', G)}x{(\Rp0, r - \ord_{\gamma\pinv}, (r - \ord_{\gamma\pinv})/2)}
	/ \\
\sbtl{\Lie(\CC G r(\gamma^2), G', G)}x{(\Rp0, r - \ord_{\gamma\pinv}, \Rpp{(r - \ord_{\gamma\pinv})/2})}.
\end{multline*}
\end{cor}

\begin{proof}
This follows immediately from Proposition \ref{prop:lattice-orth}, since the left-hand side is the Weil index associated to \((\Lie(\CC G r(\gamma^2))^\perp \cap \Lie(G')^\perp, q_{X^*, \gamma})\), once we notice that the natural maps
\begin{align*}
&\sbtl{\Lie(\CC G r(\gamma^2), G', G)}x{(\infty, \infty, (r - \ord_{\gamma\pinv})/2)} \\ \mapdefaultarrow{}
&\sbtl{\Lie(\CC G r(\gamma^2), G', G)}x{(\Rp0, r - \ord_{\gamma\pinv}, (r - \ord_{\gamma\pinv})/2)} \\ \mapdefaultarrow{}
&\sbtl{\Lie(\CC G r(\gamma^2), G', G)}x{(\Rp0, \Rp0, (r - \ord_{\gamma\pinv})/2)}
\end{align*}
induce isomorphisms
\begin{align*}
&\sbat{\Lie(\CC G r(\gamma^2), G', G)}x{(\infty, \infty, (r - \ord_{\gamma\pinv})/2)} \\ \mapisoarrow{}
&\sbtl{\Lie(\CC G r(\gamma^2), G', G)}x{(\Rp0, r - \ord_{\gamma\pinv}, (r - \ord_{\gamma\pinv})/2)}
	/ \\
&\qquad\sbtl{\Lie(\CC G r(\gamma^2), G', G)}x{(\Rp0, r - \ord_{\gamma\pinv}, \Rpp{(r - \ord_{\gamma\pinv})/2})} \\ \mapisoarrow{}
&\sbat{\Lie(\CC G r(\gamma^2), G', G)}x{(\Rp0, \Rp0, (r - \ord_{\gamma\pinv})/2)}.
\end{align*}
\end{proof}

Corollary \ref{cor:Gauss-const} is a local constancy result for Gauss sums.  It shows that certain quotients of Gauss sums can ``see'' information only up to depth \(r\) on the group side, and depth \(-r\) on the dual-Lie-algebra side.

\begin{cor}
\label{cor:Gauss-const}
The quantity
\begin{align*}
&\frac{\Gauss_{G/\CC G r(\gamma^2)}(X^*, \gamma)}{\Gauss_{G'/\CC{G'}r(\gamma^2)}(X^*, \gamma)}, \\
\intertext{respectively}
&\frac{\Gauss_{G/\CC G r(\gamma)}(X^*, \gamma)}{\Gauss_{G'/\CC{G'}r(\gamma)}(X^*, \gamma)}\dotm\frac{\Gauss_{\CCp G 0(\gamma^2)/\CCp G 0(\gamma)}(X^*, \gamma)\inv}{\Gauss_{\CCp{G'}0(\gamma^2)/\CCp{G'}0(\gamma)}(X^*, \gamma)\inv},
\end{align*}
does not change if we replace \(X^*\) by a translate under \(\sbtlpp{\Lie^*(\CC{G'}r(\gamma^2))}x{-r}\), respectively \(\sbtlpp{\Lie^*(\CC{G'}r(\gamma))}x{-r}\), and \(\gamma\) by a translate under \(\sbtl{\CC{G'}r(\gamma^2)}x r\), respectively \(\sbtl{\CC{G'}r(\gamma)}x r\).
\end{cor}

\begin{proof}
The first statement is a direct consequence of Corollary \ref{prop:lattice-orth}.  Since \(\CCp\bG 0(\gamma) \cap \CC\bG r(\gamma^2)\) equals \(\CC\bG r(\gamma)\), and similarly in \(\bG'\), by Hypothesis \initref{hyp:gamma}\subpref{bi-Lie-Lie}, the second statement follows by applying the first to \bG, \(\CCp\bG 0(\gamma^2)\), and \(\CCp\bG 0(\gamma)\).
\end{proof}

The main result of \S\ref{sec:Gauss}, Proposition \ref{prop:Gauss-to-Weil}, computes a certain integral on the group in terms of the Gauss sums \Gauss defined \via the Lie algebra.  The usual device for transferring between the group and the Lie algebra is an exponential map, but we have avoided assuming (so far in \S\ref{sec:quantitative}) that there is an exponential map, or even a Moy--Prasad isomorphism.  We wish to continue avoiding this, but we need some way of relating the behaviour of the \emph{group} character \(\phi\) appearing in Proposition \ref{prop:Gauss-to-Weil} to the dual-\emph{Lie-algebra} element \(X^*\).
By \cite{yu:models}*{\S8.1(ii) and remark after Corollary 5.6}, a Moy--Prasad isomorphism exists, regardless of any tameness hypotheses, for adjoint groups.  We state the necessary properties of such an isomorphism in Hypothesis \ref{hyp:MP-ad}.  It may be constructed as in \cite{adler:thesis}*{\S1.5}, taking into account the modifications as in \cite{steinberg:endomorphisms}*{proof of Theorem 8.2}; we do not give the details here.

\begin{hyp}
\initlabel{hyp:MP-ad}
Let \(\bM\adform\) be the adjoint quotient of \bM.  For each
	\begin{itemize}
	\item tame, twisted Levi sequence \(\vec\bM\) in \bM containing \(\gamma\), such that \(x\) belongs to \(\BB(\vec M)\),
and	\item grouplike depth vector \(\vec\jmath\) satisfying \(\vec\jmath \cvee \vec\jmath \ge \Rp{\vec\jmath}\),
	\end{itemize}
there is an isomorphism
\[
	\map[\mapisoarrow]{\matnotn{exp}{\sbat{\ol\mexp}x{\vec\jmath}}}
	{\sbat{\Lie(\vec M\adform)}x j
	}
	{\sbat{(\vec M\adform)}x{\vec\jmath}
	}.
\]
These isomorphisms satisfying the following.
	\begin{enumerate}
	\item\sublabel{ad} For all grouplike depth vectors \(\vec\jmath_k\) satisfying \(\vec\jmath_k \cvee \vec\jmath_k \ge \Rp{\vec\jmath_k}\), and elements \(Y_k \in \sbtl{\Lie(\vec M\adform)}x{\vec\jmath_k}\) and \(v_k \in \sbat{\ol\mexp}x{\vec\jmath_k}(Y_k)\), for \(k \in \sset{1, 2}\), we have that
\[
\comm{v_1}{v_2}
\qtextq{belongs to}
\sbat{\ol\mexp}x{\vec\jmath_1 \cvee \vec\jmath_2}(\comm{Y_1}{Y_2}).
\]
	\item\sublabel{Ad} For all \(i \in \tR_{\ge 0}\), grouplike depth vectors \(\vec\jmath\) satisfying \(\vec\jmath \cvee \vec\jmath \ge \Rp{\vec\jmath}\), and elements \(Y \in \sbtl{\Lie(\CC{\vec M}i(\gamma))}x{\vec\jmath}\), there is an element \(v \in \sbat{\ol\mexp}x{\vec\jmath}(Y) \cap \sbtl{\CC{\vec M}i(\gamma)}x{\vec\jmath}\) so that
\[
\comm v\gamma
\qtextq{belongs to}
\sbat{\ol\mexp}x{i + \vec j}\bigl((1 - \Ad(\gamma))Y\bigr).
\]
	\item\sublabel{refine} For all grouplike depth vectors \(\vec\jmath_k\) satisfying \(\vec\jmath_k \cvee \vec\jmath_k \ge \Rp{\vec\jmath_k}\), for \(k \in \sset{1, 2}\), if \(\Rp{\vec\jmath_1}\) and \(\Rp{\vec\jmath_2}\) are equal, then the diagram
\[\xymatrix{
\sbat{\Lie(M)}x{\max \sset{\vec\jmath_1, \vec \jmath_2}} \ar[r]\ar[d] & \sbat{\Lie(M)}x{\vec\jmath_1} \ar[d] \\
\sbat{\Lie(M)}x{\vec\jmath_2} \ar[r] & \sbat M x{\min \sset{\vec\jmath_1, \vec\jmath_2}}
}\]
commutes.
	\end{enumerate}
\end{hyp}

We can almost use Hypothesis \ref{hyp:MP-ad} to define the notion of a dual blob of a character of \(\sbtl G x r\), but we do not want to restrict ourselves to characters that factor through the adjoint quotient.  Instead, in Hypothesis \ref{hyp:phi}, as in Hypothesis \ref{hyp:K-type}, we speak in a roundabout way of dual blobs, this time \via commutators.  As with Hypothesis \ref{hyp:X*}, we isolate Hypothesis \ref{hyp:phi} only to have a convenient reference; it will automatically be satisfied when we need it, in Theorem \ref{thm:asymptotic-pi-to-pi'}.

We now return to the notation used elsewhere in the document, writing \matnotn H\bH for \(\CC\bG r(\gamma)\) (and so \matnotn H{\bH'} for \(\CC{\bG'}r(\gamma)\), subject to the proviso in Remark \ref{rem:tame-Levi}).  Let \mnotn\phi be a character of \(\sbtl{H'}x r/\sbtlp{\CC{\Der G \cap G'}r(\gamma)}x r\).  We write again \(\phi\) for its extension trivially across \(\sbtl{(\CC{\Der G \cap G'}r(\gamma), \Der G \cap G')}x{(\Rp r, r)}\) to \(\sbtl{G'}x r\), and then \matnotn{phi}{\hat\phi} for its further extension trivially across \(\sbtlp{\Der(G', G)}x{(r, s)}\) to \(\sbtl{(G', G)}x{(r, \Rp s)}\).

\begin{hyp}
\label{hyp:phi}
We have that
\[
\hat\phi(\comm{v_1}{v_2})
\qeqq
\AddChar_{X^*}(\comm{Y_1}{Y_2})
\]
for all \(j_k \in \tR_{> 0}\), and elements \(Y_k \in \sbtl{\Lie(\CCp G 0(\gamma))}x{j_k}\) and \(v_k \in \sbat{\ol\mexp}x{j_k}(Y_k) \cap \CCp G 0(\gamma))\), for \(k \in \sset{1, 2}\), such that \(j_1 + j_2 \ge r\).
\end{hyp}

Proposition \ref{prop:Gauss-to-Weil} is used in Proposition \ref{prop:Gauss-appears} to show that Gauss sums appear when evaluating invariant distributions at certain test functions.

\begin{prop}
\label{prop:Gauss-to-Weil}
We have that
\begin{align*}
&\card{\sbat{(H, \CCp{G'}0(\gamma), \CCp G 0(\gamma))}x{(\Rp0, \Rp0, (r - \ord_\gamma)/2)}}^{1/2}\times{} \\
&\qquad\uint_{\sbtl{(H, \CCp{G'}0(\gamma), \CCp G 0(\gamma))}x{(\Rp0, r - \ord_\gamma, (r - \ord_\gamma)/2)}}
	\hat\phi(\comm v\gamma)\upd v \\
\intertext{equals}
&\Gauss_{\CCp G 0(\gamma)/H}(X^*, \gamma)/\Gauss_{\CCp{G'}0(\gamma)/H'}(X^*, \gamma)\inv.
\end{align*}
\end{prop}

\begin{proof}
Put \(\mc V^\perp = \sbtl{(H, \CCp{G'}0(\gamma), \CCp G 0(\gamma))}x{(\Rp0, r - \ord_\gamma, (r - \ord_\gamma)/2)}\).
We use Notation \ref{notn:Q-and-B}.  Proposition \ref{prop:Q-and-B} gives that
\(\mf Q_\gamma\) is multiplicative modulo \(\sbtlp{\Der(G', G)}x{(r, s)}\) on \(\mc V^\perp\).

In particular, we have for any \(i \in \R\) with \(0 < i < r\) that
\begin{align*}
&\uint_{\mc V^\perp}
	\hat\phi(\mf Q_\gamma(v))\upd v \\
\intertext{equals}
&\uint_{\mc V^\perp/\spjhd{\mc V}i} \uint_{\spjhd{\mc V}i}
	\hat\phi(\mf Q_\gamma(v_1 v_2))\upd v_2\,\upd v_1 \\ ={}
&\uint_{\sbtl{(H, \CCp{G'}i(\gamma), \CCp G i(\gamma))}x{(\Rp0, r - \ord_\gamma, (r - \ord_\gamma)/2)}}
	\hat\phi(\mf Q_\gamma(v_1))\times{} \\
&\qquad\uint_{\spjhd{\mc V}i}
	\hat\phi(\mf Q_\gamma(v_2))\upd v_2\,
\upd v_1,
\end{align*}
where we have put
\[
\spjhd{\mc V}i = \sbtl{(\CC{G'}i(\gamma), \CC G i(\gamma), \CCp{G'}0(\gamma), \CCp G 0(\gamma))}x{(r - i, (r - i)/2, r - \ord_\gamma, (r - \ord_\gamma)/2)};
\]
and reasoning inductively gives that
\[
\uint_{\mc V^\perp} \hat\phi(\mf Q_\gamma(v))\upd v
\qeqq
\prod_{0 < i < r}
\uint_{\sbtl{(\CC{G'}i(\gamma), \CC G i(\gamma))}x{(r - i, (r - i)/2)}}
	\hat\phi(\mf Q_\gamma(v))\upd v.
\]
We have by Hypotheses \ref{hyp:MP-ad} and \ref{hyp:phi}, and Proposition \ref{prop:Q-to-B}, that
\begin{align*}
\hat\phi(\mf Q_\gamma(v))^2
&\qeqq
\AddChar_{1/2}(q_{X^*, \gamma}(Y))^2, \\
\intertext{hence, since both are \(p^\infty\)th complex roots of unity and \(p\) is odd, that}
\hat\phi(\mf Q_\gamma(v))
&\qeqq
\AddChar_{1/2}(q_{X^*, \gamma}(Y)),
\end{align*}
for all \(0 < i < r\), and elements \(Y \in \sbtl{\Lie(\CC{G'}i(\gamma), \CC G i(\gamma))}x{(r - i, (r - i)/2)}\) and \(v \in \sbat{\ol\mexp}x{(r - i, (r - i)/2)}(Y) \cap \sbtl{(\CC{G'}i(\gamma), \CC G i(\gamma))}x{(r - i, (r - i)/2)}\).  Lemma \ref{lem:MP-card} gives that
\begin{multline*}
\card{\sbat{(H, \CCp{G'}0(\gamma), \CCp G 0(\gamma))}x{(\Rp0, \Rp0, (r - \ord_\gamma)/2)}} \\
\qeqq
\card{\sbat{\Lie(H, \CCp{G'}0(\gamma), \CCp G 0(\gamma))}x{(\Rp0, \Rp0, (r - \ord_\gamma)/2)}},
\end{multline*}
so we have shown that
\[
\card{\sbat{(H, \CCp{G'}0(\gamma), \CCp G 0(\gamma))}x{(\Rp0, \Rp0, (r - \ord_\gamma)/2)}}^{1/2}\uint_{\mc V^\perp} \hat\phi(\mf Q_\gamma(v))\upd v
\]
equals
\begin{multline*}
\prod_{0 < i < r} \card{\sbat{\Lie(\CC{G'}i(\gamma), \CC G i(\gamma))}x{(\Rp0, (r - i)/2)}}^{1/2}\times{} \\
\uint_{\sbtl{\Lie(\CC{G'}i(\gamma), \CC G i(\gamma))}x{(r - i, (r - i)/2)}}
	\AddChar_{1/2}(q_{X^*, \gamma}(Y))\upd Y,
\end{multline*}
which we see, by arguing as above, equals
\begin{multline*}
\card{\sbat{\Lie(H, \CCp{G'}0(\gamma), \CCp G 0(\gamma))}x{(\Rp0, \Rp0, (r - \ord_\gamma)/2)}}^{1/2}\times{} \\
\uint_{\sbtl{\Lie(H, \CCp{G'}0(\gamma), \CCp G 0(\gamma))}x{(\Rp0, r - \ord_\gamma, (r - \ord_\gamma)/2)}}
	\AddChar_{1/2}(q_{X^*, \gamma}(Y))\upd Y.
\end{multline*}
By Corollary \ref{cor:lattice-orth}, we are done.
\end{proof}

\subsection{Matching distributions on groups and subgroups}
\label{sec:dist}

We recall the
	\begin{itemize}
	\item non-negative real number \mnotn r,
	\item element \(\mnotn\gamma \in G\), with associated groups \(\bP^\mp = \CC\bG{-\infty}(\gamma\pinv)\), \(\bN^\pm\), \(\bM = \CC\bG 0(\gamma)\), and \(\bH = \CC\bG r(\gamma)\),
and	\item point \(x \in \BB(H)\),
	\end{itemize}
from \S\ref{sec:depth-matrix}, of which we now require that \(r\) be positive; the
	\begin{itemize}
	\item tame, twisted Levi subgroup \(\bG'\)
	\end{itemize}
from \S\ref{sec:nearly-good}, which we require to contain \(\gamma\); and the
	\begin{itemize}
	\item element \(X^* \in \Lie^*(H')\)
and	\item characters \(\phi\) of \(\sbtl{G'}x r/\sbtl{(\CC{\Der G \cap G'}r(\gamma), G')}x{(\Rp r, r)}\) and \(\hat\phi\) of \(\sbtl{(G', G)}x{(r, \Rp s)}/\sbtl{\Der(H', G', G)}x{(\Rp r, r, \Rp s)}\),
	\end{itemize}
satisfying Hypotheses \ref{hyp:X*} and \ref{hyp:phi} (and so, indirectly, Hypothesis \ref{hyp:MP-ad}), from \S\ref{sec:Gauss}.  We use primes to denote the analogues in \(\bG'\) of constructions in \bG; so, for example, \(\bH'\) stands for \(\CC{\bG'}r(\gamma)\) (subject to the proviso in Remark \ref{rem:tame-Levi}, that we may refer directly only to the identity component of \(\bH'\)).

This section approaches the explicit computation of sample values of an invariant distribution \(T\) on \(G\), as in Lemma \ref{lem:sample}, by converting it into a computation of analogous sample values for an invariant distribution \(T'\) on the tame, twisted Levi subgroup \(G'\) that ``matches'' \(T\) in some sense.  Lemma \ref{lem:mu-G-to-G'} relates Fourier transforms of orbital integrals on \(\Lie(G)\) and \(\Lie(G')\), but it does not give us enough information to figure out the correct general matching condition.  Our first main result, Proposition \ref{prop:dist-r-to-s+}, does provide that information; and our second, Theorem \ref{thm:dist-G-to-G'}, states the matching condition and the resulting reduction.

Lemma \ref{lem:index-to-disc-X*} is similar to Lemma \ref{lem:index-to-disc-gamma}, but involves \(q_{G/G'}^r\), which we view as a proxy for the discriminant of \(X^*\), rather than the discriminant of \(\gamma\).  (See \xcite{debacker-spice:stability}*{Proposition \xref{prop:const}} and the proof of Theorem \ref{thm:asymptotic-pi-to-pi'} for the justification of our claim that \(q_{G/G'}^r\) is a reasonable proxy.)

\begin{lem}
\label{lem:index-to-disc-X*}
We have that
\[
\frac
	{\indx{\sbtlp G x 0}{\sbtl G x s}}
	{\indx{\sbtlp{G'}x 0}{\sbtl G x s}}
\qeqq
\frac
	{\card{\sbat\fg x 0}\inv[1/2]}
	{\card{\sbat{\fg'}x 0}\inv[1/2]}
\dotm
\frac{q_G^s}{q_{G'}^s}
\dotm
\frac
	{\card{\sbat G x s}\inv[1/2]}
	{\card{\sbat{G'}x s}\inv[1/2]}
\]
\end{lem}

\begin{proof}
The proof proceeds exactly as did that of Lemma \ref{lem:index-to-disc-gamma}.  We begin by reducing to the Lie algebra, where we observe that \(\indx{\sbtlp{\Lie(G)}x 0}{\sbtl{\Lie(G)}x s}\) equals both
\begin{align*}
&\card{\sbat\fg x 0}\inv\indx{\sbtl{\Lie(G)}x 0}{\sbtl{\Lie(G)}x s} \\
\intertext{and}
&\indx{\sbtlp{\Lie(G)}x 0}{\sbtlp{\Lie(G)}x s}\dotm\card{\sbat{\Lie(G)}x s}\inv.
\end{align*}
The result then follows from \xcite{debacker-spice:stability}*{Corollary \xref{cor:gxf-inv}}.
\end{proof}

Lemma \ref{lem:mu-G-to-G'} is stated in such a way that it is independent of the choices of Haar measures on \(G\) and \(G'\), but appears still to depend, in the notation of the proof, on the choices of Haar measures on \(\Cent_{G\conn}(X^*)\) and \(\Cent_{G\primeconn}(X^*)\).  However, this choice does not matter; these groups are equal, and we may use any common Haar measure on them.

\begin{lem}
\label{lem:mu-G-to-G'}
We have that
\begin{align*}
\meas(\sbtlp G x 0)\inv&\muhat^{G\conn}_{\OO'}\bigl(\chrc{\sbtl\fg x r, \AddChar_{X^*}^\vee}\bigr) \\
\intertext{equals}
\meas(\sbtlp{G'}x 0)\inv&\muhat^{G\primeconn}_{\OO'}\bigl(\chrc{\sbtl{\fg'}x r, \AddChar_{X^*}^\vee}\bigr)
\end{align*}
for all \(\OO' \in \OO^{G\primeconn}(X^* + \sbtlpp{\Lie^*(G')}x{-r})\).
\end{lem}

\begin{proof}
Choose an element \(Y^*\) in \(\OO' \cap (X^* + \sbtlpp{\Lie^*(G')}x{-r})\).  Put
\begin{align*}
\mc C ={} & \set{g \in G\conn}{\Ad^*(g)Y^* \in X^* + \sbtlpp{\fg^*}x{-r}} \\
\intertext{and}
I ={} & \meas(\sbtlp G x 0)\inv\muhat^{G\conn}_{\Ad^*(G\conn)\OO'}\bigl(\chrc{\sbtl\fg x r, \AddChar_{X^*}^\vee}\bigr),
\end{align*}
and let \(\mc C'\) and \(I'\) be the analogous objects for \(G\primeconn\).  Because of our normalising convention for \anonchrc, we have that \(\chrc{\sbtl\fg x r, \AddChar_{X^*}^\vee}\sphat\) is the characteristic function of \(X^* + \sbtlpp{\fg^*}x{-r}\), so \(I\) equals
\begin{align*}
& \meas(\sbtlp G x 0)\inv\meas(\mc C/\Cent_{G\conn}(X^*)) \\
={} & \sum_{g \in \sbtlp G x 0\bslash\mc C/\Cent_{G\conn}(X^*)} \frac{\meas(\sbtlp G x 0 g\Cent_{G\conn}(X^*)/\Cent_{G\conn}(X^*))}{\meas(\sbtlp G x 0)} \\
={} & \sum_{g \in \sbtlp G x 0\bslash\mc C/\Cent_{G\conn}(X^*)} \meas(\Int(h)\inv\sbtlp G x 0 \cap \Cent_{G\conn}(X^*))\inv;
\end{align*}
and similarly for \(I'\).  We have that \(\Cent_{G\conn}(X^*)\) equals \(\Cent_{G\primeconn}(X^*)\) (Hypothesis \initref{hyp:X*}\subpref{orbit}), so that \(\Int(g')\inv\sbtlp G x 0 \cap \Cent_{G\conn}(X^*)\) equals \(\Int(g')\inv\sbtlp{G'}x 0 \cap \Cent_{G\primeconn}(Y^*)\) for \(g' \in \mc C'\), and the natural map \anonmap{\sbtlp{G'}x 0\bslash\mc C'}{\sbtlp G x 0\bslash\mc C} is a surjection (Lemma \initref{lem:commute-Lie*}\subpref{onto} and Hypothesis \initref{hyp:X*}\subpref{orbit} again).  In fact the map is obviously also injective and equivariant for right translation by \(\Cent_{G\conn}(X^*) = \Cent_{G\primeconn}(X^*)\), so induces a bijection of \(\sbtlp{G'}x 0\bslash\mc C'/\Cent_{G\primeconn}(X^*)\) with \(\sbtlp G x 0\bslash\mc C/\Cent_{G\conn}(X^*)\).  It follows that \(I\) equals \(I'\), as desired.
\end{proof}

Lemma \ref{lem:r-to-s+} is implicit in the proof of \xcite{adler-spice:explicit-chars}*{Proposition \xref{prop:step1-support}}.  It ``gives us room'' perpendicular to \(G'\) when sampling invariant distributions at certain test functions related to K-types.  We use this room in Proposition \ref{prop:dist-r-to-s+} to compute the sample values in terms of Gauss sums.

\begin{lem}
\label{lem:r-to-s+}
We have that
\begin{align*}
q_G\inv[s]\card{\sbat\fg x 0}^{1/2}\card{\sbat G x s}^{1/2}&\chrc{\sbtl G x r, \hat\phi^\vee} \\
\intertext{equals}
q_{G'}\inv[s]\card{\sbat{\fg'}x 0}^{1/2}\card{\sbat{G'}x s}^{1/2}&\uint_{\sbtlp G x 0} g\inv\chrc{\sbtl{(G', G)}x{(r, \Rp s)}, \hat\phi^\vee}g\,\upd g.
\end{align*}
\end{lem}

\begin{proof}
Both sides are invariant under conjugation by \(\sbtlp G x 0\), and supported by the orbit under \(\sbtlp G x 0\) of \(\sbtl{(G', G)}x{(r, \Rp s)}\).  Thus, it suffices to verify the equality on \(\sbtl{(G', G)}x{(r, \Rp s)}\).  Suppose that \(t \in \R\) satisfies \(s < t < r\), and \(k\) belongs to \(\sbtl{(G', G)}x{(r, t)}\).  We show that the integral
\[
\int_{\sbtl{(G', G)}x{(\Rpp{r - t}, r - t)}} \hat\phi(\comm h k)\upd h
\]
is \(0\) unless \(k\) belongs to \(\sbtl{(G', G)}x{(r, \Rp t)}\).  By, and with the notation of, Hypotheses \ref{hyp:MP-ad} and \ref{hyp:phi}, the integral is a multiple of
\[
\int_{\sbtl{\Lie(G', G)}x{(\Rpp{r - t}, r - t)}} \AddChar\bigl(\pair{X^*}{\comm Y Z}\bigr)\upd Y,
\]
where \(Z \in \sbtl{\Lie(G', G)}x{(r, t)}\) is such that \(k\) belongs to \(\sbat{\ol\mexp}x{(r, t)}(Z)\).  In particular, the integral is \(0\) unless
\[
\pair[\big]{X^*}{\comm Y Z} = \pair{\ad^*(Z)X^*}Y
\]
belongs to \(\sbjtlp\field 0\) for all \(Y \in \sbtl{\Lie(G', G)}x{(\Rpp{r - t}, r - t)}\), hence unless \(\ad^*(Z)X^*\) belongs to \(\sbtl{\Lie^*(G', G)}x{(t - r, \Rpp{t - r})}\).  By Lemma \initref{lem:commute-Lie*}\subpref{up}, this would imply that \(Z \in \sbtl{\Lie(G', G)}x{(r, t)}\) belonged to \(\sbtl{\Lie(G', G)}x{(-\infty, \Rp t)}\), hence to \(\sbtl{\Lie(G', G)}x{(r, \Rp t)}\); and so that \(k \in \sbat{\ol\mexp}x{(r, t)}(Z)\) belonged to \(\sbtl{(G', G)}x{(r, \Rp t)}\), as claimed.

Thus, the left-hand side is supported by \(\sbtl G x r\), so the two sides agree up to a constant.  To show that the constant is \(1\), we use Lemma \ref{lem:MP-card} and \xcite{debacker-spice:stability}*{Corollary \xref{cor:gxf-inv} and Lemma \xref{lem:gxf-shift}} to see that \(\indx{\sbtl{(G', G)}x{(r, \Rp s)}}{\sbtl G x r}\) equals \(\indx{\sbtlp G x 0}{\sbtl{(G', G)}x{(\Rp0, s)}}\).  Then Lemma \ref{lem:index-to-disc-X*} gives that
\begin{align*}
q_G\inv[s]\card{\sbat\fg x 0}^{1/2}\card{\sbat G x s}^{1/2}&\meas(\sbtl G x r)\inv \\
\intertext{equals}
q_{G'}\inv[s]\card{\sbat{\fg'}x 0}^{1/2}\card{\sbat{G'}x s}^{1/2}&\meas(\sbtl{(G', G)}x{(r, \Rp s)})\inv.\qedhere
\end{align*}
\end{proof}

Lemma \ref{lem:X*-like} gives us a transformation property of functions in a certain Hecke algebra.  (Remember that we have built in a contragredient to our Hecke-algebra notation, so that \(\Hecke(G\sslash\sbtl{(M', M)}x{(r, \Rp s)}, \hat\phi)\) stands for the space of functions that transform according to \(\hat\phi^\vee\).)  After we specialise it slightly in Corollary \ref{cor:X*-like}, it will be used in Proposition \ref{prop:Gauss-appears}.

\begin{lem}
\label{lem:X*-like}
For all \(f \in \Hecke(G\sslash\sbtl{(M', M)}x{(r, \Rp s)}, \hat\phi)\) and \(i, j \in \R_{\ge 0}\) satisfying \(i + 2j \le r\) and \(j < s\), we have that
\[
f(\Int(b h)(k\gamma))
\qeqq
\hat\phi^\vee(\comm h k)
\hat\phi^\vee(\comm{\gamma\inv}h)
f(\Int(b)(k\gamma))
\]
for all
\begin{align*}
b & {}\in \sbtlp{(G', G)}x{(0, j)}, \\
h & {}\in \sbtl{(\CC{G'}i(\gamma), \CC G i(\gamma))}x{(r - i, r - (i + j))}, \\
\intertext{and}
k & {}\in \sbtl{\CC G i(\gamma)}x{i + j}.
\end{align*}
\end{lem}

\begin{proof}
Put \(t = i + j\).  We have that
\begin{align*}
\Int(h)(k\gamma)
\qeqq
&\comm h k\dotm(k\gamma)\dotm\comm{\gamma\inv}h.
\intertext{By Proposition \ref{prop:Q-and-B}, we have that \(\comm{\gamma\inv}h\) belongs to \(\sbtl{(\CC{G'}i(\gamma), \CC G i(\gamma))}x{(r, r - j)} \subseteq \sbtl{(M', M)}x{(r, \Rp s)}\).
By Lemma \initref{lem:filtration}\subpref{gp-gp}, we have that \(\comm h k\) belongs to \(\sbtl{\CC G i(\gamma)}x r \subseteq \sbtl{(M', M)}x{(r, \Rp s)}\),
and that the commutator of an element of \(\sbtlp{(G', G)}x{(0, j)}\) with one of \(\sbtl{(G', G)}x{(r, r - j)}\) or \(\sbtl G x r\) belongs to \(\sbtlp{\Der G}x r \subseteq \ker \hat\phi\).  Thus,}
f(\Int(b h)(\gamma k))
\qeqq &
\hat\phi^\vee(\Int(b)\comm h k)
f(\Int(b)(k\gamma))
\hat\phi^\vee(\Int(b)\comm{\gamma\inv}h) \\ &\qquad=
\hat\phi^\vee(\comm h k)
\hat\phi^\vee(\comm{\gamma\inv}h)
f(\Int(b)(\gamma k)).\qedhere
\end{align*}
\end{proof}

We view Corollary \ref{cor:X*-like} as stating in a philosophical sense that we can, under certain restrictive conditions, ignore the presence of the element \(\gamma\) when conjugating.  This allows us to show that certain subintegrals in Proposition \ref{prop:Gauss-appears} vanish.

\begin{cor}
\label{cor:X*-like}
With the notation of Lemma \ref{lem:X*-like}, if the inequality \(i + 2j < r\) holds, then
\[
f(\Int(b h)(k\gamma))
\qeqq
\hat\phi^\vee(\comm h k)
f(\Int(b)(k\gamma)).
\]
\end{cor}

\begin{proof}
Proposition \ref{prop:Q-and-B} gives that \(\comm h\gamma\) belongs to \(\sbtlp{\Der(G', G)}x{(r, s)} \subseteq \ker \hat\phi\).  The result now follows from Lemma \ref{lem:X*-like}.
\end{proof}

Proposition \ref{prop:Gauss-appears} isolates an important part of the proof of \xcite{adler-spice:explicit-chars}*{Proposition \xref{prop:step1-formula1}}.  For notational convenience, we assume that the action of \(\gamma\)  on \bG is compact, and state Proposition \ref{prop:Gauss-appears} in terms of the group \(G\); but we drop the assumption on \(\gamma\), and apply the result to \(M\) instead, in Proposition \ref{prop:dist-r-to-s+}.  In the context in which we use it, the function \(f\) will be fixed by conjugation by \(\sbtl{(G', G)}x{(\Rp0, s)}\); in that sense, this result is a key part in the reduction of computations on \(G\) to computations on \(G'\).

\begin{prop}
\initlabel{prop:Gauss-appears}
If the action of \(\gamma\) on \bG is compact, then, for any \(f \in \Hecke(G\sslash\sbtl{(G', G)}x{(r, \Rp s)}, \hat\phi)\), we have that
\begin{align*}
&q_{G/H}^s\abs{\Disc_{G/H}(\gamma)}^{1/2}
\indx{\sbat\fg x 0}{\sbat\fh x 0}\inv[1/2]
\indx{\sbat G x s}{\sbat{\CCp G 0(\gamma)}x s}\inv[1/2]\times{} \\
&\qquad\Gauss_{\CCp G 0(\gamma)/H}(X^*, \gamma)\inv
\uint_{\sbtlp G x 0} f(\Int(g)\gamma)\upd g \\
\intertext{equals}
&q_{G'/H'}^s\abs{\Disc_{G'/H'}(\gamma)}^{1/2}
\indx{\sbat{\fg'}x 0}{\sbat{\fh'}x 0}\inv[1/2]
\indx{\sbat{G'}x s}{\sbat{\CCp{G'}0(\gamma)}x s}\inv[1/2]\times{} \\
&\qquad\Gauss_{\CCp{G'}0(\gamma)/H'}(X^*, \gamma)\inv
\uint_{\sbtl{(G', G)}x{(\Rp0, s)}} f(\Int(j)\gamma)\upd j.
\end{align*}
\end{prop}

\begin{proof}
For this proof, put \(\dc^\perp = \sbtl{(H, \CCp{G'}0(\gamma), \CCp G 0(\gamma))}x{(\Rp0, r - \ord_\gamma, (r - \ord_\gamma)/2)}\).  (This is closely related, but not identical, to the notation of \xcite{adler-spice:explicit-chars}*{\S\xref{sec:normal}}.)  The main idea is that the integral over \(\sbtlp G x 0\) is unchanged (aside from normalisation issues) if we take it only over \(\dc^\perp\).  To show this, we chop the remainder of the domain of integration into shells, on each of which we use Corollary \ref{cor:X*-like} to show that the integral vanishes.

Suppose that \(i, j \in \R_{\ge 0}\) satisfy \(i + 2j < r\), and put
\begin{align*}
S_{i j} ={} & \sbtl{(H, G', \CCp G{r - 2j}(\gamma), \CC G i(\gamma), G)}x{(\Rp0, \Rp0, (r - \ord_\gamma)/2, j, \Rp j)} \setminus{} \\
            & \qquad\sbtl{(H, G', \CCp G{r - 2j}(\gamma), \CCp G i(\gamma), G)}x{(\Rp0, \Rp0, (r - \ord_\gamma)/2, j, \Rp j)}.
\end{align*}
(In the notation of \xcite{adler-spice:explicit-chars}*{\S\xref{sec:normal}}, we have for any \(g \in S_{i j}\) that \(i^\perp(g)\) is \(i\) and \(j^\perp(g)\) is \(j\).)  We claim that \(I_{i j} \ldef \displaystyle\int_{S_{i j}} f(\Int(g)\gamma)\upd g\) equals \(0\).

Put \(t = i + j\), and
\begin{align*}
C    & {}= \sbtl{(H, \CCp G{r - 2j}(\gamma), \CC G i(\gamma))}x{(\Rp0, (r - \ord_\gamma)/2, j)}, \\
C_+  & {}= \sbtl{(H, \CCp G{r - 2j}(\gamma), \CCp G i(\gamma), \CC{G'}i(\gamma), \CC G i(\gamma))}x{(\Rp0, (r - \ord_\gamma)/2, j, j, \Rp j)}, \\
\intertext{and}
\mc H & {}= \sbtl{(\CC{G'}i(\gamma), \CC G i(\gamma))}x{(r - i, r - t)}.
\end{align*}

For this paragraph, fix an element \(c\) of \(\sbtl{(\CC{G'}i(\gamma), \CC G i(\gamma))}x{(\Rp j, j)}\), and re-adopt Notations \ref{notn:Gauss} and \ref{notn:Q-and-B}.  For later use, we show that the integral
\[
\int_{\mc H} \hat\phi(\mf B_\gamma(h, c))\upd h
\]
is \(0\) unless \(c\) belongs to \(C_+\).  Our argument is very similar to the proof of Lemma \ref{lem:r-to-s+}.  By, and with the notation of, Hypotheses \ref{hyp:MP-ad} and \ref{hyp:phi}, we have that the integral is a multiple of
\[
\int_{\sbtl{\Lie(\CC{G'}i(\gamma), \CC G i(\gamma))}x{(r - i, r - t)}}
	\AddChar(b_{X^*, \gamma}(Y, Z))\upd Y,
\]
where \(Z \in \sbtl{(\CC{G'}i(\gamma), \CC G i(\gamma))}x{(\Rp j, j)}\) is such that \(c\) belongs to \(\sbat{\ol\mexp}x{(\Rp j, j)}(Z)\).  In particular, it is \(0\) unless
\[
b_{X^*, \gamma}(Y, Z) = \pair[\big]{\ad^*((1 - \Ad(\gamma))Z)X^*}Y
\]
belongs to \(\sbjtlp\field 0\) for all \(Y \in \sbtl{\Lie(\CC{G'}i(\gamma), \CC G i(\gamma))}x{(r - i, r - t)}\), hence unless \(\ad^*((1 - \Ad(\gamma))Z)X^*\) belongs to \(\sbtlp{\Lie^*(\CC{G'}i(\gamma), \CC G i(\gamma))}x{(i - r, t - r)}\).  We now reason as in Proposition \ref{prop:lattice-orth}.  If this containment held, then we would have by Lemma \initref{lem:commute-Lie*}\subpref{up} that \((1 - \Ad(\gamma))Z\) belonged to \(\sbtlp{\Lie(G', \CC G i(\gamma), G)}x{(-\infty, t, -\infty)}\).  This in turn would imply by Lemma \initref{lem:commute-gp}\subpref{up} that \(Z \in \sbtl{\Lie(\CC{G'}i(\gamma), \CC G i(\gamma))}x{(\Rp j, j)}\) belonged to \(\sbtlp{\Lie^*(G', \CCp G i(\gamma), \CC G i(\gamma), G)}x{(-\infty, -\infty, j, -\infty)}\), hence to \(\sbtl{\Lie(\CC{G'}i(\gamma), \CCp G i(\gamma), \CC G i(\gamma))}x{(\Rp j, j, \Rp j)}\); and so that \(c \in \sbat{\ol\mexp}x{(\Rp j, j)}(Z)\) belonged to \(\sbtl{(\CC{G'}i(\gamma), \CCp G i(\gamma), \CC G i(\gamma))}x{(\Rp j, j, \Rp j)} \subseteq C_+\).

Since \(S_{i j}\) equals \(\sbtlp{(G', G)}x{(0, j)}\dotm(C \setminus C_+)\), we have by \xcite{adler-spice:explicit-chars}*{Lemma \xref{lem:silly-integration}} that \(I_{i j}\) is a multiple of
\[
\int_{C \setminus C_+} \int_{\sbtlp{(G', G)}x{(0, j)}} f(\Int(b c)\gamma)\upd b\,\upd c.
\]
By Proposition \ref{prop:Q-and-B}, we have that \(\mf Q_\gamma\) is \(\sbtl{\CC G i(\gamma)}x t\)-valued on \(C\); and, if we write \(\ol{\mf B}_\gamma\) for the composition of \(\mf B_\gamma\) with the projection \anonmap G{G/\sbtlp{\Der(G', G)}x{(r, s)}}, so that the composition \(\hat\phi \circ \mf B_\gamma\) factors through \(\ol{\mf B}_\gamma\), then \(\ol{\mf B}_\gamma\) is bi-multiplicative on \(\mc H \times C\) and trivial on \(\mc H \times \sbtl{(H, \CCp G{r - 2j}(\gamma))}x{(\Rp0, (r - \ord_\gamma)/2)}\).  Since \(\mc H = \sbtl{(\CC{G'}i(\gamma), \CC G i(\gamma))}x{(r - i, r - t)}\) is contained in \(\sbtlp{(G', G)}x{(0, j)}\), we have by Corollary \ref{cor:X*-like} that
\begin{multline*}
\int_{\sbtlp{(G', G)}x{(0, j)}} \int_{C \setminus C_+} f(\Int(b c)\gamma)\upd c\,\upd b \\
= \sum_c \int_{\sbtlp{(G', G)}x{(0, j)}} \int_{C_+}
	\int_{\mc H} f(\Int(b h)(\mf Q_\gamma(c c_+)\gamma))\upd h\,
\upd{c_+}\,\upd b,
\end{multline*}
where the sum over \(c\) runs over the non-trivial cosets in the quotient of \(\sbtl{(\CC{G'}i(\gamma), \CC G i(\gamma))}x{(\Rp j, j)}\) by its intersection \(\sbtl{(\CC{G'}i(\gamma), \CCp G i(\gamma), \CC G i(\gamma))}x{(\Rp j, j, \Rp j)}\) with \(C_+\), equals
\begin{multline*}
\sum_c \int_{\sbtlp{(G', G)}x{(0, j)}} \int_{C_+} f(\Int(b c c_+)\gamma)
	\int_{\mc H} \hat\phi(\mf B_\gamma(h, c c_+))\upd h\,
\upd{c_+}\,\upd b \\
= \Bigl(\int_{\sbtlp{(G', G)}x{(0, j)}} \int_{C \setminus C_+} f(\Int(b c c_+)\gamma)\upd{c_+}\,\upd b\Bigr)
\Bigl(\int_{\mc H} \hat\phi(\mf B_\gamma(h, c))\upd h\Bigr) = 0,
\end{multline*}
so that \(I_{i j} = \displaystyle\int_{S_{i j}} f(\Int(g)\gamma)\upd g\) equals \(0\), as desired.

Recall that we have put \(\dc^\perp = \sbtl{(H, \CCp{G'}0(\gamma), \CCp G 0(\gamma))}x{(\Rp0, r - \ord_\gamma, (r - \ord_\gamma)/2)}\).  Since \(\sbtlp G x 0 \setminus \bigcup_{\substack{i, j \ge 0 \\ i + 2j < r}} S_{i j}\) equals \(\sbtl{(H, G', G)}x{(\Rp0, \Rp0, (r - \ord_\gamma)/2)} = \sbtl{(G', G)}x{(\Rp0, s)}\dotm\dc^\perp\), we have shown (by another application of \xcite{adler-spice:explicit-chars}*{Lemma \xref{lem:silly-integration}}) that
\begin{equation}
\tag{$*$}
\sublabel{eq:dc}
\begin{aligned}
&\meas(\sbtl{(H', \CCp{G'}0(\gamma), \CCp G 0(\gamma))}x{(\Rp0, r - \ord_\gamma, s)})\int_{\sbtlp G x 0} f(\Int(g)\gamma)\upd g \\
&\qeqq\int_{\sbtl{(G', G)}x{(\Rp0, s)}} \int_{\dc^\perp} f(\Int(j v)\gamma)\upd v\,\upd j.
\end{aligned}
\end{equation}
By Proposition \ref{prop:Q-and-B}, we have that
\begin{align*}
\comm v\gamma               && \text{belongs to} && \sbtl{(G', G)}x{(r, \Rp s)}, \\
\intertext{and}
\comm[\big]j{\comm v\gamma} && \text{belongs to} && \sbtlp{\Der(G', G)}x r \subseteq \ker \hat\phi
\end{align*}
for all \(v \in \dc^\perp\) and \(j \in \sbtl{(G', G)}x{(0, \Rp s)}\).  Thus, upon normalising the integrals in \loceqref{eq:dc} by dividing by the measure, we find that
\begin{align*}
&\indx{\sbtlp G x 0}{\sbtl{(G', G)}x{(\Rp0, s)}}\uint_{\sbtlp G x 0} f(\Int(g)\gamma)\upd g \\
\intertext{equals}
&\indx[\big]{\dc^\perp}{\sbtl{(H', \CCp{G'}0(\gamma), \CCp G 0(\gamma))}x{(\Rp0, r - \ord_\gamma, s)}}\times{} \\
&\qquad\uint_{\dc^\perp} \hat\phi^\vee(\comm v\gamma)\upd v\times{} \\
&\qquad\uint_{\sbtl{(G', G)}x{(\Rp0, s)}} f(\Int(j)\gamma)\upd j.
\end{align*}
Thus, it suffices to show that
\begin{align*}
&\frac
	{\indx[\big]{\dc^\perp}{\sbtl{(H', \CCp{G'}0(\gamma), \CCp G 0(\gamma))}x{(\Rp0, r - \ord_\gamma, s)}}}
	{\indx{\sbtlp G x 0}{\sbtl{(G', G)}x{(\Rp0, s)}}}\times{} \\
&\qquad\uint_{\dc^\perp} \hat\phi^\vee(\comm v\gamma)\upd v \\
\intertext{equals}
&\frac{\abs{\Disc_{G/H}(\gamma)}^{1/2}}{\abs{\Disc_{G'/H'}(\gamma)}^{1/2}}\dotm
\frac{\indx{\sbat\fg x 0}{\sbat\fh x 0}\inv[1/2]}{\indx{\sbat{\fg'}x 0}{\sbat{\fh'}x 0}\inv[1/2]}\dotm
\frac{\indx{\sbat G x s}{\sbat{\CCp G 0(\gamma)}x s}\inv[1/2]}{\indx{\sbat{G'}x s}{\sbat{\CCp{G'}0(\gamma)}x s}\inv[1/2]}\times{} \\
&\qquad\frac{\Gauss_{\CCp G 0(\gamma)/H}(X^*, \gamma)\inv}{\Gauss_{\CCp{G'}0(\gamma)/H'}(X^*, \gamma)\inv}\dotm
\frac{q_{G/H}^s}{q_{G'/H'}^s}.
\end{align*}
We have by Lemma \ref{lem:index-to-disc-X*} that \(\indx{\sbtlp G x 0}{\sbtl{(G', G)}x{(\Rp0, s)}}\) equals \(\indx{\sbat\fg x 0}{\sbat{\fg'}x 0}\inv[1/2]q_{G/G'}^s\indx{\sbat G x s}{\sbat{G'}x s}\inv[1/2]\), and by Lemma \ref{lem:index-to-disc-gamma} that
\begin{multline*}
\indx[\big]{\dc^\perp}{\sbtl{(H', \CCp{G'}0(\gamma), \CCp G 0(\gamma))}x{(\Rp0, r - \ord_\gamma, s)}} \\
{}= \indx[\big]{\sbtl{(H, \CCp{G'}0(\gamma), \CCp G 0(\gamma))}x{(\Rp0, \Rp0, (r - \ord_\gamma)/2)}}{\sbtl{(\CCp{G'}0(\gamma), \CCp G 0(\gamma))}x{(\Rp0, s)}},
\end{multline*}
which may be re-written as
\[
\frac
	{\indx{\sbtl{(H, \CCp G 0(\gamma))}x{(\Rp0, (r - \ord_\gamma)/2)}}{\sbtl{(H, \CCp G 0(\gamma))}x{(\Rp0, s)}}}
	{\indx{\sbtl{(H', \CCp{G'}0(\gamma))}x{(\Rp0, (r - \ord_\gamma)/2)}}{\sbtl{(H', \CCp{G'}0(\gamma))}x{(\Rp0, s)}}}\dotm
\indx{\sbtlp H x 0}{\sbtl{(H', H)}x{(\Rp0, s)}},
\]
equals
\begin{align*}
&\frac{
	\card{\sbat{(H, \CCp G 0(\gamma))}x{(\Rp0, (r - \ord_\gamma)/2)}}^{1/2}
	\abs{\Disc_{G/H}(\gamma)}\inv[1/2]
	\indx{\sbat{\CCp G 0(\gamma)}x s}{\sbat H x s}\inv[1/2]
}
{
	\card{\sbat{(H', \CCp{G'}0(\gamma))}x{(\Rp0, (r - \ord_\gamma)/2)}}^{1/2}
	\abs{\Disc_{G'/H'}(\gamma)}\inv[1/2]
	\indx{\sbat{\CCp{G'}0(\gamma)}x s}{\sbat{H'}x s}\inv[1/2]
}\times{} \\
&\qquad\indx{\sbat\fh x 0}{\sbat{\fh'}x 0}\inv[1/2]
q_{H/H'}^s
\indx{\sbat H x s}{\sbat{H'}x s}\inv[1/2] \\ {}=
&\card{\sbat{(H, \CCp{G'}0(\gamma), \CCp G 0(\gamma))}x{(\Rp0, \Rp0, (r - \ord_\gamma)/2)}}^{1/2}\times{} \\
&\qquad\frac{\abs{\Disc_{G/H}(\gamma)}\inv[1/2]}{\abs{\Disc_{G'/H'}(\gamma)}\inv[1/2]}
\indx{\sbat{\CCp G 0(\gamma)}x s}{\sbat{\CCp{G'}0(\gamma)}x s}\inv[1/2]
\indx{\sbat\fh x 0}{\sbat{\fh'}x 0}\inv[1/2]
q_{H/H'}^s.
\end{align*}
(We have used that \(\gamma\) is compact, so that
\[
\Disc_{G/\CCp G 0(\gamma)}(\gamma) = \det_{\Lie(G)/\Lie(\CCp G 0(\gamma))}(\Ad(\gamma) - 1)
\]
is a unit (by Hypothesis \initref{hyp:funny-centraliser}\subpref{Lie}), and similarly for \(G'\).)  The result now follows from Proposition \ref{prop:Gauss-to-Weil}, with an inverse on \(\Gauss\) because we are dealing with \(\phi^\vee\) rather than \(\phi\).
\end{proof}

So far, in Lemma \ref{lem:r-to-s+} and Proposition \ref{prop:Gauss-appears}, we have been dealing directly with functions on \(G\).  We want to combine these results; but, while the latter can handle the values of such functions near \(\gamma\), the former can only handle their values near the identity.  Proposition \ref{prop:dist-r-to-s+} circumvents this difficulty by dealing, not with the functions themselves, but with the values of invariant distributions at them.

\begin{prop}
\initlabel{prop:dist-r-to-s+}
If \(T\) is an invariant distribution on \(G\), then
\begin{align*}
&q_H\inv[s]\frac{\abs{\Disc_{G/H}(\gamma)}^{1/2}}{\modulus_{P^-}(\gamma)^{1/2}}
\card{\sbat\fh x 0}^{1/2}
\card{\sbat{\CCp G 0(\gamma)}x s}^{1/2}\times{} \\
&\qquad\Gauss_{\CCp G 0(\gamma)/H}(X^*, \gamma)\inv
T\bigl(\gamma\chrc{\sbtl G x r, \hat\phi^\vee}\bigr) \\
\intertext{equals}
&q_{H'}\inv[s]\frac{\abs{\Disc_{G'/H'}(\gamma)}^{1/2}}{\modulus_{P\suppm}(\gamma)^{1/2}}
\card{\sbat{\fh'}x 0}^{1/2}
\card{\sbat{\CCp{G'}0(\gamma)}x s}^{1/2}\times{} \\
&\qquad\Gauss_{\CCp{G'}0(\gamma)/H'}(X^*, \gamma)\inv
T\bigl(\gamma\chrc{\sbtl{(G', G)}x{(r, \Rp s)}, \hat\phi^\vee}\bigr).
\end{align*}
\end{prop}

\begin{proof}
Consider the function \map f G\C given by
\[
f(g) = T\bigl(g\chrc{\sbtl{(M', M, G)}x{(r, \Rp s, r)}, \hat\phi^\vee}\bigr)
\]
for all \(g \in G\).  By invariance of \(T\), we have that \(f(g)\) equals
\[
T\bigl(\chrc{\sbtl{(M', M, G)}x{(r, \Rp s, r)}, \hat\phi^\vee}\dotm g\chrc{\sbtl{(M', M, G)}x{(r, \Rp s, r)}, \hat\phi^\vee}\bigr)
\]
for all \(g \in G\), so that \(f\) belongs to \(\Hecke(G\sslash\sbtl{(M', M, G)}x{(r, \Rp s, r)}, \hat\phi)\); and also that \(f(g\gamma g\inv)\) equals
\[
T\bigl(\gamma\dotm g\inv\chrc{\sbtl{(M', M, G)}x{(r, \Rp s, r)}, \hat\phi^\vee}g\bigr)
\]
for all \(g \in G\).  Since Lemma \initref{lem:filtration}\subpref{gp-gp} gives that the commutator of \(\sbtl{(M', M)}x{(\Rp0, s)}\) with \(\sbtl{(M', M, G)}x{(r, \Rp s, r)}\) is contained in \(\sbtlp{\Der G}x r \subseteq \ker \hat\phi\), we have that \(\sbtl{(M', M)}x{(\Rp0, s)}\) stabilises \((\sbtl{(M', M, G)}x{(r, \Rp s, r)}, \hat\phi)\), hence fixes \(f\).  Thus Proposition \ref{prop:Gauss-appears}, with \bM in place of \bG, gives that
\begin{equation}
\tag{$*$}
\sublabel{eq:Gauss-appears}
\begin{aligned}
&q_{M/H}^s\frac{\abs{\Disc_{G/H}(\gamma)}^{1/2}}{\modulus_{P^-}(\gamma)^{1/2}}
\indx{\sbat\fm x 0}{\sbat\fh x 0}\inv[1/2]
\indx{\sbat M x s}{\sbat{\CCp G 0(\gamma)}x s}\inv[1/2]\times{} \\
&\qquad\Gauss_{\CCp G 0(\gamma)/H}(X^*, \gamma)\inv
\uint_{\sbtlp M x 0} f(\Int(g)\gamma)\upd g \\
&\text{equals} \\
&q_{M'/H'}^s\frac{\abs{\Disc_{G'/H'}(\gamma)}^{1/2}}{\modulus_{P\suppm}(\gamma)^{1/2}}
\indx{\sbat{\fm'}x 0}{\sbat{\fh'}x 0}\inv[1/2]
\indx{\sbat{M'}x s}{\sbat{\CCp{G'}0(\gamma)}x s}\inv[1/2]\times{} \\
&\qquad\Gauss_{\CCp{G'}0(\gamma)/H}(X^*, \gamma)\inv
f(\gamma).
\end{aligned}
\end{equation}
Now \(\displaystyle\uint_{\sbtlp M x 0} f(\Int(g)\gamma)\) equals
\[
T\Bigl(\gamma\dotm\uint_{\sbtlp M x 0} g\inv\chrc{\sbtl{(M', M, G)}x{(r, \Rp s, r)}, \hat\phi^\vee}g\,\upd g\Bigr).
\]

Since \(\chrc{\sbtl{(M', M, G)}x{(r, \Rp s, r)}, \hat\phi^\vee}\) agrees on \(\sbtl{(M', M)}x{(r, \Rp s)}\) with
\[
\frac
	{\meas(\sbtl{(M', M)}x{(r, \Rp s)})}
	{\meas(\sbtl{(M', M, G)}x{(r, \Rp s, r)})}
\chrc{\sbtl{(M', M)}x{(r, \Rp s)}, \hat\phi^\vee},
\]
we have by Lemma \ref{lem:r-to-s+} (again applied to \bM instead of \bG) that
\begin{equation}
\sublabel{eq:that-fun}
\tag{$\dag$}
q_M^s
\card{\sbat\fm x 0}\inv[1/2]
\card{\sbat M x s}\inv[1/2]
\uint_{\sbtlp M x 0} g\inv\chrc{\sbtl{(M', M, G)}x{(r, \Rp s, r)}, \hat\phi^\vee}g\,\upd g
\end{equation}
agrees on \(\sbtlp M x s\) with
\begin{equation}
\sublabel{eq:that-mult}
\tag{$\ddag$}
q_{M'}^s
\card{\sbat{\fm'}x 0}\inv[1/2]
\card{\sbat{M'}x s}\inv[1/2]
\end{equation}
times
\[
\frac
	{\meas(\sbtl{(M', M)}x{(r, \Rp s)}}
	{\meas(\sbtl{(M', M, G)}x{(r, \Rp s, r)})}
\chrc{\sbtl M x r, \hat\phi^\vee}
\]
(where the unexpected quotient of measures comes from our normalisation convention for \anonchrc, which involves dividing by the measure of the domain); hence, since \loceqref{eq:that-fun} belongs to \(\Hecke(\sbtl{(M, G)}x{(\Rp s, r)}\sslash\sbtl G x r, \hat\phi)\), that it equals the same multiple \loceqref{eq:that-mult} of
\begin{align*}
&&\frac
	{\meas(\sbtl{(M', M)}x{(r, \Rp s)}}
	{\meas(\sbtl{(M', M, G)}x{(r, \Rp s, r)})}
\frac
	{\meas(\sbtl G x r)}
	{\meas(\sbtl M x r)}
&\chrc{\sbtl G x r, \hat\phi^\vee} \\ ={}
&&\frac
	{\indx{\sbtl{(M', M)}x{(r, \Rp s)}}{\sbtl M x r}}
	{\indx{\sbtl{(M', M, G)}x{(r, \Rp s, r)}}{\sbtl G x r}}
&\chrc{\sbtl G x r, \hat\phi^\vee} \\ ={}
&&&\chrc{\sbtl G x r, \hat\phi^\vee}.
\end{align*}
That is, by \loceqref{eq:Gauss-appears}, and remembering that \(f(\gamma)\) equals \(T\bigl(\gamma\chrc{\sbtl{(M', M, G)}x{(r, \Rp s)}, \hat\phi}\bigr)\), we have shown that
\begin{align*}
&q_H\inv[s]\frac{\abs{\Disc_{G/H}(\gamma)}^{1/2}}{\modulus_{P^-}(\gamma)^{1/2}}
\card{\sbat\fh x 0}^{1/2}
\card{\sbat{\CCp G 0(\gamma)}x s}^{1/2}\times{} \\
&\qquad\Gauss_{\CCp G 0(\gamma)/H}(X^*, \gamma)\inv
T\bigl(\gamma\chrc{\sbtl G x r, \hat\phi^\vee}\bigr)
\intertext{equals}
&q_{H'}\inv[s]\frac{\abs{\Disc_{G'/H'}(\gamma)}^{1/2}}{\modulus_{P\suppm}(\gamma)^{1/2}}
\card{\sbat{\fh'}x 0}^{1/2}
\card{\sbat{\CCp{G'}0(\gamma)}x s}^{1/2}\times{} \\
&\qquad\Gauss_{\CCp{G'}0(\gamma)/H}(X^*, \gamma)\inv
T\bigl(\gamma\chrc{\sbtl{(M', M, G)}x{(r, \Rp s, r)}, \hat\phi^\vee}\bigr).
\end{align*}
The result now follows from Lemma \ref{lem:centre}, with \(0\) in place of \(r\).  (See Remark \ref{rem:gamma}.)
\end{proof}

Theorem \ref{thm:dist-G-to-G'} gives a condition under which an invariant distribution \(T\) on \(G\) may be said to match an invariant distribution \(T'\) on \(G'\), and shows that, in this case, certain sample values of those two distributions are equal.  Note the resemblance of the proof to the descent arguments appearing in \cite{jkim-murnaghan:gamma-asymptotic}*{\S\S6.2, 7.2}.  We have phrased Theorem \ref{thm:dist-G-to-G'} in a way that we think will be amenable to future calculations using Hecke-algebra isomorphisms, but, for now, the limited analogue that we prove in Theorem \ref{thm:isotypic-pi-to-pi'} suffices to handle the characters of tame supercuspidal representations in Theorem \ref{thm:asymptotic-pi-to-pi'}.
We make an essentially cosmetic re-statement of the the theorem later, as Lemma \ref{lem:dist-G-to-G'}, in a form that is amenable to use in our main result, Theorem \ref{thm:asymptotic-pi-to-pi'}.

An analysis of the proof of Theorem \ref{thm:dist-G-to-G'} shows that we only need to refer explicitly to the measures on \(H\conn\) and \(H\primeconn\), not on \(G\) and \(G'\).  (Recall, though, that the functions \(\chrc{\sbtl{G'}x r, \phi^\vee}\) and \(\chrc{\sbtl{(G', G)}x{(r, \Rp s)}, \hat\phi^\vee}\) implicitly depend on the choices of measure on \(G'\) and \(G\), respectively.)  This is the first place that we need to make use of the specific normalisation of Haar measure chosen in \S\ref{sec:p-adic-group-basics}.

\begin{thm}
\initlabel{thm:dist-G-to-G'}
Suppose that
	\begin{itemize}
	\item \(T'\) is an invariant distribution on \(G'\),
	\item \(c(T', \gamma)\) is a finitely supported, \(\OO^{H\primeconn}(\Lie^*(H'))\)-indexed vector of complex numbers such that \(\abs{\Disc_{G'/H'}(\gamma)}^{1/2}T'(\gamma\chrc{\sbtl{G'}x r, \phi^\vee})\) equals
\[
q_{H'}^s\frac{\Gauss_{G'/H'}(X^*, \gamma)}{\Gauss_{\CCp{G'}0(\gamma^2)/\CCp{G'}0(\gamma)}(X^*, \gamma)}
\sum_{\OO' \in \OO^{H\primeconn}(\Lie^*(H'))}
	c_{\OO'}(T', \gamma)\muhat^{H\primeconn}_{\OO'}\bigl(\chrc{\sbtl{\fh'}x r, \AddChar_{X^*}}\bigr),
\]
and	\item \(T\) is an invariant distribution on \(G\) such that
\begin{align*}
&\modulus_{P^-}(\gamma)^{1/2}
\card{\sbat{\CCp G 0(\gamma)}x s}\inv[1/2]
\Gauss_{G/\CCp G 0(\gamma^2)}(X^*, \gamma)\inv
T\bigl(\gamma\chrc{\sbtl{(G', G)}x{(r, \Rp s)}, \hat\phi^\vee}\bigr) \\
\intertext{equals}
&\modulus_{P\suppm}(\gamma)^{1/2}
\card{\sbat{\CCp{G'}0(\gamma)}x s}\inv[1/2]
\Gauss_{G'/\CCp{G'}0(\gamma^2)}(X^*, \gamma)\inv
T'\bigl(\gamma\chrc{\sbtl{G'}x r, \hat\phi^\vee}\bigr).
\end{align*}
	\end{itemize}
Then \(\abs{\Disc_{G/H}(\gamma)}^{1/2}T(\gamma\chrc{\sbtl G x r, \hat\phi^\vee})\) equals
\[
q_H^s
\frac{\Gauss_{G/H}(X^*, \gamma)}{\Gauss_{\CCp G 0(\gamma^2)/\CCp G 0(\gamma)}(X^*, \gamma)}
\sum_{\OO' \in \OO^{H\primeconn}(\Lie^*(H'))}
	c_{\OO'}(T', \gamma)\muhat^{H\conn}_{\OO'}\bigl(\chrc{\sbtl\fh x r, \AddChar_{X^*}}\bigr).
\]
\end{thm}

Note that \(\muhat^{H\primeconn}_{\OO'}\) and \(\muhat^{H\conn}_{\OO'}\) are Fourier transforms of different orbital integrals, as indicated by the superscripts, even though the subscripts are the same.

It does not matter for Theorem \ref{thm:dist-G-to-G'} whether \(X^* + \sbtlpp{\Lie^*(G')}x{-r}\) is the ``dual blob'' of \((\sbtl{G'}x r, \phi)\), only that Hypothesis \ref{hyp:phi} is satisfied.

\begin{proof}
By Proposition \ref{prop:dist-r-to-s+}, we have that \(q_H\inv[s]\abs{\Disc_{G/H}(\gamma)}^{1/2}T\bigl(\gamma\chrc{\sbtl G x r, \AddChar_{X^*}}\bigr)\) equals
\begin{equation}
\sublabel{eq:sample-G}
\tag{$*$}
\begin{aligned}
&q_{H'}\inv[s]
\frac{\Gauss_{G/H}(X^*, \gamma)}{\Gauss_{G'/H'}(X^*, \gamma)}
\indx{\sbat\fh x 0}{\sbat{\fh'}x 0}\inv[1/2]\times{} \\
&\qquad\abs{\Disc_{G'/H'}(\gamma)}^{1/2}
\frac{\Gauss_{\CCp G 0(\gamma^2)/\CCp G 0(\gamma)}(X^*, \gamma)\inv}{\Gauss_{\CCp{G'}0(\gamma^2)/\CCp{G'}0(\gamma)}(X^*, \gamma)\inv}\times{} \\
&\qquad\underbrace{
	\frac{\modulus_{P^-}(\gamma)^{1/2}}{\modulus_{P\suppm}(\gamma)^{1/2}}
	\indx{\sbat{\CCp G 0(\gamma)}x s}{\sbat{\CCp{G'}0(\gamma)}x s}\inv[1/2]
	\frac{\Gauss_{G/\CCp G 0(\gamma^2)}(X^*, \gamma)\inv}{\Gauss_{G'/\CCp{G'}0(\gamma^2)}(X^*, \gamma)\inv}
}_{(G')}\times{} \\
&\qquad\underbrace{
	T\bigl(\gamma\chrc{\sbtl{(G', G)}x{(r, \Rp s)}, \hat\phi^\vee}\bigr)
}_{(G')}.
\end{aligned}
\end{equation}
By assumption, we have that (\(\text{\locref{eq:sample-G}}_{G'}\)) equals \(T'\bigl(\gamma\chrc{\sbtl{G'}x r, \AddChar_{X^*}}\bigr)\), so that, again by assumption, \(q_H\inv[s]\abs{\Disc_{G/H}(\gamma)}^{1/2}T\bigl(\gamma\chrc{\sbtl G x r, \AddChar_{X^*}}\bigr)\) itself equals
\begin{align*}
&\frac{\Gauss_{G/H}(X^*, \gamma)}{\Gauss_{\CCp G 0(\gamma^2)/\CCp G 0(\gamma)}(X^*, \gamma)}\times{} \\
&\qquad\indx{\sbat\fh x 0}{\sbat{\fh'}x 0}\inv[1/2]\sum_{\OO' \in \OO^{H\primeconn}(\Lie^*(H')} c_{\OO'}(T', \gamma)\muhat_{\OO'}^{H\primeconn}\bigl(\chrc{\sbtl{\fh'}x r, \AddChar_{X^*}}\bigr).
\end{align*}
According to Waldspurger's canonical Haar measures, we have that \(\indx{\sbat\fh x 0}{\sbat{\fh'}x 0}\inv[1/2]\) equals \(\frac{\meas(\sbtlp{H'}x 0)\inv}{\meas(\sbtlp H x 0)\inv}\), so that the result follows from Lemma \ref{lem:mu-G-to-G'}.
\end{proof}

\subsection{Supercuspidal representations}
\label{sec:cuspidal}

After Lemma \ref{lem:rho-isotypic}, we will begin to re-introduce notation from the previous sections; but we begin this section without any of the accumulated notation (except for \bG itself), by recalling the construction of tame supercuspidal representations from \cite{yu:supercuspidal}.  Because the restriction is necessary for Yu's construction, we assume now that \bG is \tamefield-split and connected.

The construction is inductive in nature, and we consider only a single step in the induction; so we replace the datum of \cite{yu:supercuspidal}*{\S3, p.~590} with a quintuple \(((\bG', \bG), o, (0 < r \le r_d), \rho', (\phi_o, \chi))\), which we fix for the rest of the paper.  Here,
	\begin{itemize}
	\item our \(\bG'\) is Yu's \(\bG^{d - 1}\);
	\item our \(o\) is Yu's \(y\);
	\item our \(r\) is Yu's \(r_{d - 1}\);
	\item our \(\rho'\) is the induction up to \(\stab_{G'}(\ol o)\) of the representation \(\rho_{d - 1}\) constructed in \cite{yu:supercuspidal}*{\S4, p.~592};
and	\item our \(\phi_o\) and \(\chi\) are Yu's \(\phi_{d - 1}\) and \(\phi_d\), respectively.
	\end{itemize}
Note that \(\rho'\) is constructed by inducing \(\rho_{d - 1}\), \emph{not} \(\rho'_{d - 1}\); that is, it includes the twist by \(\phi_o\) by which those two representations differ.  The group \(\bG'\) and positive real number \(r\) here will be the same as in \S\ref{sec:nearly-good} and \S\ref{sec:depth-matrix}, respectively.  The idea is that the part of Yu's datum that he denotes by \(((\bG^0 \subsetneq \dotsb \subsetneq \bG^{d - 1}), y, (r_0 < \dotsc < r_{d - 1}), \rho_0', (\phi_0, \dotsc, \phi_{d - 1}))\) has already been used to construct \(\rho'\), which is \((\sbtl{G'}o r, \phi_o)\)-isotypic and induces an irreducible, hence supercuspidal, depth-\(r\) representation of \(G'\), which we call \(\mnotn{\pi'} \ldef \Ind_{\stab_{G'}(\ol o)}^{G'} \rho'\).

Let \mnotn{Z^*_o} be a generic element, as in \cite{yu:supercuspidal}*{\S8}, representing \((\sbtl{G'}o r, \phi_o)\), as in \cite{yu:supercuspidal}*{\S9}.  Because \bG is connected, Hypothesis \ref{hyp:Z*} is just the genericity condition \cite{yu:supercuspidal}*{\S8, p.~596, \textbf{GE}}.  Because \bG is tame, so that Moy--Prasad isomorphisms are always available, Hypothesis \ref{hyp:K-type} follows from the definition of ``representing'' \cite{yu:supercuspidal}*{\S5}.

\begin{notn}
\label{notn:Yu}
Put
\begin{align*}
\mnotn{K'_o}               & {}= \stab_{G'}(\ol o), \\
\matnotn{Jo}{\oJgp o}      & {}= \sbtl{(G', G)}o{(r, s)}, \\
\mnotn{K_o}                & {}= K'_o\dotm\oJgp o = \stab_{G'}(\ol o)\sbtl G o s, \\
\intertext{and}
\matnotn{Joplus}{\oJpgp o} & {}= \sbtl{(G', G)}o{(r, \Rp s)}. \\
\intertext{For uniformity of notation, we also put}
\matnotn{Joprime}{\ojJgp o'} & {}= G' \cap \oJgp o = \sbtl{G'}o r.
\end{align*}
Let \matnotn{phihat}{\hat\phi_o} be the extension of \((\ojJgp o', \phi_o)\) trivially across \(\sbtlp{\Der(G', G)}o{(r, s)}\) to \(\oJpgp o\).
\end{notn}

Yu uses the theory of the Weil representation to construct a representation \(\tilde\phi_o\) of \(K'_o \ltimes \oJgp o\) \cite{yu:supercuspidal}*{Theorem 11.5} that is (\(\sbtlp{G'}o 0 \ltimes 1\))-isotrivial and \((1 \ltimes \oJpgp o, \hat\phi_o)\)-isotypic.  As suggested in \cite{yu:supercuspidal-revisited}, and seen by example in \cite{debacker-spice:stability}, the use of the ``bare'', untwisted, Weil representation is not ideal.  We describe in \cite{spice:weil}, and use here, a representation \matnotn{phitilde}{\tilde\phi_o^+} (a twist of \(\tilde\phi_o\) by a character of \(K'_o \ltimes \oJgp o\) that is trivial on \(\sbtlp{G'}o 0 \ltimes \oJgp o\)) that has better intertwining and character-theoretic properties, but still satisfies \cite{yu:supercuspidal}*{\S4, \textbf{SC2}}.

\begin{defn}
\label{defn:rho-pi}
As in \cite{yu:supercuspidal}*{\S4, p.~592}, there is a unique representation \(\rho\) of \(K_o = K'_o\dotm\oJgp o\) on the tensor product of the spaces of \(\rho'\) and \(\tilde\phi_o^+\) such that
\[
(\rho\chi\inv)(k'_o j_o)
\qeqq
\rho'(k'_o) \otimes \tilde\phi_o^+(j_o)
\]
for all \(k'_o \in K'_o\) and \(j_o \in \oJgp o\).  Put \(\mnotn\pi = \Ind_{K_o}^G \rho\).
\end{defn}

Yu's analogue of \(\rho\), which is denoted by \(\rho_d\), is defined on a smaller subgroup than ours; but, if we were not using a different choice of representation \(\tilde\phi_o\) satisfying \textbf{SC2}, then its induction to \(K_o\) would be isomorphic to our representation \(\rho\), hence would not change the isomorphism class of \(\pi\).  Although our \(\pi\) is not the same as the one that the unmodified construction of \cite{yu:supercuspidal} assigns to our quintuple (because of our use of \(\tilde\phi_o^+\) instead of \(\tilde\phi_o\)), it does still arise from that unmodified construction \via a different choice of quintuple.  To obtain the parameterising datum for the unmodified construction, one absorbs into \(\phi_o\) the twist by which \(\tilde\phi_o^+\) differs from \(\tilde\phi_o\).

\begin{lem}
\label{lem:rho-isotypic}
We have that \(\rho\) is \((\oJpgp o, \hat\phi_o)\)-isotypic.
\end{lem}

\begin{proof}
By Definition \ref{defn:rho-pi}, for all \(\jp \in \oJpgp o\), the operator \(\rho(\jp)\) may be written as \(\rho'(1) \otimes \tilde\phi_o^+(1 \ltimes \jp)\).  The result thus follows from the fact that \(\tilde\phi_o^+\) is \((1 \ltimes \oJpgp o, \hat\phi_o)\)-isotypic \cite{yu:supercuspidal}*{\S4, \textbf{SC2}}.
\end{proof}

We resume the notation for the rest of the paper a bit at a time.  For now, in addition to the
	\begin{itemize}
	\item positive real number \(r\),
	\item tame, twisted Levi subgroup \(\bG'\),
and	\item element \(Z^*_o \in \Lie^*(G')\)
	\end{itemize}
that we have already, we let \(x\) be a point of \(\BB(G)\) (eventually, the point in \S\ref{sec:depth-matrix}), and recall the
	\begin{itemize}
	\item element \(X^* \in \sbtl{\Lie^*(G')}x{-r}\)
and	\item characters \(\phi\) of \(\sbtl{G'}x r\) and \(\hat\phi\) of \(\sbtl{(G', G)}x{(r, \Rp s)}\)
	\end{itemize}
from \S\ref{sec:Gauss}, of which we require that \(X^*\) belongs to \(Z^*_o + \sbjtlpp{\Lie^*(G')}{-r}\).  We do not have explicitly to impose Hypothesis \ref{hyp:X*}; it will hold automatically, by Remark \ref{rem:X*}, once we have recalled \(\gamma\).  We should regard \(x\), \(X^*\), and \(\phi\) as fixed only provisionally, until Lemma \ref{lem:dist-G-to-G'}; we need to allow it to vary in the proof of Theorem \ref{thm:asymptotic-pi-to-pi'}.  Although we do impose Hypothesis \ref{hyp:phi} here, it will not need to be explicitly stated in our main result, Theorem \ref{thm:asymptotic-pi-to-pi'}, where it will automatically be satisfied.

In Notation \ref{notn:Yu-x}, we mostly adapt Notation \ref{notn:Yu} and the following discussion from the point \(o\) to the point \(x\) (and then drop the subscript \(x\)); but note that the group \(K'\) may be smaller than the direct analogue \(\stab_{G'}(\ox)\) of \(K'_o\).

\begin{notn}
\label{notn:Yu-x}
Put
\begin{align*}
\mnotn\Jgp    & {}= \sbtl{(G', G)}x{(r, s)}, \\
\mnotn\Jpgp   & {}= \sbtl{(G', G)}x{(r, \Rp s)},
\intertext{and}
\mnotn{\Jgp'} & {}= G' \cap \Jgp = \sbtl{G'}x r, \\
\mnotn{K'}    & {}= \stab_{G'}(\ox, \phi)
\intertext{and}
\mnotn K      & {}= K'\dotm\Jgp = \stab_{G'}(\ox, \phi)\sbtl G x s.
\end{align*}
As in \cite{yu:supercuspidal}*{Proposition 11.4 and Theorem 11.5}, we may use a special isomorphism, in the sense of \cite{yu:supercuspidal}*{\S10, p.~601}, to pull back a Weil representation to \(K' \ltimes \Jgp\).  We then write \matnotn{phitilde}{\tilde\phi^+} for the twist of this pullback constructed in \cite{spice:weil}.  Recall that this twist is by a character that is trivial on \(\sbtlp{G'}x 0 \ltimes \Jgp\).  To emphasise when we are considering only the action of \Jgp, which is unaffected by the twist, we denote \(\Res^{K' \ltimes \Jgp}_\Jgp \tilde\phi^+\) by \matnotn{phitilde}{\tilde\phi}.
\end{notn}

Lemma \ref{lem:Ko-and-J} is the first of three calculations that prepare us for one of the main results of this section, Theorem \ref{thm:isotypic-pi-to-pi'}.  It seems that the two statements in Lemma \ref{lem:Ko-and-J} should admit a common generalisation, rather than just having similar proofs, but we have not been able to find such a generalisation.

\begin{lem}
\label{lem:Ko-and-J}
For all \(g' \in G'\), we have that
\begin{align*}
\Int(g')\inv K_o \cap \Jgp
&\qeqq
(\Int(g')\inv K'_o \cap \Jgp')(\Int(g')\inv\oJgp o \cap \Jgp) \\
\intertext{and}
K_o g'\Jgp \cap G'
&\qeqq
K'_o g'\Jgp'.
\end{align*}
\end{lem}

\begin{proof}
Multiplying the second equality on the left by \(g\primeinv\) shows that it, too, is a statement about \(\Int(g')\inv K'_o\).  Replacing \(o\) by \(g\primeinv\dota o\), hence \(K'_o\) and \(K_o\) by their conjugates under \(g\primeinv\), corresponds to replacing \(((\bG', \bG), o, (r \le r_d), \rho', (\phi, \chi))\) by \(((\bG', \bG), g\primeinv\dota o, (r \le r_d), \rho' \circ \Int(g), (\phi, \chi))\), which gives rise to the same representations \(\pi'\) and \(\pi\); so we may, and do, assume upon making this replacement that \(g'\) equals \(1\).

Note that \(K'_o\) normalises \oJgp o, so that \((K'_o \cap \Jgp')(\oJgp o \cap \Jgp)\) is a group.

By Lemma \ref{lem:heres-a-gp}, we may, and do, pass to a tame extension, and so assume that \(\bG'\) is a Levi subgroup of \bG.  Let \(\bU^\pm\) be the unipotent radicals of opposite parabolics with Levi component \(\bG'\).  Also by Lemma \ref{lem:heres-a-gp}, we have that \(K_o \cap \Jgp\) equals \((K_o \cap \Jgp \cap G')(K_o \cap \Jgp \cap U^+)(K_o \cap \Jgp \cap U^-)\) and \(K_o\Jgp\) equals \(K'_o(K_o \cap U^+)K'_o(K_o \cap U^-)(\Jgp \cap U^-)(\Jgp \cap U^+)\Jgp'\).

We have that \(K_o \cap \Jgp \cap G'\) equals \(K'_o \cap \Jgp'\) and \(K_o \cap U^\pm\) equals \(\sbtl{U^\pm}o s \subseteq \sbtl{(G', G)}o s = \oJgp o\), so the first equality follows.

If we write an element of its intersection with \(\bG'\) as \(g' = k'k^+ k^- j^- j^+ j'\), then \(k^-j^-\) belongs to \(U^- \cap (k'k^+)\inv G'(j^+ j')\inv \subseteq U^- \cap G' U^+ = \sset1\).  That is, \(k^- j^-\) equals \(1\).  Then \(k^+ j^+ = k^+ k^- j^- j^+\) belongs to \(U^+ \cap k\primeinv G' j\primeinv = U^+ \cap G' = \sset1\), so that \(k^+ j^+\) also equals \(1\), and hence \(g'\) equals \(k'j'\), which belongs to \(K'_o\Jgp'\).
\end{proof}

Recall that Definition \ref{defn:rho-pi} expresses the inducing representation \(\rho\) for \(\pi\) as a tensor product of the inducing representation \(\rho'\) for \(\pi'\) with a Weil representation.  We show in Lemma \ref{lem:phi-into-rho} that its (\(\tilde\phi\chi\))-isotypic subspace admits a similar tensor-product decomposition.

\begin{lem}
\initlabel{lem:phi-into-rho}
The canonical isomorphism
\[
\anonmap{\Hom_\C(\phi, \rho') \otimes_\C \Hom_\C(\tilde\phi, \tilde\phi_o)}{\Hom_\C(\tilde\phi\chi, \rho)}
\]
(all Hom spaces being in the category of vector spaces) sending \(v' \otimes \varphi\) to the map \anonmapto w{v'(1) \otimes \varphi(w)} restricts to an isomorphism
\[
\anonmap{\Hom_{\Jgp' \cap K'_o}(\phi, \rho') \otimes_\C \Hom_{\Jgp \cap \oJgp o}(\tilde\phi, \tilde\phi_o)}{\Hom_{\Jgp \cap K_o}(\tilde\phi\chi, \rho)}
\]
\end{lem}

\begin{proof}
Since \(\Hom_{\Jgp \cap K_o}(\tilde\phi\chi, \rho)\) equals \(\Hom_{\Jgp \cap K_o}(\tilde\phi, \rho\chi\inv)\), we may, and do, assume for notational convenience that \(\chi\) is trivial.  Since we will do the same thing repeatedly in later results, we emphasise that it does not change any essential idea of the proof, only keeps us from having to carry around a cumbersome factor of \(\chi\) everywhere.

Since \(\Jgp \cap K_o\) equals \((\Jgp' \cap K'_o)(\Jgp \cap \oJgp o)\) by Lemma \ref{lem:Ko-and-J}, certainly the indicated restriction has image in \(\Hom_{\Jgp \cap K_o}(\tilde\phi, \rho)\).

By Definition \ref{defn:rho-pi}, we have that \(\Res^K_{\oJgp o} \rho\) equals (not just is isomorphic to) \(\rho' \otimes_\C \tilde\phi_o\), where the space of \(\rho'\) carries the trivial action of \(\oJgp o\); so the (\(\Jgp \cap \oJgp o\))-isotrivial component of \(\Hom_\C(\tilde\phi, \rho)\) is identified \via the above isomorphism with
\begin{equation}
\tag{$*$}
\sublabel{eq:JJ}
\Hom_\C(\phi, \rho') \otimes_\C \Hom_{\Jgp \cap \oJgp o}(\tilde\phi, \tilde\phi_o).
\end{equation}
By \cite{spice:weil}, we have that \(\Hom_{\Jgp \cap \oJgp o}(\tilde\phi, \tilde\phi_o)\) equals \(\Hom_{(K' \cap K'_o) \ltimes (\Jgp \cap \oJgp o)}(\tilde\phi^+, \tilde\phi_o^+)\), in particular is stable under \(\Jgp' \cap K'_o\), and is \((\Jgp' \cap K'_o, \phi^\vee)\)-isotypic; so that the (\(\Jgp' \cap K'_o\))-isotrivial component of \loceqref{eq:JJ} is \(\Hom_{\Jgp' \cap K'_o}(\phi, \rho') \otimes_\C \Hom_{\Jgp \cap \oJgp o}(\tilde\phi, \tilde\phi_o)\).  This shows that \(\Hom_{\Jgp \cap K_o}(\tilde\phi, \rho)\) is no larger than desired.  Since we have already shown the reverse containment, we are done.
\end{proof}

We use Lemma \ref{lem:Heis-intertwining} to cut down the summands appearing in a Mackey-type formula in the proof of Theorem \ref{thm:isotypic-pi-to-pi'}.

\begin{lem}
\label{lem:Heis-intertwining}
If \(g \in G\) is such that \(\Hom_{\Jgp \cap \Int(g)\inv K_o}(\tilde\phi\chi, \rho^g)\) is non-\(0\), then \(g\) belongs to \(K_o G'\Jgp\).
\end{lem}

\begin{proof}
As in the proof of Lemma \ref{lem:phi-into-rho}, we may, and do, assume that \(\chi\) is trivial.

We have that \(\rho\) is \((\oJpgp o, \hat\phi_o)\)-isotypic (Lemma \ref{lem:rho-isotypic}), and \(\tilde\phi\) is \((\Jpgp, \hat\phi)\)-isotypic.  In particular, if the indicated Hom space is non-\(0\), then \((\Jpgp, \hat\phi)\) and \((\Int(g)\inv\oJpgp o, \hat\phi_o^g)\) agree on the intersection of their domains.  By Hypothesis \ref{hyp:K-type}, we have that \(X^* + \sbtl{\Lie^*(G', G)}x{(\Rpp{-r}, -s)}\) intersects \(\Ad^*(g)\inv(Z^*_o + \sbtl{\Lie^*(G', G)}o{(\Rpp{-r}, -s)})\).  Recall that \(X^* + \sbtlpp{\Lie^*(G')}x{-r}\) is contained in \(Z^*_o + \sbjtlpp{\Lie^*(G')}r\) by Remark \ref{rem:X*}.  By the dual-Lie-algebra analogue of \xcite{adler-spice:good-expansions}*{Lemma \xref{lem:good-comm}}, we may, and do, assume, upon adjusting \(g\) on the right by an element of \(\sbtl G x s\) and on the left by an element of \(\sbtl G o s\), which does not affect the conclusion (because \(\sbtl G o s\) is contained in \(K_o\), and \(\sbtl G x s\) is contained in \(G'\Jgp\)), that \(X^* + \sbtlpp{\Lie^*(G')}x{-r} \subseteq Z^*_o + \sbjtlpp{\Lie^*(G')}{-r}\) intersects \(\Ad^*(g)\inv(Z^*_o + \sbjtlpp{\Lie^*(G')}{-r})\).  Then Hypothesis \initref{hyp:Z*}\subpref{orbit} gives that \(g'\) belongs to \(G'\), as desired.
\end{proof}

We regard Theorem \ref{thm:isotypic-pi-to-pi'} as a hint about the existence of the Hecke-algebra isomorphisms predicted in \cite{yu:supercuspidal}*{Conjecture 17.7}.  Although it deals only with modules \(\Hom_{\Jgp'}(\phi, \pi')\) and \(\Hom_\Jgp(\tilde\phi\chi, \pi)\) for the compactly supported parts \(\Hecke(K'/\Jgp', \phi)\) and \(\Hecke(K/\Jgp, \tilde\phi\chi)\) of the Hecke algebras (for which an isomorphism is already known when \(x\) equals \(o\) \cite{yu:supercuspidal}*{Lemma 17.10}), we find it suggestive, and hope that it will be useful as a starting point for future investigations.

\begin{thm}
\initlabel{thm:isotypic-pi-to-pi'}
There is a (non-canonical) isomorphism of vector spaces
\[
\anonmap[\mapisoarrow]{\Hom_{\Jgp'}(\phi, \pi')}{\Hom_\Jgp(\tilde\phi\chi, \pi)}
\]
so that the resulting map
\[
\anonmap{\Hom_{\Jgp'}(\phi, \pi') \otimes_\C \tilde\phi^+\chi}{\Res^G_K \pi}
\]
is a \(K\)-homomorphism.
\end{thm}

\begin{proof}
Again, as in the proof of Lemma \ref{lem:phi-into-rho}, we may, and do, assume that \(\chi\) is trivial.

We have by Mackey theory that \(\Res^G_\Jgp \pi = \Res^G_\Jgp \Ind_{K_o}^G \rho\) is canonically isomorphic to \(\bigoplus_{g \in K_o\bslash G/\Jgp} \Ind_{\Jgp \cap \Int(g)\inv K_o}^\Jgp \rho^g\).  Since the space of \(\tilde\phi\) is finite-dimensional, the natural map
\[
\anonmap
	{\bigoplus_{g \in K_o\bslash G/\Jgp} \Hom_\Jgp(\tilde\phi, \Ind_{\Jgp \cap \Int(g)\inv K_o}^\Jgp \rho^g)}
	{\Hom_\Jgp(\tilde\phi, \pi)}
\]
is an isomorphism.  We make a note here that will crop up repeatedly in this proof.  Although the \emph{isomorphism classes} of the summands on the left do not depend on the specific choice of representative \(g\) for a \((K_o, \Jgp)\)-double coset, the actual \emph{sets} do.  Specifically, for \(g_1, g_2 \in G\) with \(g_2 = k_o g_1 j\), where \(k_o \in K_o\) and \(j \in \Jgp\), we have a canonical isomorphism
\[
\anoniso
	{\Ind_{\Jgp \cap \Int(g_1)\inv K_o}^\Jgp \rho^{g_1}}
	{\Ind_{\Jgp \cap \Int(g_2)\inv K_o}^\Jgp \rho^{g_2}}
\]
that sends a function \(f\) in the former space to the function \anonmapto g{\rho(k_o)f(g j)}.  This map depends only on \(g_1\) and \(g_2\), not on \(k_o\) and \(j\).  We use it freely to change representative in a double coset without explicit mention.  (Strictly speaking, what we are doing is viewing the summands not as individual spaces, but as limits of spaces with respect to the system of maps as above corresponding to various \(j\) and \(k_o\).)

For \(g \in G\), Frobenius reciprocity gives a canonical isomorphism of \(\Hom_\Jgp(\tilde\phi, \Ind_{\Jgp \cap \Int(g)\inv K_o}^\Jgp \rho^g)\) with \(\Hom_{\Jgp \cap \Int(g)\inv K_o}(\tilde\phi, \rho^g)\).  By Lemma \ref{lem:Heis-intertwining}, we have for \(g \in G\) that \(\Hom_{\Jgp \cap \Int(g)\inv K_o}(\tilde\phi, \rho^g)\) is \(0\) unless \(g\) belongs to \(K_o G'\Jgp\).  For \(g' \in G'\), we have from Lemma \ref{lem:phi-into-rho} a canonical isomorphism
\begin{multline*}
\Hom_{\Jgp' \cap \Int(g')\inv K'_o}(\phi, \rho^{\prime\,g'}) \otimes_\C \Hom_{\Jgp \cap \Int(g')\inv\oJgp o}(\tilde\phi, \tilde\phi_o^{{+}\,{g' \ltimes 1}}) \\
\mapisoarrow \Hom_{\Jgp \cap \Int(g')\inv\oJgp o}(\tilde\phi, \rho^{g'}).
\end{multline*}

Finally, by Lemma \ref{lem:Ko-and-J}, we have that the natural map \anonmap{K'_o\bslash G'/{\Jgp'}}{K_o\bslash G/\Jgp} is an injection, so that we have constructed a canonical isomorphism
\begin{multline*}
\bigoplus_{g' \in K'_o\bslash G'/{\Jgp'}}
	\Hom_{\Jgp' \cap \Int(g')\inv K'_o}(\phi, \rho^{\prime\,g'}) \otimes_\C \Hom_{\Jgp \cap \Int(g')\inv\oJgp o}(\tilde\phi, \tilde\phi_o^{{+}\,{g' \ltimes 1}}) \\
\mapisoarrow \Hom_\Jgp(\tilde\phi, \pi).
\end{multline*}
We make this explicit, and convert it into the desired isomorphism, after a few remarks.  First, write \(V'\) for the space of \(\rho'\), and \(W\) and \(W_o\) for the spaces of \(\tilde\phi\) and \(\tilde\phi_o^+\), respectively, so that the space of \(\rho\) is \(V' \otimes W_o\).  Second, note that there is a canonical identification, for each \(g' \in G'\), of \(\Hom_{\Jgp' \cap \Int(g')K'_o}(\phi, \rho^{\prime\,g'})\) with the \((\Jgp' \cap \Int(g')K'_o, \phi)\)-isotypic subspace of \(\rho^{\prime\,g'}\) (a subspace of \(V'\)); we make this identification without further comment.  Finally, note again that the summands, as isomorphism classes, are determined by the double coset \(K'_o g'\Jgp'\), but, as sets, depend on the choice of representative \(g'\).  Namely, if \(g'_1, g'_2 \in G'\) with \(g'_2 = k'_o g'_1j'\), where \(k'_o \in K'_o\) and \(j' \in \Jgp'\), then we have a canonical isomorphism
\[
\anoniso
	{\Hom_{\Jgp' \cap \Int(g'_1)\inv K'_o}(\phi, \rho^{\prime\,g'_1})}
	{\Hom_{\Jgp' \cap \Int(g'_2)\inv K'_o}(\phi, \rho^{\prime\,g'_2})}
\]
given by \anonmapto{v'}{\rho'(k'_o)\phi(j')v'}, and another canonical isomorphism
\[
\anoniso
	{\Hom_{\Jgp \cap \Int(g'_1)\inv\oJgp o}(\tilde\phi, \tilde\phi_o^{{+}\,{g'_1 \ltimes 1}})}
	{\Hom_{\Jgp \cap \Int(g'_2)\inv\oJgp o}(\tilde\phi, \tilde\phi_o^{{+}\,{g'_2 \ltimes 1}})}
\]
given by \anonmapto\varphi{\tilde\phi_o^+(k'_o) \circ \varphi \circ \tilde\phi(j')}.  Each of these maps depends only on \(g'_1\) and \(g'_2\), not on \(k'_o\) or \(j'\).

Further, each space \(\Hom_{\Jgp \cap \Int(g')\inv\oJgp o}(\tilde\phi, \tilde\phi_o^{{+}\,{g' \ltimes 1}})\) is non-\(0\) (in fact \(1\)-dimensional), and equals \(\Hom_{K' \cap \Int(g')\inv K'_o \ltimes \Jgp \cap \Int(g')\inv\oJgp o}(\tilde\phi^+, \tilde\phi_o^{{+}\,{g' \ltimes 1}})\) \cite{spice:weil}.  Choose a set \(\mc S'\) of representatives for \(K'_o\bslash G'/K'\) (not just \(K'_o\bslash G'/{\Jgp'}\)) in \(G'\), and arbitrarily choose, for each \(g' \in \mc S'\), a non-\(0\) element \(\varphi_{g'}\) of the corresponding Hom space.  (It would be pleasant to describe a canonical choice here, but we have not been able to make one.)  For \(g' \in \mc S'\), \(k'_o \in K'_o\), and \(k' \in K'\), the composition \(\tilde\phi_o^+(k'_o \ltimes 1) \circ \varphi_{g'} \circ \tilde\phi^+(k' \ltimes 1)\) depends only on the product \(k'_o g' k'\); so we may, and do, unambiguously put \(\varphi_{k'_o g'k'} = \tilde\phi_o^+(k'_o) \circ \varphi_{g'} \circ \tilde\phi^+(k')\).  We have now defined \(\varphi_{g'}\) for all \(g' \in G'\).

Now notice that we have an isomorphism (canonical subject to the choice of \(\varphi_{g'}\) above) given by
\begin{multline*}
\bigoplus_{g' \in K'_o\bslash G'/{\Jgp'}} \Hom_{\Jgp' \cap \Int(g')\inv K'_o}(\phi, \rho^{\prime\,g'}) \\
\mapisoarrow \bigoplus_{g' \in K'_o\bslash G'/{\Jgp'}} \Hom_{\Jgp' \cap \Int(g')\inv K'_o}(\phi, \rho^{\prime\,g'}) \otimes_\C \Hom_{\Jgp \cap \Int(g')\inv\oJgp o}(\tilde\phi, \tilde\phi_o^{{+}\,{g' \ltimes 1}})
\end{multline*}
sending \(\bigoplus v'_{g'}\) to \(\bigoplus v'_{g'} \otimes \varphi_{g'}\).  We have already seen that the right-hand side is canonically isomorphic to \(\Hom_\Jgp(\tilde\phi, \pi)\).  A similar but easier calculation shows that the left-hand side is canonically isomorphic to \(\Hom_{\Jgp'}(\phi, \pi')\), identified, as usual, with the \((\Jgp', \phi)\)-isotypic subspace of \(\pi'\).  The resulting isomorphism (canonical subject to the choice of \(\varphi_{g'}\))
\[
\anonmap[\mapisoarrow]{\Hom_{\Jgp'}(\phi, \pi')}{\Hom_\Jgp(\tilde\phi, \pi)},
\]
which we call \(T\), takes a particularly pleasant form.  Namely, if \(f'\) is a \((\Jgp', \phi)\)-isotypic vector in the target, which is to say a certain function \anonmap{G'}{V'}, then the image of \(f'\) is the function that sends \(w \in W\) to the function \map f G{V = V' \otimes_\C W} that vanishes off \(K_o G'\Jgp\), and, for \(g' \in G'\), \(k_o \in K_o\), and \(j \in \Jgp\), satisfies
\[
f(k_o g'j) = \rho(k_o)\bigl(f'(g') \otimes \varphi_{g'}(\tilde\phi(j)w)\bigr).
\]

The map that we want to be a \(K\)-homomorphism is, by definition, at least a \Jgp-homomorphism, so it suffices to consider the action of \(K'\).  Fix \(\gamma \in K'\).  Notice that \(\gamma\) normalises \(\Jgp\), and that, by construction, \(\varphi_{g'\gamma}\) equals \(\varphi_{g'} \circ \tilde\phi(\gamma)\); so \((\pi(\gamma)f)(k_o g'j) = f(k_o g'j\gamma)\) equals
\[
\rho(k_o)\bigl((\pi'(\gamma)f')(g') \otimes \varphi_{g'}(\tilde\phi(j)(\tilde\phi^+(\gamma)(w)))\bigr).
\]
Since \(K'\) normalises \((\Jgp', \phi)\) by Notation \ref{notn:Yu-x}, we have that \(\pi'(\gamma)f'\) is still a \((\Jgp', \phi)\)-isotypic vector.  In other words,
\[
\pi(\gamma)\bigl((T(f'))(w)\bigr)
\qeqq
(T(\pi'(\gamma)f'))(\tilde\phi^+(\gamma)w).\qedhere
\]
\end{proof}

Corollary \ref{cor:isotypic-pi-to-pi'}, which follows almost immediately from Theorem \ref{thm:isotypic-pi-to-pi'}, allows us to use Theorem \ref{thm:dist-G-to-G'} (or, rather, its reformulation Lemma \ref{lem:dist-G-to-G'}) in Theorem \ref{thm:asymptotic-pi-to-pi'}.

\begin{cor}
\label{cor:isotypic-pi-to-pi'}
For \(\gamma \in K'\), we have that
\begin{align*}
\card{\sbat{\CCp G 0(\gamma)}x s}\inv[1/2]\Gauss_{G/\CCp G 0(\gamma^2)}(X^*, \gamma)&\tr \pi\bigl(\gamma\chrc{\Jpgp, \hat\phi^\vee\chi^\vee}\bigr) \\
\intertext{equals}
\card{\sbat{\CCp{G'}0(\gamma)}x s}\inv[1/2]\Gauss_{G'/\CCp{G'}0(\gamma^2)}(X^*, \gamma)&\chi(\gamma)\tr \pi'\bigl(\gamma\chrc{\Jgp', \phi^\vee}\bigr).
\end{align*}
\end{cor}

\begin{proof}
Once more, as in the proof of Lemma \ref{lem:phi-into-rho}, we may, and do, assume that \(\chi\) is trivial.

By Lemma \initref{lem:filtration}\subpref{gp-gp}, we have that \Jgp normalises \((\Jpgp, \hat\phi)\), hence also the \((\Jpgp, \hat\phi)\)-isotypic component of \(\pi\); and that (by the Stone--von Neumann theorem) \((\Jgp, \tilde\phi)\) is the unique irreducible representation of \Jgp containing \((\Jpgp, \hat\phi)\).  It follows that the \((\Jpgp, \hat\phi)\)- and \((\Jgp, \tilde\phi)\)-isotypic components of \(\pi\) are the same, so that \(\pi(\chrc{\Jpgp, \hat\phi^\vee})\) is the projection onto the latter.  Thus, the desired result is a consequence of Theorem \ref{thm:isotypic-pi-to-pi'} and the fact \cite{spice:weil} that
\[
\tr \tilde\phi^+(\gamma)
\qeqq
\indx{\sbat{\CCp G 0(\gamma)}x s}{\sbat{\CCp{G'}0(\gamma)}x s}^{1/2}\frac{\Gauss_{G/\CCp G 0(\gamma^2)}(X^*, \gamma)}{\Gauss_{G'/\CCp{G'}0(\gamma)^2}(X^*, \gamma)}.\qedhere
\]
\end{proof}

Now, in addition to the
	\begin{itemize}
	\item positive real number \(r\),
	\item tame, twisted Levi subgroup \(\bG'\),
	\item element \(Z^*_o \in \Lie^*(G')\),
	\item point \(x \in \BB(G)\),
	\item element \(X^* \in \Lie^*(G')\),
and	\item characters \((\sbtl{G'}x r, \phi)\) and \((\sbtl{(G', G)}x{(r, \Rp s)}, \hat\phi)\)
	\end{itemize}
that we have already (although we regard \(x\), \(X^*\), and \(\phi\) as fixed only provisionally), we recall the
	\begin{itemize}
	\item element \(\gamma \in G\), with associated group \(\bH = \CC\bG r(\gamma)\),
	\end{itemize}
of which we require that
\(x\) belongs to \(\BB(H)\) and \(X^*\) to \(\Lie^*(H)\), satisfying Hypotheses \ref{hyp:funny-centraliser}, \ref{hyp:gamma}, and now \ref{hyp:gamma-central}.  We have by Hypothesis \ref{hyp:gamma-central} and Lemma \ref{lem:in-H'} (and the fact that \(X^*\) belongs to \(\Lie^*(H) \cap (Z^*_o + \sbjtlpp{\Lie^*(G')}{-r})\)) that \(\gamma\) belongs to \(G'\); and by Remark \ref{rem:X*} that Hypothesis \ref{hyp:X*} is satisfied.  We use primes to denote the analogues in \(\bG'\) of constructions in \bG; so, for example, \(\bH'\) stands for \(\CC{\bG'}r(\gamma)\) (subject to the proviso in Remark \ref{rem:tame-Levi}, that we may refer directly only to the identity component of \(\bH'\)).

We prepare for our main result, Theorem \ref{thm:asymptotic-pi-to-pi'}, with Lemma \ref{lem:dist-G-to-G'}.  This slightly reformulates Theorem \ref{thm:dist-G-to-G'}, taking advantage of Hypothesis \ref{hyp:gamma-central} to speak of \(\Gauss(\OO', \gamma)\) rather than \(\Gauss(X^*, \gamma)\) (Notation \ref{notn:Gauss}).  Note that the hypothesised relationship between \(T\) and \(T'\) in Lemma \ref{lem:dist-G-to-G'} is exactly the same as in Theorem \ref{thm:dist-G-to-G'}, although the asymptotic expansions of \(T\) and \(T'\) individually are different.

\begin{lem}
\label{lem:dist-G-to-G'}
Suppose that
	\begin{itemize}
	\item \(T'\) is an invariant distribution on \(G'\),
	\item \(c(T', \gamma)\) is a finitely supported, \(\OO^{H\primeconn}(\Ad^*(G')(Z^*_o + \sbjtlpp{\Lie^*(H')}{-r}))\)-indexed vector of complex numbers such that \(\abs{\Disc_{G'/H'}(\gamma)}^{1/2}T'(\gamma\chrc{\sbtl{G'}x r, \phi^\vee})\) equals
\begin{align*}
q_{H'}^s\sum_{\OO' \in \OO^{H\primeconn}(\Ad^*(G')(Z^*_o + \sbjtlpp{\Lie^*(H')}{-r}))}
{}	&\frac{\Gauss_{G'}(\OO', \gamma)}{\Gauss_{\CCp{G'}0(\gamma^2)/\CCp{G'}0(\gamma)}(\OO', \gamma)}\times{} \\
	&\qquad c_{\OO'}(T', \gamma)
	\Gauss_{H'}(\OO', \gamma)\inv\muhat^{H\primeconn}_{\OO'}\bigl(\chrc{\sbtl{\fh'}x r, \AddChar_{X^*}^\vee}\bigr),
\end{align*}
and	\item \(T\) is an invariant distribution on \(G\) such that
\begin{align*}
&\modulus_{P^-}(\gamma)^{1/2}
\card{\sbat{\CCp G 0(\gamma)}x s}\inv[1/2]
\Gauss_{G/\CCp G 0(\gamma^2)}(X^*, \gamma)\inv
T\bigl(\gamma\chrc{\sbtl{(G', G)}x{(r, \Rp s)}, \hat\phi^\vee}\bigr) \\
\intertext{equals}
&\modulus_{P\suppm}(\gamma)^{1/2}
\card{\sbat{\CCp{G'}0(\gamma)}x s}\inv[1/2]
\Gauss_{G'/\CCp{G'}0(\gamma^2)}(X^*, \gamma)\inv
T'\bigl(\gamma\chrc{\sbtl{G'}x r, \hat\phi^\vee}\bigr).
\end{align*}
	\end{itemize}
Then \(\abs{\Disc_{G/H}(\gamma)}^{1/2}T(\gamma\chrc{\sbtl G x r, \hat\phi^\vee})\) equals
\begin{align*}
q_H^s
\sum_{\OO' \in \OO^{H\primeconn}(Z^*_o + \sbjtlpp{\Lie^*(H')}{-r})}
{}	&\frac{\Gauss_G(\OO', \gamma)}{\Gauss_{\CCp G 0(\gamma^2)/\CCp G 0(\gamma)}(\OO', \gamma)}\times{} \\
	&\qquad c_{\OO'}(T', \gamma)
	\Gauss_H(\OO', \gamma)\inv\muhat^{H\conn}_{\OO'}\bigl(\chrc{\sbtl\fh x r, \AddChar_{X^*}^\vee}\bigr).
\end{align*}
\end{lem}

\begin{proof}
By Theorem \ref{thm:dist-G-to-G'}, it suffices to show that
\begin{align*}
&\Bigl(\frac{\Gauss_{G'/H'}(X^*, \gamma)}{\Gauss_{\CCp{G'}0(\gamma^2)/\CCp{G'}0(\gamma)}(X^*, \gamma)}\Bigr)\inv
\Bigl(\frac{\Gauss_{G'/H'}(\OO', \gamma)}{\Gauss_{\CCp{G'}0(\gamma^2)/\CCp{G'}0(\gamma)}(\OO', \gamma)}\Bigr) \\
\intertext{equals the analogous quantity ``without primes'', which is to say}
&\Bigl(\frac{\Gauss_{G/H}(X^*, \gamma)}{\Gauss_{\CCp G 0(\gamma^2)/\CCp G0(\gamma)}(X^*, \gamma)}\Bigr)\inv
\Bigl(\frac{\Gauss_{G/H}(\OO', \gamma)}{\Gauss_{\CCp G 0(\gamma^2)/\CCp G0(\gamma)}(\OO', \gamma)}\Bigr),
\end{align*}
for all orbits \(\OO'\) that intersect \(\Ad^*(G')(Z^*_o + \sbjtlpp{\Lie^*(H')}{-r})\), and for which \(\muhat^H_{\OO'}\bigl(\chrc{\sbtl\fh x r, \AddChar_{X^*}^\vee}\bigr)\) is non-\(0\).  Since the non-vanishing of the Fourier transform implies that \(\Ad^*(H)\OO'\) intersects \(X^* + \sbtlpp{\Lie^*(H)}x{-r}\), we have by Lemma \ref{lem:X*-orbits} that such an orbit \(\OO'\) in fact intersects \(X^* + \sbtlpp{\Lie^*(H')}x{-r}\).  It now follows from Corollary \ref{cor:Gauss-const} that
\begin{align*}
&\frac{\Gauss_{G/H}(X^*, \gamma)}{\Gauss_{G'/H'}(X^*, \gamma)}\dotm\frac{\Gauss_{\CCp G 0(\gamma^2)/\CCp G 0(\gamma)}(X^*, \gamma)\inv}{\Gauss_{\CCp{G'}0(\gamma^2)/\CCp{G'}0(\gamma)}(X^*, \gamma)\inv} \\
\intertext{equals}
&\frac{\Gauss_{G/H}(\OO', \gamma)}{\Gauss_{G'/H'}(\OO', \gamma)}\dotm\frac{\Gauss_{\CCp G 0(\gamma^2)/\CCp G 0(\gamma)}(\OO', \gamma)\inv}{\Gauss_{\CCp{G'}0(\gamma^2)/\CCp{G'}0(\gamma)}(\OO', \gamma)\inv},
\end{align*}
as desired.
\end{proof}

Theorem \ref{thm:asymptotic-pi-to-pi'} is our main result.  It reduces the computation of the character of \(\pi\) to that of \(\pi'\), so iterating it allows us, at least in principle, to reduce the computation of characters of positive-depth, tame supercuspidals to that of characters of depth-\(0\) supercuspidals (for tame, twisted Levi subgroups).

So far throughout the section, we have worked with a \emph{fixed} point \(x \in \BB(G)\), element \(X^* \in \Lie^*(G)\), and character \(\hat\phi\) of \((\sbtl{(G', G)}x{(r, \Rp s)}\), satisfying Hypothesis \ref{hyp:phi}.  For the proof of Theorem \ref{thm:asymptotic-exists}, we must forget our binding, and regard \(x\), \(X^*\), and \(\phi\) as free variables.  Since it is needed to relate an invariant distribution on the group (a character) to one on the Lie algebra (a Fourier transform of an orbital integral), we now also recall the mock-exponential map \(\mexp\) of \S\ref{sec:dual-blob}.

As we did for Theorem \ref{thm:asymptotic-exists}, we re-capitulate all the hypotheses that are currently in force.  We are imposing
	\begin{itemize}
	\item Hypotheses
		\ref{hyp:funny-centraliser},
		\ref{hyp:gamma} (for \emph{all} points \(x \in \BB(H)\)),
	and	\ref{hyp:gamma-central} (on \(\gamma\)),
and	\item Hypotheses
		\ref{hyp:mexp}
	and	\ref{hyp:MP-ad} (on \(\mexp\) and \(\sbat{\ol\mexp}x{\vec\jmath}\)).
	\end{itemize}
Also note that we have three exponential-type maps floating around, namely, the actual exponential map of Hypothesis \ref{hyp:mexp}; the Moy--Prasad isomorphisms of Hypothesis \ref{hyp:MP-ad}; and the map implicitly used Yu in \cite{yu:supercuspidal}*{\S9} when attaching the element \(Z^*_o\) to the character \(\phi_o\).  We require that these be compatible, in the sense that the diagram
\[\xymatrix{
\sbtl{\Lie(H)}x r        \ar[r]^\mexp\ar[d]          & \sbtl H x r          \ar[d] \\
\sbat{\Lie(M)}x r        \ar[r]^{\text{Yu}}\ar[d]    & \sbat M x r          \ar[d] \\
\sbat{\Lie(M\adform)}x r \ar[r]^{\sbat{\ol\mexp}x r} & \sbat{(M\adform)}x r
}\]
commutes for all \(x \in \BB(H)\).  We do \emph{not} need to impose Hypothesis \ref{hyp:Z*} (on \(Z^*_o\)) or \ref{hyp:K-type} (on \((\sbtl G o r, \hat\phi_o)\)), since they follow from the conditions that Yu imposes.

As stated, Theorem \ref{thm:asymptotic-pi-to-pi'} is contingent on Theorem \ref{thm:asymptotic-exists} and Lemma \ref{lem:asymptotic-check}.  Recall that these theorems carry lengthy lists of hypotheses; rather than citing them, and so incurring the weight of those hypotheses, we prefer to emphasise that they may be treated as black boxes.  As long as the necessary asymptotic expansions (with some, unspecified, coefficients) exist, and can be detected by sampling the distribution character with unrefined, minimal K-types, we are fine.  For example, we do not need to re-impose Hypothesis \ref{hyp:depth}; it was used only to prove Theorem \ref{thm:asymptotic-exists}.

Since we are assuming Theorem \ref{thm:asymptotic-exists} anyway, it may seem that the hypothesis regarding the existence of the various vectors \(c(\pi^{\prime\,g}, \gamma)\) is redundant.  The point is that we are requiring that the support of \(c(\pi^{\prime\,g}, \gamma)\) be contained in \(\OO^{(H \cap \Int(g)\inv G')\conn}(\Ad^*(G'g)\inv X^*_o)\), which may be a proper subset of \(\OO^{(H \cap \Int(g)\inv G')\conn}(\Ad^*(g)\inv Z^*_o)\).  That is, Theorem \ref{thm:asymptotic-pi-to-pi'} \emph{supposes} that only certain of those orbits allowed by Theorem \ref{thm:asymptotic-exists} actually occur in the asymptotic expansions of the various \(\pi^{\prime\,g}\), and \emph{concludes} an analogous statement for \(\pi\) (as well as actually computing the coefficients).

\begin{thm}
\label{thm:asymptotic-pi-to-pi'}
Suppose that Theorem \ref{thm:asymptotic-exists} and Lemma \ref{lem:asymptotic-check} are satisfied, and that \(\gamma\) is compact.

Suppose also that \(X^*_o\) is an element of
\(Z^*_o + \sbjtlpp{\Lie^*(G')}{-r}\)
such that, for every \(g \in G'\bslash G/H\conn\) for which \(\Ad^*(G'g)\inv X^*_o\) intersects \(\Lie^*(H)\), there is an \(\OO^{(H \cap \Int(g)\inv G')\conn}(\Ad^*(G'g)\inv X^*_o)\)-indexed vector \(c(\pi^{\prime\,g}, \gamma)\) for which
we have that
\begin{align*}
&&\Phi_{\pi^{\prime\,g}}(\gamma\dotm\mexp(Y'))&\qeqq \\
&&\sum_{\OO' \in \OO^{(H \cap \Int(g)\inv G')\conn}(\Ad^*(G'g)\inv X^*_o)}
	&\frac{\Gauss_{\Int(g)\inv G'}(\OO', \gamma)}{\Gauss_{\CCp{\Int(g)\inv G'}0(\gamma^2)/\CCp{\Int(g)\inv G'}0(\gamma)}(\OO', \gamma)}\times{} \\
&&	&\qquad c_{\OO'}(\pi^{\prime\,g}, \gamma)\Gauss_{H \cap \Int(g)\inv G'}(\OO', \gamma)\inv\Ohat^{(H \cap \Int(g)\inv G')\conn}_{\OO'}(Y')
\intertext{for all \(Y' \in \Lie(H \cap \Int(g)\inv G')\rss \cap \sbjtl{\Lie(H \cap \Int(g)\inv G')}r\).  Then we have that}
&&\Phi_\pi(\gamma\dotm\mexp(Y))&\qeqq \\
&&\chi(\gamma)\sum_{
g \in G'\bslash G/H\conn
} \sum_{\OO'}
{}
	&\frac{\Gauss_G(\OO', \gamma)}{\Gauss_{\CCp G 0(\gamma^2)/\CCp G 0(\gamma)}(\OO', \gamma)}\times{} \\
&&	&\qquad c_{\OO'}(\pi^{\prime\,g}, \gamma)\Gauss_H(\OO', \gamma)\inv\chi(\mexp(Y))\Ohat^{H\conn}_{\OO'}(Y)
\end{align*}
for all \(Y \in \Lie(H)\rss \cap \sbjtl{\Lie(H)}r\).
\end{thm}

Note that \(\mexp(Y')\) belongs to \(\sbjtl{(H \cap \Int(g)\inv G')}r\) whenever \(Y'\) belongs to \(\sbjtl{\Lie(H \cap \Int(g)\inv G')}r\), for \(g \in G\) such that \(\Ad^*(G'g\inv)X^*_o\) intersects \(\Lie^*(H)\), by Hypothesis \initref{hyp:mexp}\incpref{coset}\subpref{coset}.  Thus the notation \(\Phi_{\pi^{\prime\,g}}(\gamma\dotm\mexp(Y'))\) makes sense.

\begin{proof}
As in Lemma \ref{lem:phi-into-rho}, we may, and do, assume that \(\chi\) is trivial.

By Theorem \ref{thm:asymptotic-exists}, there is \emph{some} asymptotic expansion for \(\pi\) in terms of \(\OO^{H\conn}(\Ad^*(G)Z^*_o)\).  By Remark \ref{rem:which-char} and Lemma \ref{lem:asymptotic-check}, to check that the proposed expansion in the statement is correct, it suffices to show that, whenever \(g_0 \in G\) is such that \(\Ad^*(g_0)\inv Z^*_o\) belongs to \(\Lie^*(H)\), we have for all \(x \in \BB(H \cap \Int(g_0)\inv G')\) and \(X^* \in \sbtl{\Lie^*(H)}x{-r} \cap \Ad^*(g_0)\inv(Z^*_o + \sbjtlpp{\Lie^*(G')}{-r})\) that
\begin{align*}
\abs{\redD_H(\Ad^*(g_0)\inv Z^*_o)}^{1/2}\sum_{\OO'}
	{}
	&\frac{\Gauss_G(\OO', \gamma)}{\Gauss_{\CCp G 0(\gamma^2)/\CCp G 0(\gamma)}(\OO', \gamma)}\times{} \\
	&\qquad c_{\OO'}(\pi^{\prime\,g_0}, \gamma)\Gauss_H(\OO', \gamma)\muhat^{H\conn}_{\OO'}(\chrc{\sbtl\fh x r, \AddChar_{X^*}^\vee}),
\end{align*}
equals
\[
\int_\fh \abs{\Disc_{G/H}(\gamma)}^{1/2}\Theta_\pi(\gamma\dotm\mexp(Y))\chrc{\sbtl\fh x r, \AddChar_{X^*}^\vee}(Y)\upd Y,
\]
which, by Lemma \ref{lem:sample}, equals
\[
\frac{\meas(\sbtl H x r)}{\meas(\sbtl\fh x r)}\abs{\Disc_{G/H}(\gamma)}^{1/2}\tr \pi(\gamma\chrc{\sbtl G x r, \hat\phi^\vee}),
\]
where \(\hat\phi\) is the character of \(\sbtl G x r\) with dual blob \(X^* + \sbtlpp{\Lie^*(G)}x{-r}\).

We assume for notational convenience that \(g_0\) is the identity.  In particular, by Lemma \ref{lem:in-H'} and Hypothesis \ref{hyp:gamma-central}, if there is anything to test, which is to say if \(\Lie^*(H) \cap (Z^*_o + \sbjtlpp{\Lie^*(G')}{-r})\) is non-empty, then \(\gamma\) belongs to \(G'\) and \(Z^*_o\) belongs to \(\Lie^*(H')\).

If \(\gamma\) is not compact modulo \(\Zent(G)\), then, because \(\Zent(G')/\Zent(G)\) is compact \cite{yu:supercuspidal}*{\S3, p.~590, \textbf{D1}}, it is also not compact modulo \(\Zent(G')\), so the matching condition becomes the equality \(0 = 0\) by \cite{deligne:support}*{p.~156, Th\'eor\`eme 2} and \cite{casselman:jacquet}*{Theorem 5.2}.

If \(\gamma\) \emph{is} compact modulo \(\Zent(G)\), then, since it stabilises the image of \ox in the reduced building of \bH, we must have that \(\gamma\) fixes \ox.  Since \(\gamma\) stabilises \(X^* \in \Lie^*(H)\), hence \(X^* + \sbtlpp{\Lie^*(M')}x{-r} = X^* + \sbtlpp{\Lie^*(G')}x{-r}\) (by Lemma \initref{lem:commute-gp}\subpref{down}, say) we have by Hypothesis \initref{hyp:MP-ad}\subpref{Ad} that it also stabilises \((\sbtl{G'}x r, \phi)\), where \(\phi\) is the character of \(\sbtl{G'}x r\) with dual blob \(X^* + \sbtlpp{\Lie^*(G')}x{-r}\), hence belongs to \(K'\).  Thus, the matching condition for \(T = \Theta_\pi\) and \(T' = \Theta_{\pi'}\) in Lemma \ref{lem:dist-G-to-G'} is satisfied.

By Remark \ref{rem:which-char}, (the other direction of) Lemma \ref{lem:asymptotic-check}, and Lemma \ref{lem:sample}, we have the desired equality ``with primes''; that is to say, we \emph{know} that
\begin{align*}
\abs{\redD_{H'}(Z^*_o)}^{1/2}\sum_{\OO' \in \OO^{H\primeconn}(\Ad^*(G')X^*_o)}
	{}
	&\frac{\Gauss_{G'}(\OO', \gamma)}{\Gauss_{\CCp{G'}0(\gamma^2)/\CCp{G'}0(\gamma)}(\OO', \gamma)}\times{} \\
	&\qquad c_{\OO'}(\pi', \gamma)\Gauss_{H'}(\OO', \gamma)\muhat^{H\primeconn}_{\OO'}(\chrc{\sbtl{\fh'}x r, \AddChar_{X^*}^\vee})
\end{align*}
equals
\[
\frac{\meas(\sbtl{H'}x r)}{\meas(\sbtl{\fh'}x r)}\abs{\Disc_{G'/H'}(\gamma)}^{1/2}\tr \pi'(\gamma\chrc{\sbtl{G'}x r, \phi^\vee}),
\]
where \(\phi\) is the character of \(\sbtl{G'}x r\) with dual blob \(X^* + \sbtlpp{\Lie^*(G')}x{-r}\).

Now, according to Waldspurger's canonical Haar measures, we have that \(\sbtlp H x 0\) and \(\sbtlp{\Lie(H)}x 0\) both have the same measure; and, by Lemma \ref{lem:MP-card}, the indices \(\indx{\sbtlp H x 0}{\sbtl{(H', H)}x{(\Rp0, r)}}\) and \(\indx{\sbtlp{\Lie(H)}x 0}{\sbtl{\Lie(H', H)}x{(\Rp0, r)}}\) are the same; so \(\dfrac{\meas(\sbtl H x r)}{\meas(\sbtl{\Lie(H)}x r)}\) equals \(\dfrac{\meas(\sbtl{H'}x r)}{\meas(\sbtl{\Lie(H')}x r)}\).  By Hypothesis \initref{hyp:Z*}\subpref{good}, we have that \(\dfrac{\abs{\redD_H(Z^*_o)}}{\abs{\redD_{H'}(Z^*_o)}}\) equals \(q_{H/H'}^r\).  The desired equality thus follows from Lemma \ref{lem:dist-G-to-G'}.
\end{proof}

\printindex

\end{document}